%% file: bk-bernoulli-arxiv_v2.tex
\documentclass[bj]{imsart2}

\RequirePackage{amsthm,amsmath,amsfonts,amssymb}
\RequirePackage[numbers]{natbib}

\input{bk-matern_commands.tex}

\startlocaldefs

\endlocaldefs

\begin{document}

\begin{frontmatter}
	\title{Equivalence of measures and 
	asymptotically  optimal 
	linear prediction 
	for Gaussian 
	random fields  
	with fractional-order covariance operators}

\runtitle{Gaussian fields with fractional covariance operators}

\begin{aug}
	\author[A]{\fnms{David} \snm{Bolin}\ead[label=e1]{david.bolin@kaust.edu.sa}}
\and
\author[B]{\fnms{Kristin} \snm{Kirchner}\ead[label=e2,mark]{k.kirchner@tudelft.nl}}

\address[A]{CEMSE Division, King Abdullah University of Science and Technology, Thuwal, Saudi Arabia \printead{e1}}
\address[B]{Delft Institute of Applied Mathematics, Delft University of Technology, Delft, The Netherlands \printead{e2}}

\end{aug}

	\begin{abstract}
	 \, We consider Gaussian measures   
	$\mu, \widetilde{\mu}$ 
	on a separable Hilbert space, 
	with fractional-order covariance operators 
	$A^{-2\beta}$ resp.\   
	$\widetilde{A}^{-2\widetilde{\beta}}$\!,  
	and derive necessary and sufficient conditions 
	on $A, \widetilde{A}$ and $\beta, \widetilde{\beta} > 0$ 
	for I.~equivalence of the measures 
	$\mu$ and~$\widetilde{\mu}$, 
	and  
	II.~uniform asymptotic optimality of linear predictions 
	for $\mu$ based on the misspecified 
	measure $\widetilde{\mu}$. 
	These results  
	hold, e.g., for Gaussian processes 
	on compact metric spaces. 
	As an important special case, we consider the class 
	of generalized Whittle--Mat\'ern Gaussian random fields, 
	where $A$ and $\widetilde{A}$ are elliptic 
	second-order differential operators, 
	formulated on a bounded Euclidean 
	domain $\mathcal{D}\subset\mathbb{R}^d$ 
	and augmented with homogeneous 
	Dirichlet boundary conditions. 
	Our outcomes explain 
	why the predictive performances   
	of stationary and non-stationary models 
	in spatial statistics often are 
	comparable, 
	and 
	provide a crucial first step in deriving consistency results 
	for parameter estimation of generalized Whittle--Mat\'ern fields. 
\end{abstract}

\begin{keyword} 
	\kwd{Gaussian measures} 
	\kwd{kriging} 
	\kwd{elliptic differential operators} 
	\kwd{Whittle--Mat\'ern fields.}
\end{keyword}


\end{frontmatter}


\section{Introduction and preliminaries}\label{section:intro}

\subsection{Introduction}\label{subsec:intro}

Equivalence and orthogonality of Gaussian measures 
are essential concepts for investigating asymptotic properties 
of Gaussian random fields and stochastic processes. 
For example, they play a crucial role 
when proving 
consistency 
of maximum likelihood estimators 
for covariance parameters 
under infill asymptotics 
\citep{Anderes2010, Zhang2004}, 
or 
asymptotic optimality of 
linear predictions for random fields 
based on misspecified covariance models \citep{stein99}. 
However, for the latter equivalence 
of the Gaussian measures is not a necessary assumption. 
This is an immediate consequence of the 
\emph{necessary and sufficient} conditions 
for uniformly asymptotically optimal linear prediction,  
derived recently in \citep{bk-kriging} 
for Gaussian random fields on compact metric spaces. 
Both the necessary and sufficient 
conditions 
for equivalence of Gaussian measures as given by the 
Feldman--H\'ajek theorem \citep[][Theorem~2.25]{daPrato2014} and those 
for uniformly asymptotically optimal linear prediction 
\citep[][Assumption~3.3]{bk-kriging} 
are formulated in terms of the covariance operators 
rather than the covariance functions. 
Therefore, they may be difficult to verify. 
To the best\linebreak 
of our knowledge 
there are not even 
\emph{sufficient} conditions 
for asymptotic optimality available, 
which are easily verifiable,  
except for a few special cases such as 
stationary random fields on~$\bbR^d$\!, 
see e.g.~\citep{stein93,stein97}. 

In the present work we remedy this lack. 
By means of complex interpolation theory 
combined with operator theory for fractional powers 
of closed operators, 
we are able to characterize  
the necessary and sufficient conditions 
mentioned above 
for a wide class of stationary 
and non-stationary 
Gaussian processes. Specifically, 
in the first part we consider Gaussian measures 
on a generic separable Hilbert space 
with fractional-order covariance 
operators 
and translate the conditions of 
the Feldman--H\'ajek theorem 
(see Theorem~\ref{thm:feldman-hajek} 
in Appendix~\ref{appendix:feldman-hajek}) 
and of   
\citep[][Assumption~3.3]{bk-kriging} 
on the covariance operators $A^{-2\beta}$\!, $\At^{-2\betat}$   
to conditions on the non-fractional base operators $A, \At$. 

Our results  
are applicable to Gaussian random fields  
on compact metric spaces.  
As an important special case, we detail our outcomes for  
the class of generalized Whittle--Mat\'ern 
Gaussian random fields \citep{BK2020rational,cox2020,lindgren11}.
The Mat\'ern covariance family 
\cite{matern60} is highly popular 
in spatial statistics and machine learning,  
see e.g.~\citep{Guttorp2006,Rasmussen2006}. 
Given the parameters 
$\nu,\sigma^2\!,\kappa > 0$, 
which determine the smoothness, 
the variance, and the 
practical correlation range, respectively, 
the corresponding Mat\'ern covariance function 
is  
\begin{equation}\label{eq:matern_cov}
	\varrho(\s, \s') 
	= 
	\varrho_0\bigl( \norm{ \s - \s' }{\bbR^d} \bigr) , 
	\quad 
	\s, \s' \in \bbR^d, 
	\qquad 
	\text{where} 
	\qquad  
	\varrho_0(h)
	:= 
	\tfrac{\sigma^2}{2^{\nu-1}\Gamma(\nu)} \, (\kappa h)^{\nu} K_{\nu}(\kappa h). 
\end{equation}
Whittle \cite{whittle63} showed that the stationary solution 
$\GP\from\bbR^d\times\Omega\to\bbR$ to the 
stochastic partial differential equation (SPDE for short) 
\begin{equation}\label{eq:statmodel}
	\bigl( - \Delta + \kappa^2 \bigr)^{\beta} \GP = \white 
	\quad 
	\text{in} \;\; \bbR^d,
\end{equation} 
has covariance 
\eqref{eq:matern_cov} with range parameter $\kappa$, 
smoothness $\nu = 2\beta - \nicefrac{d}{2}$, and variance 
\begin{equation}\label{eq:spde_var}
	\sigma^2 
	= 
	(4\pi)^{-\nicefrac{d}{2}} \kappa^{-2\nu} 
	\bigl( \Gamma(\nu) / \Gamma( \nu + \nicefrac{d}{2} ) \bigr).  
\end{equation}
In the SPDE \eqref{eq:statmodel}  
$\Delta$ is the Laplacian 
(see Appendix~\ref{app:subsubsec:Ck}),   
$\beta > \nicefrac{d}{4}$, 
and $\white$ is 
Gaussian white noise. 

Motivated by this SPDE representation, 
Lindgren, Rue and Lindstr\"om \cite{lindgren11} 
suggested extensions 
of the Mat\'ern model to non-stationary models 
and to more general spatial domains. 
This has initiated an active research area, 
where spatial models based on SPDEs 
are proposed  
and investigated, see e.g.\ 
\citep{bakka2019,bolin11,fuglstad2015non-stationary,Hildeman2020}. 
Most of the extensions that have been considered  
are special cases of 
\emph{generalized Whittle--Mat\'ern} 
Gaussian random fields, 
which are defined through fractional-order SPDEs 
of the form 
\begin{equation}\label{eq:whittle-matern}
	\bigl( - \nabla \cdot(\ac \nabla) + \kappa^2 \bigr)^\beta \GP = \white 
	\quad 
	\text{in}\;\;\cD,
\end{equation}
where $\cD\subset\bbR^d$ is a bounded Euclidean 
domain with boundary $\partial\cD$,
$\kappa\from\cD\to\bbR$ is a bounded real-valued function,  
$\ac\from\cD\to\bbR^{d\times d}$ is a (sufficiently nice) positive matrix-valued function, 
and $\beta \in (\nicefrac{d}{4},\infty)$. 
The fractional power $L^\beta$ of the differential operator 
$L=- \nabla \cdot(\ac \nabla) + \kappa^2$ 
is understood in the spectral sense, where 
first $L$ is augmented with appropriate boundary conditions on $\partial\cD$
(usually, homogeneous Dirichlet or Neumann 
conditions, 
see Appendix~\ref{app:subsec:pdes}). 
For this class of models, 
$\kappa$ determines the \emph{local} correlation ranges, 
whereas $\ac$ describes 
\emph{local anisotropies}, 
see e.g.\ \citep{fuglstad2015non-stationary}. 
Whenever $\kappa$ is constant and $\ac$ is the identity matrix, 
the model \eqref{eq:whittle-matern} reduces to the 
\emph{classical} Whittle--Mat\'ern model \eqref{eq:statmodel} 
on~$\cD\subset\bbR^d$\!. 

Some properties of  
generalized Whittle--Mat\'ern fields 
have already been discussed in the literature 
\citep{BKK2018, BKK2020, cox2020, Herrmann2020}, 
but there are still 
considerably more results available for   
the original Gaussian Mat\'ern class on~$\bbR^d$\!. 
In particular, Zhang \citep{Zhang2004} and Anderes \citep{Anderes2010}   
investigated parameter estimation for Gaussian Mat\'ern fields on $\bbR^d$ 
under infill asymptotics. 
Thereby they showed that 
two Gaussian measures 
$\mu_d(0;\nu,\sigma^2\!,\kappa)$ 
and $\mu_d(0;\nu,\sigmat^2\!,\kappat)$,  
corresponding 
to zero-mean Gaussian Mat\'ern fields on $\bbR^d$ 
with parameters $\nu,\sigma^2\!,\kappa>0$ 
and $\nu,\sigmat^2\!,\kappat>0$, respectively, 
are equivalent if and only if  
\[
\begin{cases}
	\sigma^2 \kappa^{2\nu} = \sigmat^2 \kappat^{2\nu} 
	& \text{for $d \leq 3$,}\\
	\kappa = \kappat 
	\;\;\text{and}\;\; \sigma^2 = \sigmat^2  
	& \text{for $d \geq 5$}.
\end{cases}
\]
Until now, the case $d=4$ has remained open.  
Furthermore, by \cite[Theorem~12 in Chapter~4]{stein99}
$\mu_d(0;\widetilde{\nu},\sigmat^2\!,\kappat)$ 
provides uniformly asymptotically optimal linear prediction 
for $\mu_d(0;\nu,\sigma^2\!,\kappa)$ 
in any dimension $d\in \bbN$ if $\nu = \widetilde{\nu}$. 
Neither equivalence of measures 
nor asymptotic optimality of linear prediction 
have been characterized for 
generalized Whittle--Mat\'ern fields yet.\pagebreak  

Based on our general 
results for Gaussian measures  
with fractional-order covariance operators  
on a separable Hilbert space, 
combined with regularity theory 
for elliptic second-order partial differential equations, 
we are able to fill this gap in the second part of this work.  
Assuming that the coefficients $\ac,\act,\kappa,\kappat$ are smooth   
and that $\cD\subset\bbR^d$ has a smooth boundary~$\partial\cD$, 
we consider  
two Gaussian measures 
$\mu_d(0;\beta,\ac,\kappa)$ and 
$\mu_d(0;\betat, \act, \kappat)$ 
corresponding 
to  generalized Whittle--Mat\'ern fields 
\eqref{eq:whittle-matern} 
(using homogeneous Dirichlet 
boundary conditions on $\partial\cD$)
with parameters $(\beta,\ac,\kappa)$ 
and $(\betat, \act, \kappat)$, 
respectively. 
For this setting, we prove the following: 
\begin{enumerate}[label={\normalfont\Roman*.}]
	\item\label{enum:result1} 
	In dimension $d\leq 3$, 
	$\mu_d(0;\beta,\ac,\kappa)$ and 
	$\mu_d(0;\betat,\act,\kappat)$ are equivalent 
	if and only if $\beta = \betat$,  
	${\ac=\act}$ in~$\clos{\cD}$, 
	and 
	$\kappat^2-\kappa^2$ 
	satisfies certain boundary conditions 
	on $\partial\cD$. 
	In contrast, for $d\geq 4$,  
	the measures are equivalent if and only if 
	$\beta = \betat$ and 
	$\ac = \act$, 
	$\kappa^2 = \kappat^2$ in $\clos{\cD}$. 
	\item\label{enum:result2} 
	In any dimension $d \in\bbN$,  
	the model  
	$\mu_d(0;\betat, \act, \kappat)$ 
	provides uniformly asymptotically 
	optimal linear prediction for 
	the model $\mu_d(0;\beta,\ac,\kappa)$ 
	if and only if 
	$\beta = \betat$,  
	$c\ac = \act$ in~$\clos{\cD}$ holds 
	for some  $c\in(0,\infty)$,  
	and
	$\kappat^2 - c\kappa^2$ fulfills certain boundary conditions 
	on~$\partial\cD$. 
\end{enumerate}
These results 
cover the parameter range 
$\beta\in(\nicefrac{d}{4},\infty)\setminus\{k+\nicefrac{1}{4} : k \in \bbN\}$ and, in particular, 
also the case $d=4$ 
for the  
classical Mat\'ern model 
(when considered on a bounded domain). 
Moreover, 
to the best of our knowledge 
these are the first explicit results 
on equivalence of measures and 
asymptotic efficiency of  
linear predictions 
for this general class of models. 
Outcome~\ref{enum:result1} readily implies that, 
for ${d\leq 3}$, 
one cannot estimate all parameters 
of a generalized Whittle--Mat\'ern field 
consistently, 
and it provides a crucial first step 
towards showing consistency of 
maximum likelihood estimates 
for the parameters~$\beta, \ac$. 
Result~\ref{enum:result2} 
explains 
the comparable predictive performance of non-stationary  
and stationary models 
that has been 
noted for example in~\citep{fuglstad2015non-stationary}.  

The outline is as follows: 
In the next subsection  
we introduce 
preliminaries 
and our notation.  
Section~\ref{section:gm-on-hs} 
is concerned with the general case of 
Gaussian measures on Hilbert spaces 
with fractional-order covariance operators. 
These outcomes are applied, 
in Section~\ref{section:examples},  
to some first examples 
including the classical Whittle--Mat\'ern model 
on a bounded domain $\cD\subset\bbR^d$ 
and, in Section~\ref{section:wm-D}, 
to derive 
the results~\ref{enum:result1}$, $\ref{enum:result2} 
for generalized Whittle--Mat\'ern fields. 
In Section~\ref{section:numexp} the 
result~\ref{enum:result2} 
is verified 
in two simulation studies 
for non-stationary random fields on 
the unit interval $\cD = (0,1)$. 
The manuscript concludes with a discussion 
in Section~\ref{section:discussion} 
and contains five Appendices 
\ref{appendix:function-spaces}, 
\ref{appendix:feldman-hajek}, \ref{appendix:A-alpha}, 
\ref{appendix:proofs-lemmas}, \ref{appendix:interpol}
with auxiliary results and proofs.

\subsection{Preliminaries and notation}\label{subsec:notation} 

If not specified otherwise,
$\scalar{\,\cdot\,,\,\cdot\,}{E}$
denotes the inner product
on a Hilbert space~$E$, 
$\norm{\,\cdot\,}{E}$ the induced norm, 
$\id_E$ the identity on $E$, 
and 
$\cB(E)$ the Borel $\sigma$-algebra on $(E,\norm{\,\cdot\,}{E})$, 
that is the smallest $\sigma$-algebra 
containing all open sets. 
The scalar field $\bbK$ is either given 
by the real numbers $\bbR$ 
or the complex numbers~$\bbC$. 
The dual $\dual{E}$ of $E$ 
is the space 
containing all continuous linear functionals 
$f\from E\to\bbK$, 
and we call 
$\duality{\,\cdot\,, \,\cdot\,}{} 
\from \dual{E}\times E\to\bbK$, 
$\duality{f,\psi}{} := f(\psi)$ 
the duality pairing between 
$\dual{E}$ and $E$. 

The space of all bounded linear operators 
from $(E,\scalar{\,\cdot\,,\,\cdot\,}{E})$ 
to a second Hilbert space 
$(F,\scalar{\,\cdot\,,\,\cdot\,}{F})$ 
is denoted by $\cL(E;F)$. 
It is rendered 
a Banach space when equipped 
with the usual operator norm 
$\|T\|_{\cL(E;F)} := \sup_{\psi\in E\setminus\{0\}} \frac{ \norm{T\psi}{F} }{\norm{\psi}{E}}$.  
We call a linear operator $T\from E \to F$ 
an isomorphism 
if $T\in\cL(E;F)$ and $T^{-1}\in\cL(F;E)$, i.e., 
$T$ is bounded and has a bounded inverse. 
If $V$ is a vector space such that $E,F \subseteq V$ 
and if, in addition, $\id_V|_{E} \in \cL(E;F)$, 
then $E$ is continuously embedded in $F$  and 
we write 
$(E,\norm{\,\cdot\,}{E} ) 
\hookrightarrow 
(F, \norm{\,\cdot\,}{F})$. 
The notation 
$(E, \norm{\,\cdot\,}{E} ) \cong (F, \norm{\,\cdot\,}{F})$ 
indicates that  
$(E, \norm{\,\cdot\,}{E} )$ and $(F, \norm{\,\cdot\,}{F})$ 
are isomorphic, i.e., 
$(E, \norm{\,\cdot\,}{E} ) \hookrightarrow 
(F, \norm{\,\cdot\,}{F}) 
\hookrightarrow (E, \norm{\,\cdot\,}{E})$. 
Whenever $E=F$, we 
abbreviate 
$\cL(E) := \cL(E;E)$, 
and this convention holds also 
for all other spaces of operators to be introduced. 
The subspaces $\cK(E;F)\subseteq\cL(E;F)$ 
and 
$\cL_2(E;F)\subseteq\cL(E;F)$  
contain all compact operators and  
Hilbert--Schmidt operators, respectively. 
Note that $T\in\cL(E;F)$ is compact if and only if 
it is the limit in $\cL(E;F)$ of finite-rank operators, 
and $\cL_2(E;F)$ is a Hilbert space 
with 
the inner product 
$(T,S)_{\cL_2(E;F)} := \sum_{j \in \bbN}(Te_j, Se_j)_F$, 
where
$\{e_j\}_{j\in\bbN}$ is any orthonormal basis for~$E$.  
The adjoint 
of $T\in\cL(E;F)$ 
is identified with 
$T^*\in\cL(F;E)$ 
(via the Riesz maps on~$E$ and on~$F$). 
An operator $T\in\cL(E)$ 
is said to be 
orthogonal 
if $T T^* = T^* T = \id_E$, 
self-adjoint 
if $T = T^*$\!, 
nonnegative definite  
if $\scalar{T\psi, \psi}{E}\geq 0$ holds for all $\psi\in E$, 
and positive definite if 
there exists a constant  $\theta\in(0,\infty)$ such that 
$\scalar{T\psi,\psi}{E} \geq \theta \|\psi\|_E^2$ 
for all $\psi\in E$.  
A self-adjoint, nonnegative 
definite operator $T\in\cL(E)$ has a finite 
trace if $\sum_{j\in\bbN} (Te_j, e_j)_E < \infty$ 
holds for an 
(or, equivalently, any)
orthonormal basis $\{e_j\}_{j\in\bbN}$ 
of~$E$. 

A (possibly unbounded) 
linear operator $A$ on $E$ with domain 
$\scrD(A) = \{\psi\in E : \norm{A\psi}{E}<\infty \} \subseteq E$ 
is denoted by $A\from \scrD(A) \subseteq E \to E$. 
It is closed if its graph 
$\mathscr{G}(A):={\{(x,Ax): x \in \scrD(A)\}}$ 
is closed with respect to the norm 
$\norm{(x,Ax)}{\mathscr{G}(A)} := \norm{x}{E} + \norm{Ax}{E}$  
and densely defined if $\scrD(A)$ is dense in~$E$.  

Throughout this article, 
 $(\Omega,\mathscr{F},\bbP)$ 
is a complete probability space. 
For a Hilbert space 
$(E, \scalar{\,\cdot\,, \,\cdot\,}{E})$  
and $p\in[1,\infty)$, 
$L_p(\Omega;E)$ denotes the space 
of (equivalence classes of) $E$-valued, 
Bochner measurable random variables 
with finite $p$-th moment, 
with norm  
$\norm{\GP}{L_p(\Omega;E)}^p 
:=  
\int_\Omega \norm{\GP(\omega)}{E}^p  \, \rd \bbP(\omega)$. 
Further, $(\cX,d_\cX)$ is a 
connected, compact metric
space of infinite cardinality, equipped 
with a strictly positive and finite Borel measure $\nu_\cX$, 
and 
$L_2(\cX,\nu_\cX)$  
is the Hilbert space of 
(equivalence classes of) 
real-valued, Borel measurable,  
square-integrable   
functions on $\cX$,  
with  
$\norm{f}{L_2(\cX,\nu_\cX)}^2 
:= 
\int_{\cX} \seminorm{f(x)}{}^2 \, \rd \nu_\cX(x)$. 

Finally, we write 
$\bbR_{+} := (0,\infty)$ 
for the positive part of the real axis, 
$\bbN$ (or $\bbN_0$) for 
the set of positive (respectively, nonnegative) 
integers, and 
$\lfloor \,\cdot\, \rfloor$ 
(or $\lceil \,\cdot\, \rceil$)  
for the floor (respectively, ceiling) function. 

\section{Gaussian processes with fractional-order covariance operators}
\label{section:gm-on-hs} 

Throughout this section we let 
$(E, \scalar{\,\cdot\,, \,\cdot\,}{E})$ be a 
separable Hilbert space over $\bbR$ 
with $\dim(E)=\infty$. 
Furthermore, we assume that $\mu$ 
is a Gaussian measure on~$E$, 
i.e., for every $\psi\in E$, 
\[	
	\exists 
	m_\psi\in\bbR, \; \sigma_\psi^2 \in \bbR_+ 
	: 
	\quad 
	\forall 
	B \in \cB(\bbR) 
	\quad\,  
	\mu(\{ \phi\in E : \scalar{\psi, \phi}{E} \in B\})
	= 
	\bbP(\{ \omega\in\Omega: \gp_\psi(\omega) \in B\}), 
\]
where $\gp_\psi \from \Omega \to \bbR$ is a random variable,  
which is either  
Gaussian distributed with mean~$m_\psi$ and variance $\sigma_\psi^2$, or 
concentrated at $m_\psi$, i.e., 
$\bbP(\{\omega\in\Omega : \gp_\psi(\omega)=m_\psi\}) = 1$.  
Then, there exist 
a vector $m\in E$ and a 
bounded linear operator $\cC\from E \to E$ 
such that, for all $\psi,\psi' \in E$, 
\begin{equation}\label{eq:mu-mean-cov}   
	\textstyle 
	\scalar{m, \psi}{E}
	=
	\int_E \scalar{\phi, \psi}{E} \, \rd \mu(\phi), 
	\qquad  
	\scalar{\cC \psi, \psi'}{E} 
	= 
	\int_E \scalar{\psi, \phi-m}{E} \scalar{\phi-m, \psi'}{E} \,  \rd \mu(\phi). 
\end{equation}
The vector $m$ is the mean of the Gaussian measure $\mu$ 
and $\cC$ is its covariance operator. 
One can show that $\cC\from E \to E$ 
is self-adjoint, nonnegative definite, and has a finite trace 
\citep[][Theorem~2.3.1]{Bogachev1998}. 
Moreover,  
$\mu$ is uniquely determined by 
its mean $m$ and its covariance operator~$\cC$,   
and we therefore write 
$\mu = \normal(m,\cC)$. 

In this section we consider covariance operators of the 
form $\cC = A^{-2\beta}$\!, 
where~$A$ is an unbounded linear operator on $E$ 
and $\beta \in \bbR_{+}$. 
The  main objectives of this section are to characterize  
for two given Gaussian measures 
$\mu  := \normal\bigl( m, A^{-2\beta} \bigr)$  
and 
$\mut := \normal\bigl(\mt, \At^{-2\betat} \bigr)$ 
the following: 
\begin{enumerate}[label={\normalfont\Roman*.}] 
	\item 
	equivalence resp.\ orthogonality 
	of $\mu$ and $\mut$,  
	see Subsection~\ref{subsec:gm-on-hs:equivalence}, and 
	\item 
	uniform asymptotic optimality  
	of linear predictions for $\mu$ based on the misspecified 
	measure $\mut$ when $E = L_2(\cX, \nu_\cX)$, 
	see Subsection~\ref{subsec:gm-on-hs:kriging}. 
\end{enumerate}  
To this end, in Subsection~\ref{subsec:gm-on-hs:auxiliary} 
we first specify our assumptions on $A$ 
and state  
two auxiliary results. 

\subsection{Hilbert space setting and some auxiliary results}
\label{subsec:gm-on-hs:auxiliary} 

In what follows, we assume that 
$A\from\scrD(A) \subseteq E \to E$ 
is a densely defined, self-adjoint, positive definite 
linear operator, which has a compact 
inverse $A^{-1}\in\cK(E)$. 
In this case, 
$A\from\scrD(A) \subseteq E \to E$ 
is closed and there exists an 
orthonormal basis $\{e_j\}_{j\in\bbN}$ 
for $E$ consisting 
of eigenvectors of~$A$, 
with corresponding positive 
eigenvalues $(\lambda_j)_{j\in\bbN}$ 
accumulating only at~$\infty$. 
We assume that they 
are in non-decreasing order, 
$0 < \lambda_1 \leq \lambda_2 \leq \ldots$, 
and repeated 
according to multiplicity.  

For $\beta\in[0,\infty)$, the 
fractional power operator $A^\beta\from\scrD\bigl(A^\beta\bigr)\subseteq E \to E$ 
can then be defined using the spectral 
expansion 
\begin{equation}\label{eq:def:Abeta}  
	A^\beta \psi  
	:= 
	\sum\limits_{j\in\bbN} 
	\lambda_j^\beta 
	\scalar{ \psi, e_j }{E} \, e_j , 
	\qquad 
	\psi \in \scrD\bigl( A^\beta \bigr) \subseteq E. 
\end{equation}
Note that, for all $r\in[0,\infty)$, 
the domain of the operator $A^{\nicefrac{r}{2}}$\!, 
\begin{equation}\label{eq:def:hdotA}  
	\textstyle 
	\hdot{r}{A} 
	:= 
	\scrD\bigl( A^{\nicefrac{r}{2}} \bigr), 
	\quad 
	\scrD\bigl( A^{\nicefrac{r}{2}} \bigr)
	= 
	\Bigl\{ 
	\psi \in E : 
	\norm{A^{\nicefrac{r}{2}} \psi }{E}^2 
	=
	\sum\nolimits_{j\in\bbN} 
	\lambda_j^r \, 
	| \scalar{ \psi, e_j}{E} |^2
	< \infty 
	\Bigr\}, 
\end{equation} 
is itself a separable Hilbert space with respect to the inner product 
\[
	\textstyle 
	\scalar{ \phi, \psi}{r,A} 
	:= 
	\scalar{ A^{\nicefrac{r}{2}} \phi, A^{\nicefrac{r}{2}} \psi}{E} 
	= 
	\sum\nolimits_{j\in\bbN} 
	\lambda_j^r
	\, 
	\scalar{\phi, e_j}{E} 
	\scalar{e_j, \psi}{E}, 
\]
and the corresponding induced norm 
$\norm{\,\cdot\,}{r,A}$. 
Here, 
$A^0 := \id_E$ and $\hdot{0}{A} := E$. 
Recall that 
by definition  
the Cameron--Martin space 
of a Gaussian measure $\mu=\normal(m,\cC)$ on $E$ 
(aka.\ the reproducing kernel Hilbert space of $\cC$)
is the image of~$E$ under $\cC^{\nicefrac{1}{2}}$\!, 
endowed with the inner product 
$\scalar{\cC^{-\nicefrac{1}{2}}\,\cdot\,, 
	\cC^{-\nicefrac{1}{2}}\,\cdot\,}{E}$, 
cf.~\citep[][p.~44]{Bogachev1998}.   
It consists of all elements $v \in E$ 
such that the measure $\mu_v(B) := \mu(B-v)$ 
is absolutely continuous with respect to $\mu$ 
(i.e., $\mu(B)=0 \Rightarrow \mu_v(B) = 0$). 
In particular, for $\cC=A^{-2\beta}$  
we obtain that 
\begin{equation}\label{eq:general-CM-space} 
	\cC^{\nicefrac{1}{2}} (E) = A^{-\beta} (E) = \scrD\bigl( A^{\beta} \bigr) = \hdot{2\beta}{A}, 
	\qquad 
	\norm{\cC^{-\nicefrac{1}{2}}\,\cdot\,}{E}=\norm{\,\cdot\,}{2\beta,A}. 
\end{equation} 

We let $\hdot{-r}{A}$ denote the dual space 
of $\hdot{r}{A}$ after identification 
via the inner product  
on~$E$ which is continuously 
extended to a duality pairing. 
This means that 
for all $\phi\in E\subseteq \hdot{-r}{A}$,  
$\psi\in\hdot{r}{A}\subseteq E$, we have that 
$\langle \phi, \psi\rangle = \scalar{\phi, \psi}{E}$. 
It is an immediate consequence 
of these definitions that,  
for every $r,\vartheta\in\bbR$, 
the fractional power operator 
$A^\vartheta\from \hdot{r}{A} \to \hdot{r-2\vartheta}{A}$ 
is an isomorphism, 
possibly obtained as a continuous extension 
or restriction of 
$A^\vartheta \from \scrD\bigl(A^\vartheta\bigr) = \hdot{2\vartheta}{A} \to \hdot{0}{A} = E$. 
For ease of presentation, we postpone 
the technical proofs of the following 
two auxiliary results, Lemmas~\ref{lem:beta-gamma} 
and~\ref{lem:iff-A}, 
to Appendix~\ref{appendix:proofs-lemmas}. 

\begin{lemma}\label{lem:beta-gamma}  
	Let $A\from\scrD(A) \subseteq E \to E$ and 
	$\At\from\scrD(\At)\subseteq E \to E$ 
	be two densely defined, self-adjoint, 
	positive definite linear operators 
	with compact inverses 
	on $E$. 
	\begin{enumerate}[label={\normalfont(\roman*)}] 
		\item\label{lem:beta-gamma-HS}  
		Assume that there exists $\beta\in\bbR_+$ 
		such that $\At^\beta A^{-\beta}\from E \to E$ 
		is an isomorphism and additionally 
		$A^{-\beta} \At^{2\beta} A^{-\beta} 
		- \id_E \in\cL_2(E)$. 
		Then, for every $\gamma\in[-\beta,\beta]$, 
		also the operator $\At^\gamma A^{-\gamma}$ 
		is an isomorphism on $E$ 
		and 
		$A^{-\gamma} \At^{2\gamma} A^{-\gamma} 
		- \id_E \in \cL_2(E)$. 
		\item\label{lem:beta-gamma-compact} 
		Assume that there exists $\beta\in\bbR_+$ 
		such that $\At^\beta A^{-\beta}\from E \to E$ 
		is an isomorphism and additionally 
		$A^{-\beta} \At^{2\beta} A^{-\beta} 
		- \id_E \in\cK(E)$. 
		Then, for every $\gamma\in[-\beta,\beta]$, 
		also the operator $\At^\gamma A^{-\gamma}$ 
		is an isomorphism on $E$ 
		and 
		$A^{-\gamma} \At^{2\gamma} A^{-\gamma} 
		- \id_E \in \cK(E)$. 
	\end{enumerate}
\end{lemma}

\begin{lemma}\label{lem:iff-A}  
	Let 
	$A\from\scrD(A) \subseteq E \to E$ and 
	$\At\from\scrD(\At)\subseteq E \to E$ 
	be two densely defined, self-adjoint, 
	positive definite linear operators 
	with compact inverses 
	on $E$, let  
	$\beta\in[1,\infty)$, and define 
	\begin{equation}\label{eq:frakNset}
		\frakN_\beta 
		:= 
		\{n\in\bbN : n \leq \beta\} \cup\{\beta\} 
		=
		\{1,\ldots,\lfloor\beta\rfloor\}\cup\{\beta\}. 
	\end{equation} 
	\begin{enumerate}[label={\normalfont(\roman*)}] 
		\item\label{lem:iff-A-HS}  
		$\At^\gamma A^{-\gamma}$ is an 
		isomorphism on $E$ and 
		$A^{-\gamma} \At^{2\gamma} A^{-\gamma} 
		- \id_E\in\cL_2(E)$ 
		for all $\gamma \in [-\beta,\beta]$ 
		if and only if 
		for all $\eta\in\frakN_\beta$ 
		there exist an orthogonal operator~$U_\eta$ on~$E$ 
		and $S_\eta\in\cL_2(E)$ such that 
		$\id_E+ S_\eta$ is invertible on $E$ and 
		$A^{\eta-1} \At A^{-\eta} = U_\eta (\id_E+ S_\eta)$. \pagebreak
		\item\label{lem:iff-A-compact} 
		$\At^\gamma A^{-\gamma}$ is an 
		isomorphism on $E$ and 
		$A^{-\gamma} \At^{2\gamma} A^{-\gamma} 
		- \id_E\in\cK(E)$ 
		for all $\gamma \in [-\beta,\beta]$ 
		if and only if 
		for every $\eta\in\frakN_\beta$ 
		there exist 
		an orthogonal operator $W_\eta$ on $E$  
		and $K_\eta\in\cK(E)$ such that 
		$\id_E+ K_\eta$ is invertible on~$E$ and 
		$A^{\eta-1} \At A^{-\eta} = W_\eta (\id_E+ K_\eta)$.   
		\item\label{lem:iff-A-iso}  
		The linear operator  
		$\At^\gamma A^{-\gamma}\from E \to E$ is an isomorphism   
		for all $\gamma\in[-\beta,\beta]$ if and only if 
		${\At - A\in\cL\bigl( \hdot{2\eta}{A};    \hdot{2(\eta-1)}{A}   \bigr)
		\cap \cL\bigl( \hdot{2\eta}{\At}; \hdot{2(\eta-1)}{\At} \bigr)}$ 
		holds for every $\eta\in\{1,\beta\}$.  
	\end{enumerate} 
\end{lemma}

\subsection{Equivalence and orthogonality}
\label{subsec:gm-on-hs:equivalence} 

Two probability measures $\mu,\mut$ on $(E,\cB(E))$
are said to be equivalent 
if, for all Borel sets $B\in\cB(E)$, 
 $\mu(B)=0$ holds if and only if $\mut(B)=0$. 
In contrast, if there exists a Borel set 
$B\in\cB(E)$ such that $\mu(B)=0$ and $\mut(B)=1$, 
then $\mu$ and $\mut$ are said to be orthogonal. 
Two Gaussian measures $\mu, \mut$ 
are either equivalent or orthogonal 
\citep[][Theorem~2.7.2]{Bogachev1998}.  
As mentioned in the introduction, 
equivalence and orthogonality~of Gaussian measures 
are important concepts in statistical theory.  
For example, a crucial first step in proving that a parameter $\theta$ 
of a Gaussian process (with corresponding Gaussian measure~$\mu$) 
can be estimated consistently under infill asymptotics is often 
to define  $\mut$ as the Gaussian measure corresponding to 
the process with parameter $\widetilde{\theta} \neq \theta$ 
and to show that $\mu$ and $\mut$ are orthogonal, 
see~\citep{Zhang2004}. 

The following proposition provides necessary and sufficient conditions 
for equivalence of two Gaussian measures $\mu$ and $\mut$ 
when they have fractional-order covariance operators. 

\begin{proposition}\label{prop:A-equiv} 
	Let $A\from\scrD(A) \subseteq E \to E$ and 
	$\At\from\scrD(\At)\subseteq E \to E$ 
	be two densely defined, self-adjoint, 
	positive definite linear operators 
	with compact inverses on $E$. 
	In addition, 
	let $\beta\in[1,\infty)$,  
	$\betat \in \bbR_+$ be such that 
	$A^{-2\beta}$ and $\At^{-2\betat}$ 
	have finite traces on~$E$, 
	let $m,\mt \in E$,  
	and define $\delta:=\nicefrac{\betat}{\beta}\in\bbR_+$. 
	The Gaussian measures 
	$\mu = \normal\bigl( m, A^{-2\beta} \bigr)$ and 
	$\mut = \normal\bigl( \mt, \At^{-2\betat} \bigr)$ 
	are either equivalent or orthogonal. 
	They are equivalent if and only if 
	the following two conditions 
	are satisfied: 
	\begin{enumerate}[label={\normalfont(\alph*)}]
		\item\label{prop:A-equiv-iff-a} 
		the difference of the means 
		satisfies $m - \mt \in \hdot{2\beta}{A}$\!; 
		\item\label{prop:A-equiv-iff-b} 
		for all $\eta\in\frakN_\beta$, 
		where $\frakN_\beta$ is defined as in \eqref{eq:frakNset},  
		there exist an orthogonal operator 
		$U_\eta\in\cL(E)$ and   
		$S_\eta\in\cL_2(E)$  
		such that  
		$A^{\eta-1} \At^\delta  A^{-\eta} 
		= U_\eta (\id_E + S_\eta)$ 
		and $\id_E+S_\eta$ is invertible on~$E$. 
	\end{enumerate} 
	
	Condition~\ref{prop:A-equiv-iff-b} is in particular satisfied, 
	whenever 
	\[
		\At^{ \delta } - A \in 
		\cL_2\bigl( \hdot{2\eta}{A}; \hdot{2(\eta-1)}{A} \bigr) 
		\quad 
		\forall \eta\in\frakN_\beta, 
		\quad 
		\text{and} 
		\quad 
		\At^{ \delta } - A \in 
		\cL\bigl( \hdot{2\delta \eta}{\At}; \hdot{2\delta(\eta-1)}{\At} \bigr) 
		\quad 
		\forall\eta\in\{1,\beta\}. 
	\] 
\end{proposition}

\begin{proof}
	In order to derive the equivalence statement, 
	we apply the Feldman--H\'ajek theorem, see 
	Theorem~\ref{thm:feldman-hajek} 
	in Appendix~\ref{appendix:feldman-hajek}: 
	$\mu$ and $\mut$ are equivalent if and only if 
	\begin{enumerate}[label={\normalfont(\roman*)}]
		\item\label{proof:prop:A-equiv-iff-i} 
		the Cameron--Martin spaces 
		$\hdot{2\beta}{A}$  
		and 
		$\hdot{2\betat}{\At}$  
		(see \eqref{eq:general-CM-space})  
		are norm equivalent Hilbert spaces;  
		\item\label{proof:prop:A-equiv-iff-ii}  
		the difference of the means satisfies 
		$m - \mt \in \hdot{2\beta}{A}$\!; and 
		\item\label{proof:prop:A-equiv-iff-iii}  
		the operator 
		$A^{-\beta} \At^{2\betat} A^{-\beta} - \id_E$ 
		is a Hilbert--Schmidt operator on $E$. 
	\end{enumerate}
	
	It remains to prove that conditions 
	\ref{proof:prop:A-equiv-iff-i} and \ref{proof:prop:A-equiv-iff-iii} 
	are equivalent to 
	condition~\ref{prop:A-equiv-iff-b} of the proposition. 
	By Lemma~\ref{lem:iff-A}\ref{lem:iff-A-HS}, 
	applied for 
	the pair of operators $A, \At^\delta$, 
	\ref{prop:A-equiv-iff-b} 
	is equivalent to  
	$(\At^\delta)^\gamma A^{-\gamma} 
	= \At^{\delta\gamma} A^{-\gamma}$ 
	being an isomorphism on~$E$ 
	and 
	$A^{-\gamma}(\At^{\delta})^{2\gamma} A^{-\gamma} - \id_E 
	= 
	A^{-\gamma}\At^{2\delta\gamma} A^{-\gamma} - \id_E\in\cL_2(E)$ 
	for $\gamma\in[-\beta,\beta]$. 
	The choice $\gamma:=\beta$ 
	shows that 
	$\hdot{2\beta}{A} 
	\cong 
	\hdot{2\betat}{\At}$ 
	and 
	$A^{-\beta}\At^{2\betat} A^{-\beta} - \id_E \in \cL_2(E)$, i.e., 
	\ref{proof:prop:A-equiv-iff-i} and \ref{proof:prop:A-equiv-iff-iii} hold. 
	Conversely, if 
	\ref{proof:prop:A-equiv-iff-i} and~\ref{proof:prop:A-equiv-iff-iii} 
	are satisfied, 
	then by Lemma~\ref{lem:beta-gamma}\ref{lem:beta-gamma-HS} 
	we obtain that 
	$\At^{\delta\gamma} A^{-\gamma}$ 
	is an isomorphism on~$E$ 
	and 
	$A^{-\gamma}\At^{2\delta\gamma} A^{-\gamma} - \id_E \in \cL_2(E)$ 
	for all $\gamma\in[-\beta,\beta]$. 
	Thus, 
	\ref{prop:A-equiv-iff-b} is satisfied by 
	Lemma~\ref{lem:iff-A}\ref{lem:iff-A-HS}. 
	
	Since
	$A^{\eta-1} \At^\delta A^{-\eta} - \id_E 
	= A^{\eta-1} (\At^\delta - A) A^{-\eta}$, 
	the condition
	$\At^{ \delta } - A \in 
	\cL_2\bigl( \hdot{2\eta}{A}; \hdot{2(\eta-1)}{A} \bigr)$  
	implies that   
	$S_\eta 
	:= 
	A^{\eta-1} \At^\delta A^{-\eta} 
	- \id_E  \in \cL_2(E)$ 
	for all $\eta\in\frakN_\beta$. 
	Furthermore, 
	if 
	also 
	$\At^{ \delta } - A$ 
	is in  
	$\cL\bigl( \hdot{2\delta \eta}{\At}; \hdot{2\delta(\eta-1)}{\At} \bigr) 
	= 
	\cL\bigl( \hdot{2\eta}{\At^\delta}; \hdot{2(\eta-1)}{\At^\delta} \bigr) $   
	for $\eta\in\{1,\beta\}$, 
	then 
	by Lemma~\ref{lem:iff-A}\ref{lem:iff-A-iso} 
	(applied for the pair of operators $A,\At^\delta$) 
	$\At^{\delta\gamma} A^{-\gamma}$ is an isomorphism on $E$ for all 
	$\gamma\in[-\beta,\beta]$. 	
	Therefore, 
	$\id_E + S_\eta 
	= 
	A^{\eta-1} \At^{-\delta(\eta-1)} \At^{\delta \eta} A^{-\eta}$ 
	is invertible on $E$, and 
	\ref{prop:A-equiv-iff-b} holds 
	for the choice $U_\eta=\id_E$ for all 
	$\eta\in\frakN_\beta$. 
\end{proof} 

By applying Proposition~\ref{prop:A-equiv} 
for the pair of measures 
$\widehat{\mu}:=\normal\bigl(m,\widehat{A}^{-2}\bigr)$ 
and $\mut = \normal\bigl(\mt, \At^{-2\betat}\bigr)$, 
where ${\widehat{A} := A^\beta}$\!, 
we also obtain a corresponding result 
which includes the case 
that $\beta\in(0,1)$.  

\begin{corollary}\label{cor:A-equiv} 
	Under the assumptions of Proposition~\ref{prop:A-equiv} 
	on $m,\mt,A,\At$, and for   
	$\beta,\betat\in\bbR_+$ such that  
	$A^{-2\beta}\!, \At^{-2\betat}$ 
	have finite traces,  
	the Gaussian measures 
	$\mu = \normal\bigl( m, A^{-2\beta} \bigr)$ 
	and  
	$\mut = \normal\bigl( \mt, \At^{-2\betat} \bigr)$ 
	are equivalent if and only if 
	\begin{enumerate*}[label={\normalfont(\alph*)}]
		\item\label{cor:A-equiv-a} 
		$m - \mt \in \hdot{2\beta}{A}$ 
		and 
		\item\label{cor:A-equiv-b}  
		there exist an orthogonal operator $U\in\cL(E)$ and 
		$S\in\cL_2(E)$
		such that 
		$\At^{\betat} A^{-\beta} = U (\id_{E}+S)$ and  
		$\id_{E}+S$ is invertible on $E$. 
	\end{enumerate*} 
\end{corollary} 

\begin{remark} 
	At first glance, 
	condition~\ref{cor:A-equiv-b} of Corollary~\ref{cor:A-equiv} 
	seems 
	easier compared to~\ref{prop:A-equiv-iff-b} 
	of Proposition~\ref{prop:A-equiv}. 
	We note that the advantage of the latter 
	is that, whenever $\beta=\betat$,  
	it is formulated solely in terms of the base operator $\At$ 
	(and not of powers thereof).  
\end{remark} 

\subsection{Uniformly asymptotically optimal linear prediction}
\label{subsec:gm-on-hs:kriging} 

Throughout this subsection, we let 
$(\cX,d_\cX)$ be a connected, compact metric space 
with positive, finite Borel measure $\nu_\cX$  
(see Subsection~\ref{subsec:notation}) 
and we consider $E = L_2(\cX,\nu_\cX)$.  
Suppose that $\GP\from \cX \times \Omega \to\bbR$ 
is a square-integrable random field with mean 
$m\in L_2(\cX,\nu_\cX) $ and 
covariance operator $A^{-2\beta}$\!, 
where 
$A\from\scrD(A)\subseteq L_2(\cX,\nu_\cX) \to L_2(\cX,\nu_\cX)$ 
is as described 
in Subsection~\ref{subsec:gm-on-hs:auxiliary}. 
Let $\mu = \normal\bigl( m, A^{-2\beta} \bigr)$ 
be the Gaussian measure corresponding to $\GP$ 
and define $\pE[\,\cdot\,]$ as the expectation under $\mu$. 
That is, for 
a random variable $Y\from\Omega\to L_2(\cX,\nu_\cX)$  
with distribution $\mu$ 
and a Borel measurable function 
$g \from  L_2(\cX,\nu_\cX) \to \bbR$, 
we have that  
$m = \pE[Y] := \int_{L_2(\cX,\nu_\cX)} y \,\rd\mu(y)$ 
and, provided that the integral 
$\int_{L_2(\cX,\nu_\cX)} g(y) \,\rd\mu(y)$ exists, 
we define  
$\pE[g(Y)] := \int_{L_2(\cX,\nu_\cX)} g(y) \,\rd\mu(y)$. 

To characterize optimal linear prediction for $\GP$, 
we introduce the centered process 
$\GP^0 := \GP - m$ 
and the vector space $\cZ^0$  
consisting of all linear combinations of the form 
${\alpha_1 \GP^0(x_1) + \ldots + \alpha_K \GP^0(x_K)}$, 
where $K\in\bbN$ and $\alpha_j \in \bbR$, $x_j \in \cX$ 
for all $j\in\{1,\ldots,K\}$.
We then define the Hilbert space~$\cH^0$ as the closure of $\cZ^0$ 
with respect to the norm $\norm{\,\cdot\,}{\cH^0}$ 
induced by the $L_2(\Omega,\bbP)$ inner product,  
\[ 
	\textstyle 
	\biggl( 
	\sum\limits_{j=1}^{K}  \alpha_j  \GP^0(x_j), 
	\, 
	\sum\limits_{k=1}^{K'} \alpha_k' \GP^0(x_k')
	\biggr)_{\cH^0} 
	:= 
	\sum\limits_{j=1}^{K}  
	\sum\limits_{k=1}^{K'} \alpha_j \alpha_k' 
	\pE\bigl[ \GP^0(x_j) \GP^0(x_k') \bigr].
\] 
Since any observation or linear predictor of $\GP$ 
can be represented as $h = c + h^0 $ for $c\in\bbR$ 
and $h^0\in \cH^0$\!, 
we introduce the Hilbert space $\cH$ 
as the direct sum $\cH := \bbR\oplus \cH^0$ 
equipped with the graph norm 
$\|h\|_\cH^2 = |c|^2 + \|h^0\|_{\cH^0}^2$. 
Suppose now that we want to predict $h\in\cH$ 
based on a set of observations $\{y_{nj}\}_{j=1}^n$,  
where $y_{nj} = c_{nj} + y_{nj}^0$ for $c_{nj}\in\bbR$ 
and $y_{nj}^0\in\cH^0$\!. 
Then, the best linear 
predictor 
(aka.\ kriging predictor, see e.g.\ \citep[][Section~1.2]{stein99} 
and \citep[][Section~2]{bk-kriging})  
of $h$ based on these observations 
is the $\cH$-orthogonal projection of $h$ onto the subspace 
\begin{equation}\label{eq:def:cHn}  
	\textstyle 
	\cH_n := 
	\bbR \oplus \cH^0_n 
	= 
	\biggl\{ \alpha_0 + \sum\limits_{j=1}^n \alpha_j y_{nj}^0 
	: \alpha_0, \ldots, \alpha_n \in \bbR \biggr\}, 
	\qquad
	\cH^0_n  
	:= 
	\operatorname{span}\bigl\{y_{nj}^0 \bigr\}_{j=1}^n. 
\end{equation} 
That is, the best 
linear predictor 
$h_n\in\cH_n$ satisfies
\begin{equation}\label{eq:def:hn}
	\textstyle 
	\scalar{ h_n - h, g_n}{\cH} 
	= 0
	\quad 
	\forall 
	g_n \in \cH_n, 
	\quad\;\; 
	\text{and} 
	\qquad 
	\| h_n - h \|_{\cH}
	= 
	\inf_{g_n \in \mathcal{H}_n} \| g_n - h \|_{\cH}. 
\end{equation}  

The question 
is now what happens 
if we replace $h_n$ 
with another linear predictor $\widetilde{h}_n$, 
which is computed based on an incorrect model. 
Specifically, let 
$\mut = \normal\bigl( \mt, \At^{-2\betat} \bigr)$ 
be a second Gaussian measure with corresponding expectation operator 
$\widetilde{\pE}[\,\cdot\,]$, 
and let $\widetilde{h}_n$ 
be the best 
linear predictor for the model $\mut$. 
We are interested in the quality 
of $\widetilde{h}_n$ 
compared to~$h_n$ 
asymptotically as $n\to\infty$. 
For this purpose, 
we assume that the 
set of observations 
$\bigl\{ \{ y_{nj} \}_{j=1}^n : n\in\bbN \bigr\}$ 
yields $\mu$-consistent kriging prediction, i.e.,  
\begin{equation}\label{eq:ass:Hn-dense}
	\lim\limits_{n\to\infty} 
	\pE\bigl[ (h_n - h)^2 \bigr]
	= 
	\lim\limits_{n\to\infty} 
	\| h_n - h \|_{\cH}^2   
	= 0, 
\end{equation}
and we let $\cS^\mu_{\mathrm{adm}}$ 
denote the set of all admissible sequences 
of observations which provide $\mu$-consistent kriging   
prediction:  
\begin{equation}\label{eq:def:S-adm} 
	\begin{split} 
		\cS^\mu_{\mathrm{adm}}  
		:= 
		\bigl\{ \{\cH_n\}_{n\in\bbN} \; \bigl| \; 
		&\forall n \in \bbN:  
		\cH_n \text{ is as in \eqref{eq:def:cHn} with}\,
		\dim(\cH_n^0)=n,  
		\\
		&\forall h\in\cH:  
		\{h_n\}_{n\in\bbN}   
		\text{ defined by \eqref{eq:def:hn}		
			satisfy \eqref{eq:ass:Hn-dense}} \bigr\}.  
	\end{split} 
\end{equation} 
By combining the results of 
Lemmas~\ref{lem:beta-gamma} and~\ref{lem:iff-A} 
with \citep[][Theorem~3.8]{bk-kriging} 
we obtain the following result 
on uniformly asymptotically optimal linear prediction when 
misspecifying $\mu$ by $\mut$. 

\begin{proposition}\label{prop:A-kriging} 
	Let $h_n, \widetilde{h}_n$ denote the best linear
	predictors of $h\in\cH$ based on~$\cH_n$~and the measures
	${\mu = \normal\bigl( m, A^{-2\beta} \bigr)}$ 
	and
	$\mut = \normal\bigl( \mt, \At^{-2\betat} \bigr)$, respectively. 
	Here, assume that 
	$m, \mt \in  L_2(\cX,\nu_\cX)$, 
	$A\from\scrD(A) \subseteq  L_2(\cX,\nu_\cX) \to  L_2(\cX,\nu_\cX)$ and 
	$\At\from\scrD(\At)\subseteq  L_2(\cX,\nu_\cX) \to  L_2(\cX,\nu_\cX)$ 
	are densely defined, self-adjoint, positive definite linear operators  
	with compact inverses on $L_2(\cX,\nu_\cX)$. 
	In addition, $\beta,\betat \in \bbR_+$ 
	are such that 
	$A^{-2\beta}$ and $\At^{-2\betat}$ 
	have finite traces on $ L_2(\cX,\nu_\cX)$ 
	and $\delta:=\nicefrac{\betat}{\beta}$. 
	
	\begin{enumerate}[label={\normalfont\Roman*.}, itemsep=3pt] 
		\item\label{prop:A-kriging-I} 
		Set $\cH_{-n}:=\bigl\{ h\in\cH:\pE\bigl[ (h_n - h)^2 \bigr]> 0 \bigr\}$. 
		Any of the 
		following four asymptotic statements,
		\begin{align}
			\lim_{n\to\infty}
			\sup_{h\in \cH_{-n}}
			\frac{
				\pE\bigl[ ( \widetilde{h}_n - h)^2 \bigr]
			}{
				\pE\bigl[ ( h_n - h)^2 \bigr]
			} 
			&= 1, &
			\lim_{n\to\infty}
			\sup_{h\in \cH_{-n}}
			\frac{
				\widetilde{\pE}\bigl[ ( h_n - h)^2 \bigr]
			}{
				\widetilde{\pE}\bigl[ ( \widetilde{h}_n - h)^2 \bigr]
			} 
			&= 1, 
			\label{eq:prop:A-kriging-1} 
			\\	
			\lim_{n\to\infty}
			\sup_{h\in \cH_{-n}}
			\left|
			\frac{
				\widetilde{\pE}\bigl[ ( h_n - h)^2 \bigr]
			}{
				\pE\bigl[ ( h_n - h)^2 \bigr]
			} - c 
			\right| 
			&= 0, &
			\;\;\;  
			\lim_{n\to\infty}
			\sup_{h\in \cH_{-n}}
			\left|
			\frac{
				\pE\bigl[ ( \widetilde{h}_n - h)^2 \bigr]
			}{
				\widetilde{\pE}\bigl[ ( \widetilde{h}_n - h)^2 \bigr]
			} - \frac{1}{c} 
			\right| 
			&= 0, 
			\label{eq:prop:A-kriging-2} 
		\end{align} 
		holds for all 
		$\{\cH_n\}_{n\in\bbN} \in \cS^\mu_{\mathrm{adm}}$ 
		(and in \eqref{eq:prop:A-kriging-2} for some $c\in\bbR_+$)
		if and only if 
		\begin{enumerate}[label={\normalfont(\alph*)}]
			\item\label{prop:A-kriging-iff-a} 
			the difference of the means 
			satisfies $m - \mt \in \hdot{2\beta}{A}$\!; 
			\item\label{prop:A-kriging-iff-b} 
			there exist $c\in\bbR_+$,   
			an orthogonal operator 
			$W$ on $L_2(\cX,\nu_\cX)$, 
			and $K\in\cK(L_2(\cX,\nu_\cX))$ 
			such that  
			$c^{\nicefrac{1}{2}} \At^{\betat} A^{-\beta} 
			= W (\id_{L_2(\cX,\nu_\cX)} + K)$ 
			and $\id_{L_2(\cX,\nu_\cX)} + K$ is invertible. 
		\end{enumerate} 
		In this case, the constant $c\in\bbR_+$ in condition \ref{prop:A-kriging-iff-b} 
		coincides with that in \eqref{eq:prop:A-kriging-2}.  
		
		\item\label{prop:A-kriging-II} 
		For $\beta\in[1,\infty)$, 
		condition \ref{prop:A-kriging-iff-b} 
		is equivalent to requiring that 
		there exists $c\in\bbR_+$ such that 
		for all $\eta\in\frakN_\beta$, 
		where $\frakN_\beta$ is defined as in \eqref{eq:frakNset}, 
		there exist   
		an orthogonal operator 
		$W_\eta$ on $L_2(\cX,\nu_\cX)$  
		and 
		${K_\eta \in \cK(L_2(\cX,\nu_\cX))}$   
		such that 
		$c^{\frac{1}{2\beta}} A^{\eta-1} \At^\delta A^{-\eta} 
		= W_\eta ( \id_{L_2(\cX,\nu_\cX)} + K_\eta )$ 
		and 
		$ \id_{L_2(\cX,\nu_\cX)}+K_\eta$ is invertible 
		on $L_2(\cX,\nu_\cX)$. 
		This is satisfied, whenever the following holds:  
		\[
			\forall \eta\in\frakN_\beta: 
			\quad 
			c^{\frac{1}{2\beta}} \At^{ \delta } - A \in 
			\cK\bigl( \hdot{2\eta}{A}; \hdot{2(\eta-1)}{A} \bigr) 
			\cap
			\cL\bigl( \hdot{2\delta \eta}{\At}; \hdot{2\delta(\eta-1)}{\At} \bigr). 
		\] 
	\end{enumerate} 
\end{proposition} 

\begin{proof} 
	By \cite[][Theorem~3.8 and Lemma~B.1]{bk-kriging} 
	any of the assertions in 
	\eqref{eq:prop:A-kriging-1} or \eqref{eq:prop:A-kriging-2}  
	holds for every  
	$\{\cH_n\}_{n\in\bbN} \in \cS^\mu_{\mathrm{adm}}$ 
	and 
	for some constant $c\in\bbR_+$ 
	if and only if 
	\begin{enumerate}[label={\normalfont(\roman*)}]
		\item\label{proof:prop:A-kriging-iff-i} 
		the Cameron--Martin spaces 
		$\hdot{2\beta}{A}$ 
		and 
		$\hdot{2\betat}{\At}$ 
		are norm equivalent Hilbert spaces;  
		\item\label{proof:prop:A-kriging-iff-ii} 
		the difference of the means satisfies 
		$m - \mt \in \hdot{2\beta}{A}$\!; and 
		\item\label{proof:prop:A-kriging-iff-iii} 
		$A^{-\beta} \At^{2\betat} A^{-\beta} - c^{-1} \id_E 
		= 
		c^{-1} \bigl( A^{-\beta} \bigl( c\At^{2\betat} \bigr) A^{-\beta} - \id_E \bigr)$ 
		is compact on $E$. 
	\end{enumerate} 
	The proof for \ref{prop:A-kriging-II}, 
	when $\beta\in[1,\infty)$, can then be completed 
	as in the proof of Proposition~\ref{prop:A-equiv}, 
	namely by using Lemma~\ref{lem:beta-gamma}\ref{lem:beta-gamma-compact} 
	and 
	Lemma~\ref{lem:iff-A}\ref{lem:iff-A-compact}/\ref{lem:iff-A-iso} 
	for the pair of operators~$A$ and 
	$c^{\frac{1}{2\beta}}  \At^\delta$\!. 
	Finally, the general statement \ref{prop:A-kriging-I} 
	for $\beta,\betat\in\bbR_+$ follows 
	similarly as Corollary~\ref{cor:A-equiv}. 
\end{proof}

\section{Some explicit choices for the base operators}
\label{section:examples} 

In this section we illustrate the abstract results of 
Section~\ref{section:gm-on-hs}
by two first examples 
before discussing their 
implications for 
generalized Whittle--Mat\'ern 
fields in the next section: 
In Subsection~\ref{subsec:examples:abstract} 
we consider 
the case that the base operators 
$A$ and $\At$ diagonalize with respect to  
the same eigenbasis $\{e_j\}_{j\in\bbN}$ for~$E$. \linebreak
This setting applies to 
classical Whittle--Mat\'ern fields 
with constant coefficients, which solve SPDEs 
of the form \eqref{eq:statmodel}   
on a bounded domain $\cD\subset\bbR^d$. 
We subsequently discuss this example 
in Subsection~\ref{subsec:wexamples:stat-wm}. 

\subsection{Operators with the same eigenbasis}
\label{subsec:examples:abstract}

We note that the scope for applications of the 
following corollary is considerably wider 
than the classical Whittle--Mat\'ern 
class discussed in 
Subsection~\ref{subsec:wexamples:stat-wm}. 
For instance, it can be used for 
periodic random fields 
on $\cX = [0,1]^d$ as considered by 
Stein~\cite{stein97}, 
random fields on 
the sphere $\cX=\mathbb{S}^2$ defined 
via the spherical harmonics, 
see e.g.\ \cite{Guinness2016} and \cite[Section~6.3]{bk-kriging}, 
or more generally Gaussian processes 
on compact Riemannian manifolds 
defined via the eigenfunctions 
of the Laplace--Beltrami operator~\cite{Borovitskiy2020}. 

\begin{corollary}\label{cor:same-eigenbasis}  
	Let $A\from\scrD(A) \subseteq E \to E$ and 
	$\At\from\scrD(\At)\subseteq E \to E$ 
	be two densely defined, self-adjoint, 
	positive definite linear operators 
	with compact inverses on $E$. 
	In addition, 
	assume that 
	$A$ and $\At$ diagonalize 
	with respect to the same 
	eigenbasis $\{e_j\}_{j\in\bbN}$ for $E$, i.e., 
	there exist corresponding eigenvalues 
	$\bbR_+ \ni \lambda_j, \widetilde{\lambda}_j \to \infty$ 
	(as $j\to\infty$) 
	such that $A e_j = \lambda_j e_j$ 
	and $\At e_j = \lambdat_j e_j$ 
	for all $j\in\bbN$. 
	Let $m,\mt\in E$, and 
	assume that 
	$\beta,\delta \in \bbR_+$ are such that 
	$A^{-2\beta}$ and $\At^{-2\delta\beta}$ 
	have finite traces on~$E$. 
	Then, the Gaussian measures 
	$\mu = \normal\bigl( m, A^{-2\beta} \bigr)$ 
	and
	$\mut = \normal\bigl( \mt, \At^{-2\delta\beta} \bigr)$
	satisfy the following: 
	\begin{enumerate}[label={\normalfont\Roman*.}]
		\item\label{cor:same-eigenbasis:CM} 
		The Cameron--Martin spaces for 
		$\mu$ and $\mut$ are isomorphic, 
		with equivalent norms, 
		if and only if there exist $c_-,c_+\in\bbR_+$ 
		such that $c_j:= \lambdat_j^\delta \lambda_j^{-1} \in [c_-,c_+]$ 
		for all $j\in\bbN$. 
		\item\label{cor:same-eigenbasis:equivalence}  
		The measures $\mu$ and $\mut$ are equivalent 
		if and only if 
		$m-\mt\in\hdot{2\beta}{A}$ 
		and 
		$\sum_{j\in\bbN} (c_j - 1)^2 < \infty$. 
		\item\label{cor:same-eigenbasis:kriging} 
		Any of the four assertions 
		in \eqref{eq:prop:A-kriging-1}, \eqref{eq:prop:A-kriging-2} 
		holds for 
		all $\{\cH_n\}_{n\in\bbN}\in\cS^\mu_{\mathrm{adm}}$ 
		and in \eqref{eq:prop:A-kriging-2} 
		for some $c\in\bbR_+$ 
		if and only if 
		$m-\mt\in\hdot{2\beta}{A}$  
		and 
		$\lim_{j\to\infty} c_j = \hat{c}$ for some $\hat{c}\in\bbR_+$. 
		Then, $c = \hat{c}^{-2\beta}$\!. 
	\end{enumerate} 
\end{corollary} 

\begin{proof} 
	Define  
	$c_- 
	:= 
	\inf_{j\in\bbN} \lambdat_j^\delta\lambda_j^{-1} 
	\in[0,\infty)$ 
	and 
	$c_+ 
	:= 
	\sup_{j\in\bbN} \lambdat_j^\delta\lambda_j^{-1}
	\in(0,\infty]$.  
	Then, we obtain that 
	$\norm{\phi}{2\delta\beta,\At} 
	\leq 
	c_+^{\beta} \norm{\phi}{2\beta,A}$ 
	and 
	$c_-^{\beta} 
	\norm{\psi}{2\beta,A} 
	\leq 
	\norm{\psi}{2\delta\beta,\At}$ 
	for all $\phi\in\hdot{2\beta}{A}$ and $\psi\in\hdot{2\delta\beta}{\At}$\!. 
	Therefore, 
	the Cameron--Martin spaces for $\mu$ and $\mut$, 
	see~\eqref{eq:general-CM-space}, 
	are isomorphic 
	if and only if 
	$c_->0$ and $c_+<\infty$. 
	
	Note that $(c_j)_{j\in\bbN}\subset\bbR_+$ 
	and $\sum_{j\in\bbN}(c_j-1)^2<\infty$ 
	imply that 
	$0<c_-\leq c_j\leq c_+<\infty$
	for all~$j\in\bbN$. 
	Thus,~\ref{cor:same-eigenbasis:equivalence} 
	follows from~\ref{cor:same-eigenbasis:CM} 
	and the Feldman--H\'ajek theorem, 
	Theorem~\ref{thm:feldman-hajek},  
	since by the mean value~theo-\vspace{-3mm}\linebreak\pagebreak
	
	\noindent
	rem  
	for $t\mapsto t^{2\beta}$, 
	$\| A^{-\beta} \At^{2\delta\beta} A^{-\beta} - \id_E \|_{\cL_2(E)}^2
	= 
	\sum_{j\in\bbN} 
	\bigl( c_j^{2\beta} - 1 \bigr)^2
	=    
	4\beta^2 
	\sum_{j\in\bbN}  
	\xi_j^{2(2\beta-1)} 
	( c_j - 1 )^2$\!, 
	where $(\xi_j)_{j\in\bbN}$ 
	satisfy that 
	$\xi_j \in [\min\{c_j,1\}, \max\{c_j,1\}] 
	\subseteq [\min\{c_-,1\}, \max\{c_+,1\}]$ 
	for all $j\in\bbN$. 
	
	Finally, the assertion~\ref{cor:same-eigenbasis:kriging} 
	has already been observed in 
	\cite[Corollary~5.1]{bk-kriging};  
	there formulated in terms of the 
	ratio $\gammat_j/\gamma_j\to c$ 
	(as $j\to\infty$) 
	of the eigenvalues 
	$(\gammat_j)_{j\in\bbN}$ 
	and $(\gamma_j)_{j\in\bbN}$
	of the covariance operators $\CC=\At^{-2\delta\beta}$ 
	and $\cC = A^{-2\beta}$. 
	Thus,  
	$\gammat_j / \gamma_j 
	=  \lambdat_j^{-2\delta\beta} \lambda_j^{2\beta} 
	= c_j^{-2\beta}$ 
	and the claim follows. 
\end{proof}

\begin{remark} 
	Note that for  
	$\mu=\normal\bigl( 0, A^{-2\beta} \bigr)$ 
	and
	$\mut=\normal\bigl( 0,\At^{-2\delta\beta} \bigr)$ 
	the conditions 
	in all parts 
	of Corollary~\ref{cor:same-eigenbasis} 
	are independent of $\beta\in\bbR_+$. 
	This implies that once a property  
	(equivalent Cameron--Martin spaces, 
	equivalence of measures, 
	or uniformly asymptotically optimal linear prediction)
	is established for 
	$\mu,\mut$ 
	and a fixed $\beta=\beta_0\in\bbR_+$, 
	it follows also for all other meaningful values 
	of $\beta\in\bbR_+$ 
	so that $A^{-2\beta}\!,\At^{-2\delta\beta}$ 
	have finite traces. 
	Thus, in the case 
	that $A$ and $\At$ diagonalize with 
	respect to the same eigenbasis, 
	besides concluding the corresponding 
	property for $\beta\leq\beta_0$  
	(by means of Lemma~\ref{lem:beta-gamma}) 
	one obtains it also for $\beta>\beta_0$. 
	This observation 
	holds even for more general 
	base operators $A,\At$. 
	Specifically, if their fractional 
	powers commute, i.e., 
	$\scrD( \At^\vartheta A^r ) 
	= 
	\scrD( A^r \At^\vartheta )$ 
	and 
	$\At^\vartheta A^r \psi 
	= 
	A^r \At^\vartheta \psi$  
	for all $r,\vartheta\in\bbR$ and 
	$\psi\in\scrD( \At^\vartheta A^r )$, 
	then the operators $A^{\eta-1}\At^\delta A^{-\eta}$, 
	$\eta\in\frakN_\beta$, 
	appearing in the conditions of 
	Propositions~\ref{prop:A-equiv} and~\ref{prop:A-kriging} 
	simplify to $\At^\delta A^{-1}$ 
	and the conditions become independent of $\beta\in[1,\infty)$. 
\end{remark} 

\subsection{Whittle--Mat\'ern operators with constant coefficients} 
\label{subsec:wexamples:stat-wm}

We now discuss \emph{classical}  
Whittle--Mat\'ern fields 
solving the SPDE \eqref{eq:statmodel}   
on a bounded domain. 
To this end, 
let $\emptyset\neq\cD\subset\bbR^d$ be a  
connected, bounded 
and open domain, 
with Lipschitz boundary $\partial\cD$. 
Further,  
\begin{equation}\label{eq:L-stat}
	L v  := \bigl( -\Delta + \kappa^2 \bigr) v, 
	\qquad 
	v \in \scrD(L) := H^2(\cD) \cap H^1_0(\cD), 
\end{equation}
is the negative Laplacian, 
shifted by $\kappa^2\in[0,\infty)$ 
and augmented 
with homogeneous Dirichlet boundary 
conditions, 
see Appendix~\ref{app:subsec:pdes}. 
By Proposition~\ref{prop:wellposed} 
$L^{-2\beta}$ has a finite trace 
if and only if $\beta\in(\nicefrac{d}{4},\infty)$. 
Recall that for the 
SPDE \eqref{eq:statmodel} on $\bbR^d$  
this condition corresponds 
to a positive smoothness 
parameter $\nu=2\beta-\nicefrac{d}{2}\in\bbR_+$.  
For two classical Whittle--Mat\'ern fields 
with parameters $(\beta,\tau,\kappa)$ and $(\betat,\taut,\kappat)$, 
where $\tau,\taut$ scale the variances 
of the fields, cf.~\eqref{eq:model_ex2},  
we obtain the following result 
from Corollary~\ref{cor:same-eigenbasis}. 

\begin{corollary}\label{cor:stat-wm} 
	Let $d\in\bbN$, 
	$\beta,\betat\in(\nicefrac{d}{4},\infty)$,  
	$\tau,\taut\in\bbR_+$,  
	and let $L,\Lt$ be defined as in \eqref{eq:L-stat} 
	with shift parameters $\kappa^2\in[0,\infty)$ 
	and $\kappat^2\in[0,\infty)$, respectively. 
	Assume that 
	$m,\mt\in L_2(\cD)$ 
	and consider the Gaussian measures 
	$\mu=\normal\bigl( m,\tau^{-2} L^{-2\beta} \bigr)$ and 
	$\mut=\normal\bigl( \mt,\taut^{-2} \Lt^{-2\betat} \bigr)$ 
	on the Hilbert space $L_2(\cD)$.   
	\begin{enumerate}[label={\normalfont\Roman*.}]
		\item\label{cor:stat-wm:CM} 
		The Cameron--Martin spaces for 
		$\mu,\mut$  
		are isomorphic, 
		with equivalent norms, 
		if and only if $\beta=\betat$. 
		\item\label{cor:stat-wm:equivalence}  
		In dimension $d\leq 3$, 
		$\mu$ and $\mut$ are equivalent 
		if and only if 
		$\beta=\betat$, 
		$\tau=\taut$ 
		and $m-\mt\in\Hdot{2\beta}$\!. 
		In dimension $d\geq 4$, 
		$\mu$ and $\mut$ are equivalent 
		if and only if 
		$\beta=\betat$, 
		$\tau=\taut$, 
		$m-\mt\in\Hdot{2\beta}$ 
		and $\kappa^2=\kappat^2$\!. 
		\item\label{cor:stat-wm:kriging} 
		In every dimension $d\in\bbN$, 
		any of the four assertions 
		in \eqref{eq:prop:A-kriging-1}, \eqref{eq:prop:A-kriging-2} 
		holds for 
		every sequence $\{\cH_n\}_{n\in\bbN}\in\cS^\mu_{\mathrm{adm}}$ 
		and in \eqref{eq:prop:A-kriging-2} for some $c\in\bbR_+$ 
		if and only if 
		$\beta=\betat$ and $m-\mt\in\Hdot{2\beta}$\!. 
	\end{enumerate} 
\end{corollary} 

\begin{proof}
	Letting $(\widehat{\lambda}_j)_{j\in\bbN}$ denote the 
	eigenvalues of the negative Dirichlet 
	Laplacian $-\Delta$
	with corresponding eigenfunctions 
	$\{e_j\}_{j\in\bbN}$ 
	forming an orthonormal basis of $L_2(\cD)$, 
	we find that $\{e_j\}_{j\in\bbN}$  
	is also an eigenbasis 
	for $L$ and for $\Lt$, 
	with eigenvalues 
	$\lambda_j = \widehat{\lambda}_j + \kappa^2$ 
	and $\lambdat_j = \widehat{\lambda}_j + \kappat^2$, respectively. 
	The asymptotic behavior 
	$\widehat{\lambda}_j \eqsim j^{\nicefrac{2}{d}}$, 
	see \eqref{eq:lambdaj}, 
	shows that 
	${c_j := \taut^{\nicefrac{1}{\beta}} \tau^{-\nicefrac{1}{\beta}} 
		\lambdat_j^{\delta} 
		\lambda_j^{-1} \in [c_-, c_+]}$ 
	holds for some ${c_-,c_+\in\bbR_+}$   
	and all $j\in\bbN$ 
	if and only if $\delta=\nicefrac{\betat}{\beta}=1$; 
	then,  
	$\lim_{j\to\infty} c_j = \bigl( \frac{\taut}{\tau} \bigr)^{\nicefrac{1}{\beta}}$ 
	so that $\tau=\taut$ is necessary for $\mu\sim\mut$. 
	Then, 
	$(c_j - 1)^2 \eqsim \bigl( \kappat^2 - \kappa^2 \bigr)^2 j^{-\nicefrac{4}{d}}$ 
	and all assertions  
	follow from Corollary~\ref{cor:same-eigenbasis}.\pagebreak 
\end{proof}

\section{Generalized Whittle--Mat\'ern  
	fields on Euclidean domains}
\label{section:wm-D} 

Throughout this section, 
let $\cD\subset\bbR^d$  
be a nonempty, connected, bounded and open domain 
with Lipschitz continuous 
boundary $\partial\cD$ 
(see~Definition~\ref{def:smooth-bdry} 
in Appendix~\ref{appendix:function-spaces}) 
and closure $\clos{\cD} = \cD\cup\partial\cD$. 

The purpose of this section 
is to generalize the results 
obtained in Corollary~\ref{cor:stat-wm} 
for classical Whittle--Mat\'ern 
fields to the class of generalized 
Whittle--Mat\'ern 
fields \eqref{eq:whittle-matern} on $\cD$, 
where $\kappa$ and $\ac$  
are functions 
describing 
spatially varying correlation ranges 
and anisotropies, respectively. 
The difficulty 
of this generalization 
lies in the fact that 
the covariance operators 
of two generalized 
Whittle--Mat\'ern fields
do not necessarily  
have the same eigenfunctions. 
For this reason, more 
sophisticated arguments 
and tools from 
spectral theory and PDE theory 
are needed.  
We refer 
to Appendix~\ref{appendix:function-spaces} 
for an overview
of several important results from PDE theory 
and all relevant function spaces, 
such as the Lebesgue spaces 
$L_p(\cD)$, $L_p(\partial\cD)$, $p\geq 1$, 
the spaces of smooth functions 
$C^\infty(\clos{\cD})$ and $C^\infty_c(\cD)$, 
the (fractional-order) Sobolev spaces $H^r(\cD)$ for $r\in\bbR_+$, 
and the subspace $H^1_0(\cD)\subset H^1(\cD)$. 

\subsection{Setting and summary of the main results}
\label{subsec:wm-D:summary}

In order to properly define the 
class of generalized Whittle--Mat\'ern fields \eqref{eq:whittle-matern},
we consider for $\beta\in\bbR_+$ the fractional-order SPDE 
\begin{align}\label{eq:Lbeta}
	L^\beta \GP =  \white, 
	\qquad 
	\bbP\text{-almost surely}, 
\end{align}
where $\white$ denotes Gaussian white noise 
on the Hilbert space $L_2(\cD)$, and 
$L^{\beta}$ is a (possibly fractional) power 
of an elliptic differential operator $L$
which determines the covariance structure of 
the random field~$\GP\from\clos{\cD}\times\Omega\to\bbR$.
Specifically, we assume that  
$L\from\mathscr{D}(L) \subseteq L_2(\cD) \cap H^1_0(\cD) \to L_2(\cD)$
is a linear, symmetric, second-order 
differential operator in divergence form 
with homogeneous Dirichlet boundary conditions 
(see Appendix~\ref{app:subsec:pdes}), 
formally given by 
\begin{equation}\label{eq:L-div} 
	L v  = - \nabla \cdot(\ac \nabla v) + \kappa^2 v,
	\qquad 
	v \in \scrD(L) \subseteq L_2(\cD) \cap H^1_0(\cD). 
\end{equation}
Here, we suppose that  
$\ac$ and $\kappa$ 
in~\eqref{eq:L-div}  and 
the spatial domain $\cD\subset\bbR^d$ 
satisfy the following conditions. 

\begin{assumption}
	\begin{enumerate}[label={\normalfont\Roman*.}] 
		\item\label{ass:a-kappa-D:a}  
		$\ac\from\overline{\cD}\to\bbR^{d\times d}$ 
		is symmetric 
		and
		uniformly positive definite, i.e.,
		\[ 
			\exists a_0 > 0 : 
			\quad
			\forall \xi \in \bbR^d : 
			\quad 
			\operatorname{ess} \inf\nolimits_{\s\in\cD} 
			\xi^\trsp \ac(\s) \xi 
			\geq a_0 \norm{ \xi }{\bbR^d}^2. 
		\]
		In addition,  
		$\ac=(\ac_{jk})_{j,k=1}^d$ is smooth, 
		$\ac_{jk} \in C^\infty(\clos{\cD})$ 
		for all  $j,k \in\{1,\ldots,d\}$. 
		\item\label{ass:a-kappa-D:kappa}  
		$\kappa \from \clos{\cD}\to\bbR$ 
		is smooth, $\kappa\in C^\infty(\clos{\cD})$. 
		\item\label{ass:a-kappa-D:D}  
		The domain $\cD\subset\bbR^d$ 
		has a smooth boundary $\partial\cD$ of 
		class $C^\infty$, 
		see~Definition~\ref{def:smooth-bdry} 
		in Appendix~\ref{appendix:function-spaces}. 
	\end{enumerate}\label{ass:a-kappa-D} 
\end{assumption} 

Provided that Assumptions~\ref{ass:a-kappa-D}.I--II    
are satisfied, the differential operator 
$L$ in \eqref{eq:L-div} 
is strongly elliptic and 
induces a symmetric, 
continuous and coercive   
bilinear form on~$H^1_0(\cD)$,    
\begin{equation}\label{eq:a-L} 
	a_L \from H^1_0(\cD)\times H^1_0(\cD) \to \bbR, 
	\qquad 
	a_L (u,v) 
	:= 
	\scalar{ \ac \nabla u, \nabla v }{L_2(\cD)}
	+ 
	\scalar{ \kappa^2 u, v }{L_2(\cD)}. 
\end{equation}
The domain of the 
operator $L\from\scrD(L)\subset L_2(\cD) \to L_2(\cD)$ 
is given by 
$\mathscr{D}(L) = 
H^2(\cD)\cap H^1_0(\cD)$ and, 
in particular, we find that 
$L$ 
is densely defined and self-adjoint. 
Furthermore, the Rellich--Kondrachov 
compactness theorem 
\cite[Theorem~6.3]{AdamsSobolev2003}  
implies that 
$L^{-1}\from L_2(\cD)\to H^1_0(\cD)\subset L_2(\cD)$ 
is compact on~$L_2(\cD)$, 
see Appendix~\ref{app:subsec:pdes} 
for more details.  
For this reason, there exists 
a countable system of eigenfunctions 
$\{e_j\}_{j\in\bbN}$ of $L$ 
which can be chosen as  
an orthonormal basis for $L_2(\cD)$. 
We assume that the 
corresponding positive eigenvalues 
$( \lambda_j )_{j\in\bbN}$
are in non-decreasing order, 
$0<\lambda_1\leq\lambda_2\leq\ldots$, 
and repeated according to multiplicity.  
The fractional power operator $L^\beta$ in 
the SPDE \eqref{eq:Lbeta} 
is then defined in the spectral sense 
as in \eqref{eq:def:Abeta}, 
with $A:=L$ and $E:=L_2(\cD)$.

Weyl's law \citep[Theorem~6.3.1]{Davies1995}  
states that the eigenvalues 
$(\lambda_j )_{j\in\bbN}\subset\bbR_+$ 
of the strongly elliptic second-order differential operator~$L$ 
satisfy the spectral asymptotics 
\begin{align}\label{eq:lambdaj}
	\exists c_\lambda, C_\lambda \in\bbR_+ : 
	\quad 
	c_\lambda j^{\nicefrac{2}{d}} 
	\leq 
	\lambda_j 
	\leq 
	C_\lambda j^{\nicefrac{2}{d}}
	\quad
	\forall j \in \bbN.
\end{align}

Existence and uniqueness of the solution $\GP$ 
to \eqref{eq:Lbeta} 
thus follow from 
\citep[Proposition~2.3, Remark~2.4]{BKK2020} 
or  
\citep[Lemma~3]{cox2020}. 
We recapitulate this result in 
the next proposition. 
\begin{proposition}\label{prop:wellposed}
	For $\cD\subset\bbR^d$, $d\in\bbN$, 
	let the differential operator $L$ 
	be as in \eqref{eq:L-div}, and 
	suppose that $\ac$ and $\kappa$ satisfy 
	Assumptions~\ref{ass:a-kappa-D}.I--II. 
	Then, the SPDE \eqref{eq:Lbeta} has a unique solution 
	$\GP\in L_p(\Omega;L_2(\cD))$ 
	for any $p\in[1,\infty)$---or, in other words, 
	the probability distribution 
	of $\GP$ in \eqref{eq:Lbeta} 
	defines a Gaussian measure on 
	the Hilbert space $L_2(\cD)$---if 
	and only if 
	$\beta \in (\nicefrac{d}{4}, \infty)$. 
\end{proposition}


In the case that the parameters  
in~\eqref{eq:Lbeta} and \eqref{eq:L-div} 
are given by the 
triple 
\[
(\beta, \ac, \kappa) 
\in 
\bbR_+ 
\times 
C^{\infty}(\clos{\cD})^{d\times d} 
\times 
C^{\infty}(\clos{\cD}),
\] 
we say that $\GP$ solves the SPDE~\eqref{eq:Lbeta}
for $(\beta, \ac, \kappa)$.  
By Proposition~\ref{prop:wellposed} 
the probability distribution 
of the zero-mean generalized Whittle--Mat\'ern field 
$\GP\from\cD\times\Omega\to\bbR$ 
solving \eqref{eq:Lbeta} 
for the parameter triple  
$(\beta, \ac, \kappa)$
defines a Gaussian measure 
$\mu_d(0;\beta,\ac,\kappa)$ on $L_2(\cD)$ 
if and only if $\beta \in(\nicefrac{d}{4},\infty)$. 
In this case, 
for every Borel set $B\in\cB(L_2(\cD))$, 
it is given by 
\[
	\mu_d(0;\beta,\ac,\kappa)(B) 
	= 
	\bbP( \{ \omega\in\Omega : \GP(\,\cdot\,, \omega) \in B, 
	\; \GP \text{ solves } \eqref{eq:Lbeta} \} ). 
\]
Thus, it has mean zero  
and trace-class covariance operator 
$\cC=L^{-2\beta}\in\cL( L_2(\cD) )$, 
cf.~\eqref{eq:mu-mean-cov}, 
\begin{align*}  
	\scalar{\cC \psi, \psi'}{L_2(\cD)} 
	&= 
	\int_{L_2(\cD)} \scalar{\psi, \phi}{L_2(\cD)} \scalar{\phi, \psi'}{L_2(\cD)} \, \rd\mu(\phi)
	\\ 
	&= 
	\int_\Omega  \scalar{\psi, \GP(\,\cdot\,, \omega)}{L_2(\cD)}  
	\scalar{ \GP(\,\cdot\,, \omega), \psi' }{L_2(\cD)} \, \rd \bbP(\omega)
	= 
	\bigl( L^{-2\beta} \psi, \psi' \bigr)_{L_2(\cD)}. 
\end{align*} 

In summary, 
for $\beta\in(\nicefrac{d}{4},\infty)$, 
the Whittle--Mat\'ern field 
$\GP\from \cD\times\Omega \to \bbR$ 
in \eqref{eq:Lbeta} induces 
a Gaussian 
measure on $L_2(\cD)$ 
given by 
$\mu_d(0;\beta,\ac,\kappa) 
= 
\normal\bigl( 0, L^{-2\beta} \bigr)$, 
see \eqref{eq:mu-mean-cov}.  
More generally, we 
consider  
for an arbitrary mean value function $m\in L_2(\cD)$: 
\begin{equation}\label{eq:def:mu}
	\mu_d(m;\beta,\ac,\kappa) 
	:= 
	\normal\bigl( m, L^{-2\beta} \bigr).  
\end{equation}
The goal of this section 
is to identify the following: 
\begin{enumerate}[label=(\alph*)] 
	\item\label{enum:goal1} the Cameron--Martin space 
	for the Gaussian measure 
	$\mu_d(m;\beta,\ac,\kappa)$, 
	as well as   
	necessary and sufficient conditions
	for the Cameron--Martin spaces of two measures 
	$\mu_d(m; \beta, \ac, \kappa)$ and 
	$\mu_d(\mt; \betat, \act, \kappat)$ 
	to be isomorphic and norm equivalent;  
	\item\label{enum:goal2}  
	necessary and sufficient conditions
	for two measures 
	$\mu_d(m; \beta, \ac, \kappa)$ and 
	$\mu_d(\mt; \betat, \act, \kappat)$ 
	to be equivalent (respectively, orthogonal); and 
	\item\label{enum:goal3} 
	necessary and sufficient conditions 
	for $\mu_d(\mt; \betat, \act, \kappat)$ 
	to provide uniformly asymptotically optimal linear prediction
	in the case that $\mu_d(m; \beta, \ac, \kappa)$ 
	is the correct model. 
\end{enumerate} 

These questions are addressed in the following  
Subsections~\ref{subsec:wm-D:CM}, \ref{subsec:wm-D:equiv} 
and \ref{subsec:wm-D:kriging}. 
We will see that 
the necessary and sufficient conditions 
mentioned in \ref{enum:goal1}, \ref{enum:goal2} and \ref{enum:goal3} above 
all include the requirement that $\beta=\betat$. 
Depending on the value of $\beta\in(\nicefrac{d}{4},\infty)$, 
it is solely the behavior of $\delta_{\ac} := \act - \ac$ 
and $\delta_{\kappa^2} := \kappat^2 - \kappa^2$ 
at the boundary $\partial\cD$ that matters 
for \ref{enum:goal1}, see Theorem~\ref{thm:CM}. 
Finally, 
for \ref{enum:goal2} and \ref{enum:goal3} 
also conditions on $\ac,\act$ and $\kappa^2\!,\kappat^2$ 
inside the domain $\cD\subset\bbR^d$ are imposed 
which for \ref{enum:goal2}, 
equivalence of measures, additionally depend 
on the dimension $d\in\bbN$, 
see Theorems~\ref{thm:equivalence} and \ref{thm:kriging}. 
We summarize the main outcomes 
of these theorems in Table~\ref{tab:summary}.

\begin{table}
	\caption{Necessary and sufficient conditions for 
		(a) isomorphic, norm equivalent 
		Cameron--Martin spaces 
		of $\mu := \mu_d(0;\beta,\ac,\kappa)$ 
		and $\mut := \mu_d(0;\betat,\act,\kappat)$, 
		(b) equivalence of measures $\mu\sim\mut$, and 
		(c) uniformly asymptotically optimal linear prediction 
		when misspecifying $\mu$ by $\mut$. 
		Here, $\delta_{\ac}(\s) := \act(\s) -\ac(\s)$, 
		$\delta_{c,\kappa^2}(\s) := \kappat^2(\s) - c\kappa^2(\s)$,  
		$\delta_{\kappa^2}(\s) := \delta_{1,\kappa^2}(\s)$, 
		$\onormal$ is the outward pointing normal on $\partial\cD$, 
		and ``b.c.'' stands for ``boundary conditions''\!.}
	\label{tab:summary}
	\begin{tabular}{@{}lcccc@{}}
		\hline\rule{0pt}{-\normalbaselineskip} \\[-10pt]  
		& \multicolumn{4}{c}{Interval for $\beta$, 
			assuming that $\beta\notin\{k+\nicefrac{1}{4}:k\in\bbN\}$} \\
		\cline{2-5}\rule{0pt}{-\normalbaselineskip}  \\[-7pt] 
		Conditions for 
		&
		$(\nicefrac{d}{4},\nicefrac{5}{4})$ 
		& 
		$(\nicefrac{5}{4},\nicefrac{9}{4})$ 
		&
		$(\nicefrac{9}{4},\nicefrac{13}{4})$ 
		& 
		$(\nicefrac{13}{4},\infty)$ 
		\\[1mm] \hline\rule{0pt}{-\normalbaselineskip} \\[-7pt]
		Isomorphic 
		& 
		\multicolumn{1}{c|}{
			\multirow{2}{*}{$\beta=\betat$}} 
		& 
		\multicolumn{1}{c|}{$\beta = \betat$}  
		& 
		\multicolumn{2}{c}{$\beta = \betat$} 
		\\
		Cameron--Martin spaces 
		& 
		\multicolumn{1}{c|}{} 
		& 
		\multicolumn{1}{c|}{+ b.c.\ on $\delta_{\ac}$}  
		& 
		\multicolumn{2}{c}{+ b.c.\ on $\delta_{\ac}$ and $\delta_{\kappa^2}$} 
		\\[1mm] \cline{2-5}\rule{0pt}{-\normalbaselineskip}  \\[-7pt] 
		Asymptotically optimal   
		& 
		\multicolumn{2}{c|}{
			$\beta = \betat, \ c \ac=\act$} 
		& 
		\multicolumn{1}{c|}{$\beta = \betat, \ c\ac=\act,$}    
		& 
		$\beta = \betat, \ c\ac=\act\;$ 
		\\
		linear prediction
		& 
		\multicolumn{2}{c|}{for some $c\in(0,\infty)$} 
		& 
		\multicolumn{1}{c|}{$\bigl(\ac\nabla\delta_{c,\kappa^2}\bigr)\big|_{\partial\cD}\cdot\onormal = 0$}    
		& 
		+ b.c.\ on $\delta_{c,\kappa^2}$ 
		\\[1mm] \cline{2-5}\rule{0pt}{-\normalbaselineskip}  \\[-7pt] 
		Equivalence of measures
		& 
		\multicolumn{2}{c|}{
			$\beta = \betat, \ \ac=\act$} 
		& 
		\multicolumn{1}{c|}{$\beta = \betat, \ \ac=\act,$}    
		& 
		$\beta = \betat, \ \ac=\act\;$ 
		\\
		in dimension $d\leq 3$ 
		& 
		\multicolumn{2}{c|}{} 
		& 
		\multicolumn{1}{c|}{$\bigl(\ac\nabla\delta_{\kappa^2}\bigr)\big|_{\partial\cD}\cdot\onormal = 0$}    
		& 
		+ b.c.\ on $\delta_{\kappa^2}$ 
		\\[1mm] \cline{2-5}\rule{0pt}{-\normalbaselineskip}  \\[-7pt]  
		Equivalence of measures
		& 
		\multicolumn{4}{c}{
			\multirow{2}{*}{$\beta = \betat, \quad \ac=\act, \quad \kappa^2 = \kappat^2$}}
		\\
		in dimension $d\geq 4$ 
		& 
		& 
		& 
		& 
		\\ 
		\hline
	\end{tabular}
\end{table}

\subsection{Cameron--Martin spaces}
\label{subsec:wm-D:CM}

We first characterize the  
function spaces 
which in the context 
of Whittle--Mat\'ern fields 
act as  
Cameron--Martin spaces. 
The next result is a generalization of 
\citep[][Lemma~3.1]{thomee2006} 
(where $L=-\Delta$ and $r \in \bbN_0$).  

\begin{lemma}\label{lem:Hdot-Sobolev} 
	Suppose that Assumptions~\ref{ass:a-kappa-D}.I--III 
	are satisfied and 
	let $\Hdot{r}$ be defined 
	according to \eqref{eq:def:hdotA} 
	with $E=L_2(\cD)$ and $L$ 
	as in~\eqref{eq:L-div}.   
	Then, for every 
	$r\in\bbR_+$, the space 
	$\Hdot{r}$ is a subspace 
	of~$H^r(\cD)$ 
	and 
	$\bigl( \Hdot{r}, \norm{\,\cdot\,}{r,L} \bigr) 
	\hookrightarrow 
	\bigl( H^{r}(\cD), \norm{\,\cdot\,}{H^r(\cD)} \bigr)$. 
	Furthermore, 
	for every $r\in\bbR_+\setminus\frakE$, 
	where 
	\begin{equation}\label{eq:exception} 
		\frakE 
		:= 
		\{2k + \nicefrac{1}{2} : k \in \bbN_0 \}, 
	\end{equation}
	we have the identification 
	\begin{equation}\label{eq:Hdot-Sobolev} 
		\Hdot{r} 
		= 
		\bigl\{ 
		v\in H^{r}(\cD) : 
		\bigl( \kappa^2 - \nabla \cdot (\ac \nabla) \bigr)^j v = 0 
		\text{ in } 
		L_2(\partial\cD)   
		\ \; \forall 
		j\in \bbN_0 
		\text{ with } 
		j \leq \bigl\lfloor \tfrac{2r-1}{4} \bigr\rfloor 
		\bigr\}, 
	\end{equation} 
	and 
	on the space $\Hdot{r}$ the norm 
	$\norm{\,\cdot\,}{r,L}$ is equivalent to the 
	Sobolev norm $\norm{\,\cdot\,}{H^r(\cD)}$. 
\end{lemma}  

\begin{remark} 
	The coefficients $\ac,\kappa$ of the second-order 
	differential operator $L$ in \eqref{eq:L-div} enter 
	the characterization  \eqref{eq:Hdot-Sobolev} 
	of the Hilbert space $\Hdot{r}$ 
	(that is, the domain of the operator $L^{\nicefrac{r}{2}}$)  
	only for $r > \nicefrac{5}{2}$. 
	In this case, it is solely the behavior of $\ac$ and $\kappa$ 
	at the boundary~$\partial\cD$ that determines $\Hdot{r}$\!. 
	If $r\in\frakE$ belongs to the exception set, 
	then the norm~$\norm{\,\cdot\,}{r,L}$ generates 
	a strictly finer topology than the Sobolev 
	norm~$\norm{\,\cdot\,}{H^r(\cD)}$,  
	cf.~\cite[][Theorem~11.7 in Chapter~1]{LionsMagenesI}. 
	We discuss this further in Section~\ref{section:discussion}. 
\end{remark}   

\begin{proof}[Proof of Lemma~\ref{lem:Hdot-Sobolev}] 
	We recall that by the divergence theorem, 
	Theorem~\ref{thm:div-thm} in Appendix~\ref{appendix:function-spaces}, 
	we have 
	\begin{equation}\label{eq:div-identity} 
		\begin{split} 
			( v_1, (\kappa^2 - \nabla\cdot(\ac \nabla)) v_2 &)_{L_2(\cD)} 
			- 
			\scalar{ (\kappa^2 - \nabla\cdot(\ac \nabla)) v_1, v_2}{L_2(\cD)} 
			\\
			&= 
			\int_{\partial\cD} 
			\bigl[ v_2 (\ac\nabla v_1 \cdot \onormal) - v_1 (\ac\nabla v_2 \cdot \onormal) \bigr] 
			\, \rd S 
			\quad\;\; 
			\forall 
			v_1, v_2 \in H^2(\cD), 
		\end{split} 
	\end{equation} 
	where $\rd S$ is the $(d-1)$-dimensional 
	surface measure on $\partial\cD$ 
	and $\onormal\from\partial\cD\to\bbR^d$ is 
	the outward pointing unit normal vector field, 
	see also 
	Subsection~\ref{app:subsubsec:Lp} 
	and Remark~\ref{rem:onormal} 
	in Appendix~\ref{appendix:function-spaces}.  
	
	\textbf{Step 1}: $\supseteq$ in \eqref{eq:Hdot-Sobolev}. 
	First, we consider the case $r\in(0,2]$, $r\neq\nicefrac{1}{2}$, 
	in \eqref{eq:Hdot-Sobolev}. 
	The equivalence   
	\begin{equation}\label{eq:Hdot-Sobolev-1-2} 
		\bigl(\Hdot{r}, \norm{\,\cdot\,}{r, L} \bigr) 
		\cong 
		\bigl(H^r(\cD)\cap H^1_0(\cD), \norm{\,\cdot\,}{H^r(\cD)} \bigr), 
		\quad 
		r\in[1,2],  
	\end{equation}
	can be shown in a similar manner      
	as \cite[][Lemma~2]{cox2020}. 
	Moreover, since 
	$\Hdot{0} = L_2(\cD)$ and  $\Hdot{1}=H^1_0(\cD)$, 
	it follows from \cite[][Theorem~8.1]{Grisvard1967} 
	that 
	\begin{align} 
		\bigl(\Hdot{r}, \norm{\,\cdot\,}{r, L} \bigr) 
		&\cong 
		\bigl(\{v\in H^r(\cD) : v=0 \text{ in }L_2(\partial\cD)\}, 
		\norm{\,\cdot\,}{H^r(\cD)} \bigr), 
		&& 
		r\in(\nicefrac{1}{2},1),   
		\label{eq:Hdot-Sobolev-12-1}
		\\
		\bigl(\Hdot{r}, \norm{\,\cdot\,}{r, L} \bigr) 
		&\cong 
		\bigl(H^r(\cD), \norm{\,\cdot\,}{H^r(\cD)} \bigr), 
		&& 
		r\in(0,\nicefrac{1}{2}), 
		\label{eq:Hdot-Sobolev-0-12}
	\end{align} 
	see also \cite[][Theorems~11.5 and 11.6 in Chapter~1]{LionsMagenesI}. 
	
	Now let 
	$r = 2k + r_0$ 
	for some $k\in\bbN$ and 
	$r_0\in(0,\nicefrac{1}{2})$,   
	and assume that $v \in H^{r}(\cD)$ 
	is such that 
	$( \kappa^2 - \nabla \cdot (\ac \nabla) )^j v = 0$ 
	in $L_2(\partial\cD)$ 
	for all $j\in\{0,1,\ldots,k-1\}$. 
	Then, by using the boundary conditions of $v$ 
	and of the eigenfunctions $\{e_j\}_{j\in\bbN}$ 
	in \eqref{eq:div-identity}, 
	we obtain that 
	\begin{equation}\label{eq:sigma-sigma0} 
		\begin{split}  
			\norm{v}{r, L}^2
			&= 
			\sum\limits_{j\in\bbN} \lambda_j^{2k+r_0} \scalar{v, e_j}{L_2(\cD)}^2 
			=
			\sum\limits_{j\in\bbN} 
			\lambda_j^{r_0} 
			\bigl( v, (\kappa^2 - \nabla \cdot (\ac \nabla) )^k e_j \bigr)_{L_2(\cD)}^2 
			\\
			&=
			\sum\limits_{j\in\bbN} 
			\lambda_j^{r_0}
			\bigl( (\kappa^2 - \nabla \cdot (\ac \nabla) )^k v, e_j \bigr)_{L_2(\cD)}^2 
			= 
			\bigl\| (\kappa^2 - \nabla \cdot (\ac \nabla))^k v \bigr\|_{r_0, L}^2. 
		\end{split} 
	\end{equation} 
	By the identification \eqref{eq:Hdot-Sobolev-0-12}, 
	there exist constants $C'\!,C\in\bbR_+$, independent of $v$, such that 
	\begin{equation}\label{eq:estimate-sigma0}
		\bigl\| (\kappa^2 - \nabla \cdot (\ac \nabla))^k v \bigr\|_{r_0,L}^2 
		\leq 
		C'  
		\bigl\| (\kappa^2 - \nabla \cdot (\ac \nabla))^k v \bigr\|_{H^{r_0}(\cD)}^2 
		\leq C 
		\norm{v}{H^{2k+r_0}(\cD)}^2, 
	\end{equation} 
	where we used 
	the regularity of $\kappa\in C^{\infty}(\clos{\cD})$, 
	$\ac\in C^{\infty}(\clos{\cD})^{d\times d}$ 
	in the last step. 
	This shows that $v\in \Hdot{r}$ and  
	$\norm{v}{r,L}\leq C \norm{v}{H^r(\cD)}$, 
	where the constant $C\in\bbR_+$ is 
	independent of $v$. 

	Assume now that $r=2k+r_0$   
	for some $k\in\bbN$ and $r_0 \in (\nicefrac{1}{2}, 2]$, 
	and let $v \in H^r(\cD)$ be such that 
	$( \kappa^2 - \nabla \cdot (\ac \nabla) )^j v = 0$ 
	in $L_2(\partial\cD)$ 
	for all $j\in \{0,1,\ldots,k\}$. 
	Then, 
	as in \eqref{eq:sigma-sigma0}, we obtain  that 
	${\norm{v}{r,L}^2  
	=  
	\bigl\| (\kappa^2-\nabla\cdot(\ac\nabla))^k v \bigr\|_{r_0,L}^2}$. 
	Since by assumption also 
	the trace of ${(\kappa^2 - \nabla \cdot (\ac \nabla) )^k v}$
	vanishes in $L_2(\partial\cD)$, we conclude 
	by the equivalences in \eqref{eq:Hdot-Sobolev-1-2} 
	and \eqref{eq:Hdot-Sobolev-12-1} 
	that the estimates in \eqref{eq:estimate-sigma0}
	also hold in this case, 
	with $C'\!, C\in\bbR_+$ 
	independent of $v$.   
	
	\textbf{Step 2:} $\subseteq$ in \eqref{eq:Hdot-Sobolev}. 
	For the reverse inclusion we show that 
	\textbf{a)} for all $r\in\bbR_+$ 
	and all $v\in \Hdot{r}$, we have that  
	$v \in H^r(\cD)$ with $\norm{v}{H^r(\cD)} \leq C \norm{v}{r,L}$,   
	and \textbf{b)} in the case that 
	$r\notin\frakE$, see \eqref{eq:exception}, 
	every $v\in\Hdot{r}$ also satisfies the boundary conditions 
	in \eqref{eq:Hdot-Sobolev}.  
	For \textbf{a)} we first prove the regularity 
	result ${\bigl( \Hdot{r}, \norm{\,\cdot\,}{r,L}\bigr) 
	\hookrightarrow 
	\bigl( H^r(\cD), \norm{\,\cdot\,}{H^r(\cD)}\bigr)}$
	for all integers 
	$r\in\{ \{ 2k-1, 2k \} : k\in\bbN\}$,  
	via induction with respect to $k\in\bbN$. 
	The cases $r\in\{1,2\}$ (i.e., $k=1$) 
	are part of \eqref{eq:Hdot-Sobolev-1-2}. 
	
	For the induction step $k-1\to k$, let $k\geq 2$ 
	and 
	$v\in \Hdot{2k-1} = \scrD\bigl( L^{k-\nicefrac{1}{2}} \bigr)$. 
	Then, there exists  
	$\psi\in L_2(\cD)$ such that  
	$v = L^{-(k-\nicefrac{1}{2})} \psi$ 
	and 
	$\widetilde{v} := L^{-(k-\nicefrac{3}{2})} \psi$ 
	satisfies 
	${\widetilde{v}\in\scrD\bigl( L^{k-\nicefrac{3}{2}} \bigr) = \Hdot{2k-3}}$.  
	Thus,  
	$L v = \widetilde{v} \in H^{2k-3}(\cD)$ 
	follows from the induction hypothesis, 
	and there exists a 
	constant $C'\!\in\bbR_+$, 
	which is independent of $v\in\Hdot{2k-1}$, 
	such that 
	$\norm{Lv}{H^{2k-3}(\cD)} 
	\leq C' \norm{Lv}{2k-3,L} 
	= C' \norm{v}{2k-1,L}$. 
	As $v\in \Hdot{2k-1}\subset \scrD(L)$, 
	this regularity of 
	$Lv \in H^{2k-3}(\cD)$ 
	implies by Theorem~\ref{thm:regularity}
	that 
	$v \in H^{2k-1}(\cD)$,  
	\begin{align*} 
		\norm{v}{H^{2k-1}(\cD)} 
		&\leq 
		\widehat{C} \, 
		\bigl( 
		\norm{L v}{H^{2k-3}(\cD)}  
		+ 
		\norm{v}{H^{2k-2}(\cD)}
		\bigr) 
		\\
		&\leq 
		\widehat{C} 
		\left( 
		C' 
		\norm{v}{2k-1,L} 
		+ 
		C'' 
		\norm{v}{2k-2,L}
		\right) 
		\leq 
		C 
		\norm{v}{2k-1,L},  
	\end{align*} 
	where all constants are independent of $v$. 
	Here, we also used that 
	by the induction hypothesis 
	$\bigl( \Hdot{2k-2}, \norm{\,\cdot\,}{2k-2,L}\bigr) 
	\hookrightarrow 
	\bigl( H^{2k-2}(\cD), \norm{\,\cdot\,}{H^{2k-2}(\cD)}\bigr)$
	holds. 
	Suppose now that 
	$v \in \Hdot{2k} = \scrD\bigl( L^k \bigr)$. 
	Then, similarly as above, 
	we obtain from the induction hypothesis 
	that 
	$Lv \in H^{2k-2}(\cD)$ 
	with 
	$\norm{Lv}{H^{2k-2}(\cD)} 
	\leq C' \norm{v}{2k,L}$ 
	and, again by 
	Theorem~\ref{thm:regularity}, 
	the regularity  
	$v \in H^{2k}(\cD)$ 
	follows, with 
	\[
		\norm{v}{H^{2k}(\cD)} 
		\leq 
		\widehat{C} \, 
		\bigl( 
		\norm{L v}{H^{2k-2}(\cD)}  
		+ 
		\norm{v}{H^{2k-1}(\cD)}
		\bigr) 
		\leq 
		C 
		\norm{v}{2k,L}.  
	\]

	By means of complexification and interpolation arguments 
	(see Lemma~\ref{lem:hdot-interpol} 
	in Appendix~\ref{appendix:interpol},   
	\cite[][Theorem~1 in Section~4.3.1]{Triebel1978} 
	and \cite[][Theorem~2.6]{Lunardi2018})
	we subsequently obtain the continuous embedding  
	${\bigl( \Hdot{r}, \norm{\,\cdot\,}{r,L}\bigr) 
	\hookrightarrow 
	\bigl( H^r(\cD), \norm{\,\cdot\,}{H^r(\cD)}\bigr)}$ 
	for the whole range $r\in\bbR_+$. 
	
	\textbf{Step 2b)} Finally, 
	it can also be shown via induction with respect to 
	$k\in\bbN_0$ 
	that 
	\begin{align*} 
		&
		\forall r\in(2k,2k+\nicefrac{1}{2}), 
		&
		\forall v\in \Hdot{r}:
		&& 
		\bigl( \kappa^2 - \nabla \cdot (\ac\nabla) \bigr)^j v 
		&= 
		0
		\;\ \text{in}\;\; 
		L_2(\partial\cD), 
		\quad  
		0\leq j \leq k-1, 
		\\
		&
		\forall r\in(2k+\nicefrac{1}{2},2k+2], 
		\hspace{-0.1cm} 
		& 
		\forall v\in \Hdot{r}: 
		&&
		\bigl( \kappa^2 - \nabla \cdot (\ac\nabla) \bigr)^j v 
		&= 
		0
		\;\ \text{in}\;\;  
		L_2(\partial\cD), 
		\quad  
		0\leq j \leq k.  
	\end{align*}
	Specifically, 
	the case $k=0$ is part of 
	\eqref{eq:Hdot-Sobolev-1-2}, 
	\eqref{eq:Hdot-Sobolev-12-1} 
	and \eqref{eq:Hdot-Sobolev-0-12}. 
	For the induction step $k-1\to k$, 
	let $k\in\bbN$, and $v_1\in\Hdot{r_1}$, 
	$v_2\in\Hdot{r_2}$, 
	where $r_1 \in (2k,2k+\nicefrac{1}{2})$ 
	and $r_2 \in (2k+\nicefrac{1}{2},2k+2]$. 
	As we have already proven, Sobolev regularity follows: 
	$v_1\in H^{r_1}(\cD)$ and 
	$v_2\in H^{r_2}(\cD)$.  
	Since $r_1 > 2k$ and 
	${r_2 > 2k + \nicefrac{1}{2}}$, 
	the trace theorem,   
	Theorem~\ref{thm:trace} in 
	Appendix~\ref{appendix:function-spaces}, 
	guarantees that the traces are well-defined,  
	${(\kappa^2 - \nabla\cdot(\ac\nabla))^{j_1} v_1 \in L_2(\partial\cD)}$  
	 and 
	$(\kappa^2 - \nabla\cdot(\ac\nabla))^{j_2} v_2 \in L_2(\partial\cD)$ 
	for all $j_1 \in \{0,1,\ldots,k-1\}$ and $j_2 \in \{0,1,\ldots,k\}$, respectively.
	Furthermore, the induction hypothesis implies that 
	$Lv_1 \in \Hdot{r_1-2}$ and  
	$Lv_2 \in \Hdot{r_2-2}$ satisfy the boundary conditions 
	\begin{align*} 
		(\kappa^2 - \nabla\cdot(\ac\nabla))^{j_1} v_1
		= (\kappa^2 - \nabla\cdot(\ac\nabla))^{j_1-1} (Lv_1) 
		&= 0 
		\;\ \text{in}\;\; 
		L_2(\partial\cD), 
		\quad 1\leq j_1 \leq k-1 , 
		\\
		(\kappa^2 - \nabla\cdot(\ac\nabla))^{j_2} v_2
		= 
		(\kappa^2 - \nabla\cdot(\ac\nabla))^{j_2-1} (Lv_2) 
		&= 0 
		\;\ \text{in}\;\; 
		L_2(\partial\cD), 
		\quad 1\leq j_2 \leq k. 
	\end{align*} 
	Since $v_1, v_2 \in \Hdot{2} = H^2(\cD)\cap H^1_0(\cD)$, 
	we obtain that also 
	$v_1=v_2=0$ in $L_2(\partial\cD)$. 
\end{proof} 

Now we are ready to characterize the 
Cameron--Martin space for 
$\mu_d(m;\beta,\ac,\kappa)$ 
in~\eqref{eq:def:mu}. 

\begin{proposition}\label{prop:CM} 
	Let $d\in\bbN$, 
	$\beta\in(\nicefrac{d}{4},\infty)$, $m\in L_2(\cD)$ 
	and suppose 
	Assumptions~\ref{ass:a-kappa-D}.I--III. 
	Then,~the Cameron--Martin space 
	of the Gaussian measure 
	$\mu_d(m;\beta,\ac,\kappa)$ 
	in~\eqref{eq:def:mu} 
	with covariance operator $\cC=L^{-2\beta}$ 
	is given by 
	$\cC^{\nicefrac{1}{2}}(L_2(\cD)) = \Hdot{2\beta}$\!, 
	cf.\ \eqref{eq:def:hdotA}, 
	and 
	it is continuously embedded in 
	$H^{2\beta}(\cD)$.   
	
	In the case that 
	$2\beta\notin\frakE$, 
	with $\frakE$ as given in \eqref{eq:exception}, 
	it can be identified as in 
	\eqref{eq:Hdot-Sobolev} 
	and there 
	exist constants $c_0,c_1,\dual{c}_0,\dual{c}_1>0$, 
	depending on $\beta,\ac,\kappa,\cD$, 
	such that 
	\begin{align} 
		c_0 \norm{v}{H^{2\beta}(\cD)}^2 
		&\leq 
		\bigl( \cC^{-\nicefrac{1}{2}} v, \cC^{-\nicefrac{1}{2}}v \bigr)_{L_2(\cD)} 
		\leq 
		c_1 \norm{v}{H^{2\beta}(\cD)}^2  
		&& 
		\forall 
		v \in \Hdot{2\beta} 
		= 
		\cC^{\nicefrac{1}{2}}(L_2(\cD)), 
		\label{eq:prop:CM-Cinv}
		\\
		\dual{c}_0 \norm{v}{H^{-2\beta}(\cD)}^2 
		&\leq 
		\;\;\;
		\bigl( \cC^{\nicefrac{1}{2}} v, \cC^{\nicefrac{1}{2}} v \bigr)_{L_2(\cD)}
		\;\;\;\,
		\leq 
		\dual{c}_1 \norm{v}{H^{-2\beta}(\cD)}^2  
		&& 
		\forall 
		v \in \Hdot{-2\beta} =  
		\cC^{-\nicefrac{1}{2}}(L_2(\cD)). 
		\label{eq:prop:CM-C}
	\end{align}  
\end{proposition} 

\begin{proof} 
	That the Cameron--Martin space 
	is given by $\Hdot{2\beta}$ 
	has already been observed 
	in~\eqref{eq:general-CM-space}. 
	Furthermore, whenever 
	$2\beta\notin\frakE = \{ 2k + \nicefrac{1}{2} : k\in\bbN_0 \}$, 
	we obtain  
	\eqref{eq:prop:CM-Cinv} 
	from Lemma~\ref{lem:Hdot-Sobolev} 
	which also implies the norm 
	equivalence \eqref{eq:prop:CM-C} 
	on $\Hdot{-2\beta}$ 
	as this is the dual space of $\Hdot{2\beta}$. 
\end{proof} 

\begin{remark}\label{rem:continuity} 
	Proposition~\ref{prop:CM}  
	shows that under 
	Assumptions~\ref{ass:a-kappa-D}.I--III
	the Cameron--Martin space 
	for the Gaussian measure 
	$\mu_d(m;\beta,\ac,\kappa)$ 
	in \eqref{eq:def:mu} 
	with $\beta \in (\nicefrac{d}{4},\infty)$ 
	is 
	\[ 
		\cC^{\nicefrac{1}{2}}(L_2(\cD)) 
		= 
		\Hdot{2\beta} 
		\hookrightarrow 
		H^{2\beta}(\cD) 
		\hookrightarrow 
		C^0(\clos{\cD})  
	, 
	\] 
	where the last relation 
	is one of the Sobolev embeddings, 
	see e.g.\ \cite[Theorem~4.6.1.(e)]{Triebel1978}. 
	In particular, the random field 
	$\GP\from\clos{\cD}\times\Omega\to\bbR$ 
	which solves the SPDE~\eqref{eq:Lbeta}
	for $(\beta,\ac,\kappa)$ 
	is continuous ($\bbP$-almost surely and in $L_p$-sense 
	for any $p\in[1,\infty)$)
	and its covariance kernel $\varrho$ 
	is continuous on $\clos{\cD\times\cD}$. 
\end{remark} 

An important consequence of Proposition \ref{prop:CM}  and  
Lemma~\ref{lem:iff-A}\ref{lem:iff-A-iso} is the following result 
on equivalence of  
Cameron--Martin spaces for 
Gaussian measures 
defined as in \eqref{eq:def:mu} with  
different parameters. 

\begin{theorem}\label{thm:CM} 
	Suppose Assumption~\ref{ass:a-kappa-D}.III 
	and that each of the 
	parameter tuples   
	$(\ac,\kappa)$,   
	$(\act,\kappat)$
	fulfills Assumptions~\ref{ass:a-kappa-D}.I--II. 
	Let $\beta\in\bbR_{+}$ be such that 
	$2\beta\notin\frakE$, 
	with $\frakE$ as in~\eqref{eq:exception}, 
	and let $L, \Lt$ be defined 
	as in \eqref{eq:L-div}
	with coefficients $\ac, \kappa$ 
	and $\act, \kappat$, 
	respectively. 
	Then, for all $\gamma\in[-\beta, \beta]$, the  
	operator $\Lt^\gamma L^{-\gamma}$ 
	is an isomorphism on $L_2(\cD)$ 
	(and, thus, $\Hdot{2\gamma}\!, \hdot{2\gamma}{\Lt}$ 
	are norm equivalent spaces)  
	if and only if, 
	for all $j\in\bbN_0$ with 
	$j \leq \lfloor \beta - \nicefrac{5}{4}\rfloor$, 
	the following hold:  
	\begin{equation}\label{eq:condition-iso}
		\begin{split} 
			\forall v\in\Hdot{2\beta}:
			\quad 
			\bigl( \kappa^2 - \nabla \cdot (\ac \nabla) \bigr)^j 
			\bigl( \delta_{\kappa^2} - \nabla\cdot (\delta_{\ac} \nabla) \bigr) v 
			&= 
			0 
			\;\ \text{in}\;\;  
			L_2(\partial\cD), 
			\\
			\forall \widetilde{v}\in\hdot{2\beta}{\Lt}: 
			\quad 
			\bigl( \kappat^2 - \nabla \cdot (\act \nabla) \bigr)^j 
			\bigl( \delta_{\kappa^2} - \nabla\cdot (\delta_{\ac} \nabla) \bigr) \widetilde{v}  
			&= 
			0 
			\;\ \text{in}\;\; 
			L_2(\partial\cD). 
		\end{split} 
	\end{equation}
	Here, 
	we set $\delta_{\kappa^2}(\s) := \kappat^2(\s) - \kappa^2(\s)$ 
	and $\delta_{\ac}(\s) := \act(\s) - \ac(\s)$ for all $\s\in\clos{\cD}$. 
	
	Furthermore, 
	the Cameron--Martin spaces 
	of two Gaussian measures 
	$\mu_d(0;\beta,\ac,\kappa)$,  
	$\mu_d(0;\betat,\act,\kappat)$, 
	defined according to~\eqref{eq:def:mu} 
	with   
	$\beta,\betat \in (\nicefrac{d}{4},\infty)$,  
	where $d\in\bbN$ and 
	$2\beta\notin\frakE$,   
	are isomorphic  
	with equivalent norms  
	if and only if 
	$\beta = \betat$
	and \eqref{eq:condition-iso} 
	holds for all 
	$j\in\bbN_0$ with 
	$j \leq \lfloor \beta - \nicefrac{5}{4}\rfloor$. 
\end{theorem} 

\begin{proof} 
	In order to derive the first assertion, 
	we distinguish 
	two cases,  
	\textbf{Case I:} $\beta\in(0,1)$, $\beta\neq\nicefrac{1}{4}$ 
	and 
	\textbf{Case II:} $\beta\in[1,\infty)$, $2\beta\notin\frakE$. 
	
	In \textbf{Case I}, 
	$\beta\in(0,1)$, $\beta\neq\nicefrac{1}{4}$, 
	there are no conditions imposed in \eqref{eq:condition-iso} 
	and we obtain the relation    
	$\bigl(\Hdot{2\beta}, \norm{\,\cdot\,}{2\beta,L}\bigr) 
	\cong 
	\bigl(\hdot{2\beta}{\Lt}, \norm{\,\cdot\,}{2\beta,\Lt}\bigr)$ 
	from one of the identifications 
	in \eqref{eq:Hdot-Sobolev-1-2}, \eqref{eq:Hdot-Sobolev-12-1}  
	or \eqref{eq:Hdot-Sobolev-0-12}. 
	Consequently, $\Lt^\beta L^{-\beta}$ 
	is an isomorphism on $L_2(\cD)$ 
	and by complexification and interpolation, 
	see Lemma~\ref{lem:hdot-interpol} in Appendix~\ref{appendix:interpol}, 
	the same is true for $\Lt^\gamma L^{-\gamma}$ 
	and all $\gamma\in[-\beta,\beta]$. 
	
	\textbf{Case II:} 
	For $\beta\in[1,\infty)$, 
	Lemma~\ref{lem:iff-A}\ref{lem:iff-A-iso} 
	shows that 
	$\Lt^\gamma L^{-\gamma}$ is an isomorphism 
	on $L_2(\cD)$ for every  
	$\gamma\in[-\beta, \beta]$ 
	if and only if 
	$\Lt - L \in \cL\bigl( \Hdot{2\eta}; \Hdot{2(\eta-1)}\bigr) \cap  
	\cL\bigl( \hdot{2\eta}{\Lt}; \hdot{2(\eta-1)}{\Lt} \bigr)$ 
	holds for $\eta\in\{1,\beta\}$. 
	The claim 
	then follows from identifying 
	$\Hdot{2\beta-2}$ and $\hdot{2\beta-2}{\Lt}$ 
	according to~\eqref{eq:Hdot-Sobolev} 
	in Lemma~\ref{lem:Hdot-Sobolev}, 
	combined with 
	the regularity 
	$(\Lt-L)v\in H^{2\eta-2}(\cD)$ 
	which holds  
	for all $v\in\Hdot{2\eta}\cup\hdot{2\eta}{\Lt} \subseteq H^{2\eta}(\cD)$ 
	and every $\eta\in\{1,\beta\}$,  
	since  
	$\kappa,\kappat\in C^\infty(\clos{\cD})$ 
	and 
	$\ac,\act\in C^\infty(\clos{\cD})^{d\times d}$ 
	are smooth. 
	
	We now prove the second claim.   
	By Proposition~\ref{prop:CM} 
	the Cameron--Martin spaces are 
	$\Hdot{2\beta}$ and $\hdot{2\betat}{\Lt}$\!. 
	If we identify the Hilbert space 
	$L_2(\cD)$ with the  
	space $\ell^2$ of square-summable sequences, 
	Weyl's law~\eqref{eq:lambdaj} 
	(applied for 
	$L$ and $\Lt$)
	shows that 
	$\Hdot{2\beta}$ can be identified 
	with 
	${\bigl\{ 
	(c_j)_{j\in\bbN} 
	: \bigl\{ j^{2\beta/d} c_j \bigr\}_{j\in\bbN} \in\ell^2 
	\bigr\} \subset \ell^2}$ 
	and 
	$\hdot{2\betat}{\Lt}$ 
	with 
	$\bigl\{ 
	(\widetilde{c}_j)_{j\in\bbN} 
	: \bigl\{ j^{2\betat/d} \, \widetilde{c}_j \bigr\}_{j\in\bbN} \in\ell^2 
	\bigr\} 
	\subset \ell^2$. 
	For   
	this reason, 
	$\Hdot{2\beta}$ and 
	$\hdot{2\betat}{\Lt}$ can be isomorphic 
	only if $\beta = \betat$. 
	In the case that $\beta=\betat$ and $2\beta\notin\frakE$, 
	sufficiency and necessity of the 
	conditions \eqref{eq:condition-iso} 
	for $\Hdot{2\beta}$ 
	and $\hdot{2\betat}{\Lt} = \hdot{2\beta}{\Lt}$ 
	to be isomorphic follow from 
	the first part of this theorem. 
\end{proof} 

We end this subsection with a discussion 
of the conditions \eqref{eq:condition-iso}. 
In what follows, we suppose that the assumptions of 
Theorem~\ref{thm:CM} on the coefficients of $L,\Lt$ 
and on the domain $\cD\subset\bbR^d$ are satisfied. 
Firstly, we note that for 
$\beta\in(0,\nicefrac{5}{4})$ 
no boundary conditions on 
$\delta_{\kappa^2}$ or $\delta_{\ac}$ are imposed 
and the spaces $\Hdot{2\beta}$ and $\hdot{2\beta}{\Lt}$
are isomorphic, independently of the choice of $\kappa,\kappat,\ac,\act$. 
Next, consider the case that 
$\beta\in(\nicefrac{5}{4},\nicefrac{9}{4})$.  
Then, the conditions 
\eqref{eq:condition-iso}  
say that 
\[
	\bigl( \delta_{\kappa^2} - \nabla\cdot(\delta_{\ac} \nabla) \bigr) v 
	= 
	\bigl( \kappat^2 - \nabla\cdot( \act \nabla) \bigr) v 
	- 
	\bigl( \kappa^2 - \nabla\cdot( \ac \nabla) \bigr) v 
	= 0
	\;\ \text{in}\;\; L_2(\partial\cD) 
\]
has to hold 
for every  $v \in \Hdot{2\beta} \cup \hdot{2\beta}{\Lt}$\!. 
By \eqref{eq:Hdot-Sobolev}, 
for all $\beta\in(\nicefrac{5}{4},\infty)$, 
every $v\in\Hdot{2\beta} \cup \hdot{2\beta}{\Lt}$ 
satisfies the boundary condition  
$v=0$  
in $L_2(\partial\cD)$. 
Therefore, in this case  
 \eqref{eq:condition-iso} 
simplifies to the requirement 
that 
$\nabla\cdot( \delta_{\ac} \nabla v) =0$ 
in $L_2(\partial\cD)$ 
for all $v\in \Hdot{2\beta} \cup \hdot{2\beta}{\Lt}$\!. 
In particular, note that no assumptions are imposed on 
$\kappa,\kappat$. 
Finally, we consider the case that  
$c\ac=\act$ for some $c\in\bbR_+$ 
and $\beta\in(\nicefrac{9}{4}, \nicefrac{13}{4})$. 
Since $\Hdot{2\beta}\cong\hdot{2\beta}{\Lt}$ holds 
if and only if $\hdot{2\beta}{cL}\cong\hdot{2\beta}{\Lt}$\!, 
we thus need that for all 
$v\in\Hdot{2\beta}=\hdot{2\beta}{cL}$ and 
$\widetilde{v}\in\hdot{2\beta}{\Lt}$: 
\begin{equation}\label{eq:cond-delta-kappa} 
	\bigl( \kappa^2 - \nabla \cdot (\ac \nabla) \bigr) 
	\bigl( \delta_{c,\kappa^2} v\bigr) 
	=
	0
	\;\ \text{in} \;\; L_2(\partial\cD) 
	\quad 
	\text{and}  
	\quad 
	\bigl( \kappat^2 - c \nabla \cdot (\ac \nabla) \bigr)  
	\bigl( \delta_{c,\kappa^2} \widetilde{v} \bigr) 
	=
	0
	\;\ \text{in} \;\; L_2(\partial\cD), 
\end{equation}
where $\delta_{c,\kappa^2}(\s) := \kappat^2(\s) - c\kappa^2(\s)$. 
Since $( \kappa^2 - \nabla \cdot (\ac \nabla) )v = v = 0$ 
in $L_2(\partial\cD)$,  
this gives  
\begin{align*} 
	0 
	= 
	\bigl( \kappa^2 - \nabla \cdot (\ac \nabla) \bigr) 
	\bigl( \delta_{c,\kappa^2} v\bigr) 
	&= 
	\delta_{c,\kappa^2} \bigl( \kappa^2 - \nabla \cdot (\ac \nabla) \bigr) v 
	- 
	2 ( \ac \nabla v ) \cdot \nabla \delta_{c,\kappa^2}
	- 
	v \nabla \cdot (\ac \nabla \delta_{c,\kappa^2}  ) 
	\\
	&= 
	- 2 ( \ac \nabla v ) \cdot \nabla \delta_{c,\kappa^2} 
	\;\  \text{in} \;\; L_2(\partial\cD), 
\end{align*} 
for all $v\in\Hdot{2\beta}$ 
and, similarly, 
$( \ac \nabla \widetilde{v} ) \cdot \nabla \delta_{c,\kappa^2} 
= 0$ in $L_2(\partial\cD)$ follows for all $\widetilde{v}\in\hdot{2\beta}{\Lt}$\!. 
The traces of $v, \widetilde{v}$ vanish in $L_2(\partial\cD)$ 
and $\partial\cD$ is smooth. 
Therefore, also the traces of all tangential components of 
$\nabla v, \nabla\widetilde{v}$ 
vanish and  
${\nabla v =  \tfrac{\partial v }{\partial\onormal} \onormal}$, 
${\nabla \widetilde{v} =  \frac{\partial \widetilde{v}}{\partial\onormal} \onormal}$ 
with equality in $L_2\bigl( \partial\cD;\bbR^d \bigr)$, 
where  
$\onormal$ 
is the outward pointing unit normal on~$\partial\cD$, 
see Remark~\ref{rem:onormal} in Appendix~\ref{appendix:function-spaces}. 
For $\beta\in(\nicefrac{9}{4}, \nicefrac{13}{4})$ and 
${r \in(\nicefrac{3}{2}, 2)}$, 
$\Hdot{2\beta}\!, \hdot{2\beta}{\Lt}$ 
are dense in $H^r(\cD)\cap H^1_0(\cD)$ 
and the trace map  
$v \mapsto \bigl\{ \tfrac{\partial^j v}{\partial \onormal^j} : j = 0,1 \bigr\}$
of 
$H^r(\cD) \to H^{r-\nicefrac{1}{2}}(\partial\cD) \times H^{r-\nicefrac{3}{2}} (\partial\cD)$ 
is surjective, 
see~Theorem~\ref{thm:trace} in Appendix~\ref{appendix:function-spaces}. 
Since also $H^{r-\nicefrac{3}{2}} (\partial\cD)$ 
is dense in $L_2(\partial\cD)$, 
the requirement 
\eqref{eq:cond-delta-kappa} simplifies 
to the following condition on 
$\delta_{c,\kappa^2}=\kappat^2 - c \kappa^2$:  
\begin{align}
	\hspace{1cm}  
	\forall v \in \Hdot{2\beta} : 
	&& 
	(\ac\nabla v) \cdot \nabla\delta_{c,\kappa^2} 
	=
	\tfrac{\partial v}{\partial \onormal} (\ac\onormal) \cdot \nabla\delta_{c,\kappa^2} 
	&= 
	0
	\;\ \text{in} \;\; L_2(\partial\cD) 
	\hspace{2cm}  
	\notag 
	\\
	\Longleftrightarrow\qquad 
	&&
	(\ac  \nabla \delta_{c,\kappa^2}) \cdot \onormal 
	&= 
	0 
	\;\ \text{on} \;\; \partial\cD. 
	\qquad 
	\label{eq:condition-delta-kappa} 
\end{align}

\subsection{Equivalence and orthogonality of Whittle--Mat\'ern measures}
\label{subsec:wm-D:equiv}

The main outcomes of this section 
are necessary and sufficient conditions 
on the parameters involved 
for two Gaussian measures 
$\mu_d(m;\beta,\ac,\kappa)$ 
and 
$\mu_d(\mt;\betat,\act,\kappat)$, 
defined according to \eqref{eq:def:mu}, 
to be equivalent, see Theorem~\ref{thm:equivalence}. 
In order to derive this result, we first formulate 
three lemmas which will guarantee 
sufficiency (Lemma~\ref{lem:HS:d<4}) 
and necessity (Lemmas~\ref{lem:a-notcompact} and~\ref{lem:kappa-notHS}) 
of the conditions.   

\begin{lemma}\label{lem:HS:d<4} 
	Let $d\in\{1,2,3\}$ and let  
	$\beta \in (\nicefrac{d}{4},\infty)$ be such that 
	$2\beta\notin\frakE$, 
	with $\frakE$ as given in~\eqref{eq:exception}. 
	In addition, suppose Assumption~\ref{ass:a-kappa-D}.III 
	and 
	let the operators~$L$ and~$\Lt$ be defined as 
	in \eqref{eq:L-div} 
	with coefficients $\ac,\kappa$ 
	and $\ac,\kappat$, respectively, where 
	$\ac$ fulfills Assumption~\ref{ass:a-kappa-D}.I 
	and $\kappa,\kappat$ are such that  
	Assumption~\ref{ass:a-kappa-D}.II 
	is satisfied and 
	\eqref{eq:condition-iso} holds for all $j\in\bbN_0$ with  
	$j\leq\lfloor\beta-\nicefrac{5}{4}\rfloor$. 
	Then, the operator 
	$\Lt^\beta L^{-\beta}$ is an isomorphism on $L_2(\cD)$ 
	and 
	$L^{-\beta} \Lt^{2\beta} L^{-\beta} - \id_{L_2(\cD)}$  
	is Hilbert--Schmidt  on $L_2(\cD)$. 
\end{lemma} 

\begin{proof} 
	Firstly, we note that by Theorem~\ref{thm:CM} 
	the operator 
	$\Lt^\gamma L^{-\gamma}$ is an isomorphism on $L_2(\cD)$ 
	for all $\gamma\in[-\beta,\beta]$. 
	To prove the Hilbert--Schmidt property 
	of 
	$L^{-\beta} \Lt^{2\beta} L^{-\beta} - \id_{L_2(\cD)}$,  
	we distinguish  
	\textbf{Case~I:} $\beta\in(\nicefrac{d}{4},1)$ 
	and 
	\textbf{Case II:} $\beta\in[1,\infty)$, $2\beta\notin\frakE$. 
	
	\textbf{Case I:} For $\beta\in(\nicefrac{d}{4},1)$, 
	we first observe the identity 
	\begin{align*} 
		\Lt^{2\beta} 
		- 
		L^{2\beta}  
		&=
		\tfrac{1}{2} 
		\bigl(\Lt^{\beta} + L^\beta \bigr) 
		\bigl(\Lt^{\beta} - L^\beta \bigr)  
		+ 
		\tfrac{1}{2} 
		\bigl(\Lt^{\beta} - L^\beta \bigr) 
		\bigl(\Lt^{\beta} + L^\beta \bigr) 
		\\
		&=
		\tfrac{1}{2} 
		\bigl(\Lt^{\beta} + L^\beta \bigr) 
		\bigl(\Lt^{\beta} - L^\beta \bigr)  
		+ 
		\tfrac{1}{2} 
		\bigl[ 
		\bigl(\Lt^{\beta} + L^\beta \bigr) 
		\bigl(\Lt^{\beta} - L^\beta \bigr)  
		\bigr]^*\!. 
	\end{align*} 
	Since for $S\in\cL_2(E)$ we have $S^*\in\cL_2(E)$ 
	with $\norm{S^*}{\cL_2(E)} = \norm{S}{\cL_2(E)}$, 
	we estimate 
	\begin{align*} 
		\bigl\| 
		L^{-\beta} \bigl( \Lt^{2\beta} - L^{2\beta} \bigr) L^{-\beta} 
		\bigr\|_{\cL_2(L_2(\cD))} 
		&\leq 
		\bigl\| 
			L^{-\beta}
			\bigl( \Lt^{\beta} + L^\beta \bigr) 
			\bigl( \Lt^{\beta} - L^\beta \bigr) 
			L^{-\beta}
		\bigr\|_{\cL_2(L_2(\cD))} 
		\\
		&\leq 
		\bigl( 
		\bigl\|  
			L^{-\beta}\Lt^{\beta} \bigr\|_{\cL(L_2(\cD))} 
		+ 1 
		\bigr) 
		\bigl\| 
			\bigl(\Lt^\beta - L^\beta\bigr) L^{-\beta} 
		\bigr\|_{\cL_2(L_2(\cD))}. 
	\end{align*} 
	By the isomorphism property of $\Lt^\beta L^{-\beta}$\!,  
	the operator $L^{-\beta}\Lt^{\beta}$ is bounded 
	on $L_2(\cD)$. 
	Furthermore, since $(\Lt-L)\psi = \delta_{\kappa^2}\psi$ 
	and $\delta_{\kappa^2} := \kappat^2-\kappa^2 \in C^\infty(\clos{\cD})$, 
	we find that $\Lt-L\in\cL( L_2(\cD) )$.  
	Thus, by Lemma~\ref{lem:A-alpha-difference} 
	and Remark~\ref{rem:A-alpha-difference} 
	in Appendix~\ref{appendix:A-alpha}, 
	also $\Lt^\beta-L^\beta\in\cL(L_2(\cD))$,  
	and 
	\[
		\bigl\|
			\bigl( \Lt^\beta - L^\beta \bigr) L^{-\beta} 
		\bigr\|_{\cL_2(L_2(\cD))} 
		\leq 
		\bigl\|
			\Lt^\beta - L^\beta 
		\bigr\|_{\cL(L_2(\cD))} 
		\bigl\| 
			L^{-\beta} 
		\bigr\|_{\cL_2(L_2(\cD))} 
		<\infty.  
	\]
	Here, the Hilbert--Schmidt property of 
	$L^{-\beta} \in \cL_2(L_2(\cD))$ 
	for $\beta\in(\nicefrac{d}{4},1)$ follows from  
	the spectral 
	asymptotics \eqref{eq:lambdaj} of the operator $L$ 
	since, for any $\eps\in\bbR_+$,  
	\begin{equation}\label{eq:weyl-HS}   
		\bigl\| L^{- (\nicefrac{d}{4}+\eps) } \bigr\|_{\cL_2(L_2(\cD))}^2 
		= 
		\sum\limits_{j\in\bbN} \lambda_j^{-\nicefrac{d}{2} - 2\eps}
		\leq 
		c_\lambda^{-\nicefrac{d}{2} - 2\eps} 
		\sum\limits_{j\in\bbN} j^{- 1 - \nicefrac{(4\eps)}{d}} 
		< \infty. 
	\end{equation}
	
	\textbf{Case II:}
	Let $\frakN_\beta$ be as in~\eqref{eq:frakNset}
	and  
	$\eta\in\frakN_\beta$. 
	Pick $\eps_0 \in (0,\nicefrac{1}{2})$  
	such that 
	$2\eta-\nicefrac{d}{2}-\eps_0\notin\frakE$ 
	holds for all $\eta\in\frakN_\beta$. 
	Then, by Lemma~\ref{lem:Hdot-Sobolev},    
	on $\Hdot{2\eta-\nicefrac{d}{2}-\eps_0}$ 
	the norm $\norm{\,\cdot\,}{2\eta-\nicefrac{d}{2}-\eps_0, L}$ 
	is equivalent to the 
	Sobolev norm $\norm{\,\cdot\,}{H^{2\eta-d/2-\eps_0}(\cD)}$. 
	Furthermore, $\Hdot{2\eta}$ is dense 
	in $\Hdot{2\eta-\nicefrac{d}{2}-\eps_0}$ and 
	for any fixed ${\psi\in\Hdot{2\eta-\nicefrac{d}{2}-\eps_0}}$\!,  
	$\delta\in\bbR_+$ there exists 
	$v_\delta\in\Hdot{2\eta}$ such that  
	$\norm{\psi-v_\delta}{2\eta-\nicefrac{d}{2}-\eps_0, L}< \delta$. 
	As \eqref{eq:condition-iso}   
	is assumed, for every $\eta\in\frakN_\beta$ 
	and all 
	$j\in\bbN_0$ with 
	$j\leq\lfloor\eta-\nicefrac{5}{4}\rfloor$, 
	we have 
	$\bigl( \kappa^2 - \nabla\cdot(\ac\nabla)\bigr)^j \bigl( \delta_{\kappa^2} v \bigr) = 0$ 
	in $L_2(\partial\cD)$ 
	for all $v \in \Hdot{2\eta}$.  
	Since $1-\eps_0\in(\nicefrac{1}{2},1)$,  
	by the trace theorem, 
	Theorem~\ref{thm:trace} in Appendix~\ref{appendix:function-spaces}, 
	there are $C, \widehat{C}, C'\! \in\bbR_+$ 
	independent of $\delta, v_\delta$ and $\psi$ such that, 
	for all  
	$j\in\bbN_0$ with 
	$j\leq\lfloor\eta-\nicefrac{5}{4}\rfloor$, 
	\begin{align*} 
		\bigl\| 
		\bigl( \kappa^2 
		&- \nabla\cdot(\ac\nabla)\bigr)^j \bigl( \delta_{\kappa^2} \psi \bigr) 
		\bigr\|_{L_2(\partial\cD)} 
		= 
		\bigl\| 
		\bigl( \kappa^2 - \nabla\cdot(\ac\nabla)\bigr)^j 
		\bigl( \delta_{\kappa^2} (\psi - v_\delta) \bigr)  
		\bigr\|_{L_2(\partial\cD)}
		\\ 
		&\leq 
		C   
		\bigl\| 
		\bigl( \kappa^2 - \nabla\cdot(\ac\nabla)\bigr)^j 
		\bigl( \delta_{\kappa^2} (\psi - v_\delta) \bigr)  
		\bigr\|_{H^{1-\eps_0}(\cD)} 
		\leq 
		\widehat{C} \,    
		\| \psi - v_\delta \|_{H^{2j+1-\eps_0}(\cD)} 
		\\ 
		&\leq 
		\widehat{C} \,    
		\| \psi - v_\delta \|_{H^{2\eta-3/2-\eps_0}(\cD)} 
		\leq 
		\widehat{C} \,   
		\| \psi - v_\delta \|_{H^{2\eta-d/2-\eps_0}(\cD)}
		\leq 
		C' 
		\| \psi - v_\delta \|_{2\eta-\nicefrac{d}{2}-\eps_0, L} 
		< \delta. 
	\end{align*}  
	As $\psi\in\Hdot{ 2\eta-\nicefrac{d}{2}-\eps_0}$ 
	and $\delta\in\bbR_+$ were arbitrary,  
	we conclude that 
	for every $\eta\in\frakN_\beta$ and 
	all $j\in\bbN_0$ with $j\leq\lfloor\eta-\nicefrac{5}{4}\rfloor$, 
	the following behavior on the boundary is satisfied:  
	\begin{equation}\label{eq:condition-hs}
		\forall \psi\in\Hdot{2\eta-\nicefrac{d}{2}-\eps_0}:
		\quad 
		\bigl( \kappa^2 - \nabla \cdot (\ac \nabla) \bigr)^j 
		\bigl( \delta_{\kappa^2} \psi \bigr) 
		= 
		0 
		\;\ \text{in}\;\; 
		L_2(\partial\cD) . 
	\end{equation}
	Furthermore, we have 
	$\nicefrac{d}{2} + \eps_0\in(\nicefrac{1}{2},2)$. 
	Therefore, 
	the regularity of 
	$\delta_{\kappa^2} \in C^\infty(\clos{\cD})$ 
	and \eqref{eq:condition-hs} 
	imply using the identification \eqref{eq:Hdot-Sobolev} 
	for $\Hdot{2(\eta-1)}$ that 
	$B := \Lt-L \in 
	\cL\bigl( \Hdot{2\eta-\nicefrac{d}{2}-\eps_0}, \Hdot{2(\eta-1)} \bigr)$  
	holds for every $\eta\in\frakN_\beta$. 
	This is equivalent to 
	$L^{\eta-1} B L^{-\eta+\nicefrac{d}{4}+\nicefrac{\eps_0}{2}} 
	\in \cL(L_2(\cD))$ 
	and we conclude that the operator 
	$S_\eta:= 
	L^{\eta-1} \Lt L^{-\eta} - \id_{L_2(\cD)} 
	= 
	L^{\eta-1} (\Lt - L) L^{-\eta}$ 
	is Hilbert--Schmidt on $L_2(\cD)$, since
	\begin{align*} 
		\| L^{\eta-1} (\Lt - L) L^{-\eta} \|_{\cL_2(L_2(\cD))} 
		&= 
		\bigl\| L^{\eta-1} B L^{-\eta+\nicefrac{d}{4}+\nicefrac{\eps_0}{2}} 
		L^{- \nicefrac{d}{4}-\nicefrac{\eps_0}{2}} \bigr\|_{\cL_2(L_2(\cD))} 
		\\
		&\leq 
		\bigl\| L^{\eta-1} B L^{-\eta+\nicefrac{d}{4}+\nicefrac{\eps_0}{2}} 
		\bigr\|_{\cL(L_2(\cD))} 
		\bigl\| L^{- \nicefrac{d}{4}-\nicefrac{\eps_0}{2}} \bigr\|_{\cL_2(L_2(\cD))}  
		< \infty 
	\end{align*}
	follows for all $\eta\in\frakN_\beta$ 
	by recalling \eqref{eq:weyl-HS}. 
	We thus obtain 
	the Hilbert--Schmidt property 
	of the operator $L^{-\beta}\Lt^{2\beta}L^{-\beta}-\id_{L_2(\cD)}$ 
	from Lemma~\ref{lem:iff-A}\ref{lem:iff-A-HS},  
	using $U_\eta=\id_{L_2(\cD)}$ for all $\eta\in\frakN_\beta$. 
\end{proof}  

For ease of presentation, 
the proof of the next lemma 
is postponed to Appendix~\ref{appendix:proof-a-notcompact}. 

\begin{lemma}\label{lem:a-notcompact}  
	Let $c\in\bbR_+$, $d\in\bbN$,  
	and suppose Assumption~\ref{ass:a-kappa-D}.III. 
	Let~$L$ and~$\Lt$ be defined as 
	in~\eqref{eq:L-div} 
	with coefficients $\ac, \kappa$ 
	and $\act, \kappat$, respectively, where 
	$\ac,\act$ fulfill Assumption~\ref{ass:a-kappa-D}.I 
	and 
	$\kappa,\kappat$ satisfy Assumption~\ref{ass:a-kappa-D}.II. 
	If $c \ac \neq \act$, then the operator 
	$L^{-\nicefrac{1}{4}} \Lt^{\nicefrac{1}{2}} L^{-\nicefrac{1}{4}} 
	- c^{\nicefrac{1}{2}} \id_{L_2(\cD)}$  
	is not compact on $L_2(\cD)$. 
\end{lemma}

\begin{lemma}\label{lem:kappa-notHS} 
	Let $d\in\bbN$, $d\geq 4$,   
	and suppose Assumption~\ref{ass:a-kappa-D}.III. 
	Let the operators~$L$ and~$\Lt$ be defined as 
	in \eqref{eq:L-div} 
	with coefficients $\ac,\kappa$ 
	and $\ac,\kappat$, respectively, where 
	$\ac$ fulfills Assumption~\ref{ass:a-kappa-D}.I 
	and $\kappa,\kappat$ satisfy 
	\noindent Assumption~\ref{ass:a-kappa-D}.II.  
	If $\kappa^2 \neq \kappat^2$, then 
	$L^{-\nicefrac{1}{2}} \Lt L^{-\nicefrac{1}{2}} - \id_{L_2(\cD)}$  
	is not Hilbert--Schmidt on $L_2(\cD)$. 
\end{lemma}

\begin{proof} 
	As in Theorem~\ref{thm:CM}, 
	we define $\delta_{\kappa^2}\in C^\infty(\clos{\cD})$ 
	by 
	$\delta_{\kappa^2}(\s) := \kappat^2(\s) - \kappa^2(\s)$, 
	$\s \in \clos{\cD}$. 
	Furthermore, 
	we set 
	$\delta^+(\s) := \max\{ \delta_{\kappa^2}(\s),0 \}$ 
	and 
	$\delta^-(\s) := -\min\{\delta_{\kappa^2}(\s),0 \}$,  
	$\s\in\clos{\cD}$.

	\textbf{Step 1}: 
	We first prove the claim for the case 
	that either $\delta^+(\s) \geq \delta_0 \in\bbR_+$ 
	holds for all $\s\in\clos{\cD}$ 
	or 
	${\delta^-(\s) \geq \delta_0 \in \bbR_+}$  
	holds for all $\s\in\clos{\cD}$. 
	Then, 
	$\delta_{\kappa^2}^{-1}(\s) := 1/\delta_{\kappa^2}(\s)$ 
	is well-defined, 
	$\delta_{\kappa^2}^{-1}\in C^\infty(\clos{\cD})$, and 
	the multiplier  
	$\mult{\delta_{\kappa^2}}\from L_2(\cD) \to L_2(\cD)$, 
	$v\mapsto \delta_{\kappa^2} v$, 
	is an isomorphism with 
	$\mult{\delta_{\kappa^2}}^{-1} = \mult{\delta_{\kappa^2}^{-1}}$. 
	Moreover, for every $v\in H_0^1(\cD)$, we have  
	$\delta_{\kappa^2} v 
	= \delta_{\kappa^2}^{-1} v = 0$ in $L_2(\partial\cD)$ as well as  
	$\nabla(\delta_{\kappa^2} v) 
	= v \nabla\delta_{\kappa^2} + \delta_{\kappa^2}\nabla v$
	and 
	$\nabla\bigl( \delta_{\kappa^2}^{-1} v \bigr) 
	= v \nabla\delta_{\kappa^2}^{-1} + \delta_{\kappa^2}^{-1}\nabla v$ 
	in $L_2(\cD)$. 
	Combining these relations 	
	with the identification 
	$\bigl(\Hdot{1}, \norm{\,\cdot\,}{1,L} \bigr) 
	\cong 
	\bigl( H^1_0(\cD), \norm{\,\cdot\,}{H^1(\cD)} \bigr)$, 
	see \eqref{eq:Hdot-Sobolev},  
	shows that 
	$\mult{\delta_{\kappa^2}}, \mult{\delta_{\kappa^2}^{-1}} 
	\in \cL\bigl( \Hdot{1} \bigr)$. 
	Since the operators 
	$\mult{\delta_{\kappa^2}}, \mult{\delta_{\kappa^2}^{-1}}$ 
	are self-adjoint on~$L_2(\cD)$, 
	also $\mult{\delta_{\kappa^2}}, \mult{\delta_{\kappa^2}^{-1}} 
	\in \cL\bigl( \Hdot{-1} \bigr)$ follows.   
	We conclude that  
	$L^{-\nicefrac{1}{2}} M_{\delta_{\kappa^2}} L^{\nicefrac{1}{2}}$ 
	is bounded on $L_2(\cD)$ and has a bounded 
	inverse,   
	$L^{-\nicefrac{1}{2}} M_{\delta_{\kappa^2}^{-1}} L^{\nicefrac{1}{2}} 
	\in \cL(L_2(\cD))$.
	Thus, 
	\begin{align*} 
		\bigl\| 
		L^{-\nicefrac{1}{2}}
		(\Lt - L) 
		L^{-\nicefrac{1}{2}} 
		\bigr\|_{\cL_2(L_2(\cD))} 
		&= 
		\bigl\| 
		L^{-\nicefrac{1}{2}} 
		\mult{\delta_{\kappa^2}} 
		L^{\nicefrac{1}{2}} L^{-1} 
		\bigr\|_{\cL_2(L_2(\cD))} 
		\\[-3pt]
		&\geq 
		\bigl\| 
		L^{-\nicefrac{1}{2}} 
		\mult{\delta_{\kappa^2}^{-1}} 
		L^{\nicefrac{1}{2}} 
		\bigr\|_{\cL(L_2(\cD))}^{-1} 
		\bigl\| L^{-1} \bigr\|_{\cL_2(L_2(\cD))} . 
	\end{align*}  
	The asymptotic behavior \eqref{eq:lambdaj} 
	implies 
	that
	$\bigl\| L^{-1} \bigr\|_{\cL_2(L_2(\cD))}^2 
	= 
	\sum\nolimits_{j\in\bbN} \lambda_j^{-2} 
	\geq 
	C_\lambda^{-2} 
	\sum\nolimits_{j\in\bbN} 
	j^{-1} =\infty$ 
	for $d\geq 4$ 
	and, hence, 
	$L^{-\nicefrac{1}{2}} \Lt L^{-\nicefrac{1}{2}} - \id_{L_2(\cD)} 
	= 
	L^{-\nicefrac{1}{2}}
	(\Lt - L)L^{-\nicefrac{1}{2}} \notin \cL_2(L_2(\cD))$.  
	
	\textbf{Step 2a}: 
	Suppose now that $\emptyset\neq \cD_0 \Subset \cD$ 
	is an open ball $\cD_0 := B(\s_0,r_0)$
	with center $\s_0\in\cD$ and radius $r_0\in\bbR_+$ 
	such that $\delta^+(\s)\geq \delta_0 \in\bbR_+$ 
	for all $\s\in\clos{\cD}_0$. 
	Then, 
	the self-adjoint compact 
	operator 
	$L^{-\nicefrac{1}{2}} \mult{\delta_{\kappa^2}}L^{-\nicefrac{1}{2}} 
	\in\cK(L_2(\cD))$ 
	has infinitely many positive 
	eigenvalues 
	${\mu_1^+ \geq \mu_2^+ \geq \ldots > 0}$  
	that are bounded from below 
	by those of the compact operator
	${L_0^{-\nicefrac{1}{2}} \mult{\delta_{\kappa^2}|_{\cD_0}} L_0^{-\nicefrac{1}{2}} 
		\in\cK(L_2(\cD_0))}$, 
	where 
	$L_0 \from \scrD(L_0) \subset L_2(\cD_0) \to L_2(\cD_0)$   
	is defined as in \eqref{eq:L-div} 
	with respect to the spatial domain $\cD_0\Subset\cD$ 
	and the coefficients $\ac|_{\clos{\cD}_0}$ and~$\kappa|_{\clos{\cD}_0}$. 
	This follows from 
	the min-max theorem, 
	see e.g.\ \cite[][Theorem X.4.3]{DunfordSchwartz1963}, 
	showing 
	that the eigenvalues 
	$\widetilde{\mu}_1^+ \geq \widetilde{\mu}_2^+ \geq \ldots >0$ 
	of the positive operator  
	$L_0^{-\nicefrac{1}{2}} \mult{\delta_{\kappa^2}|_{\cD_0}} 
	L_0^{-\nicefrac{1}{2}}$ satisfy  
	\begin{align*}
		0 < 
		\widetilde{\mu}_n^+ 
		&=  
		\max_{\substack{U_0 \subset L_2(\cD_0), \\  
				\dim(U_0) = n }} \, 
		\min_{w_0 \in U_0\setminus \{0\}} 
		\frac{\scalar{L_0^{-\nicefrac{1}{2}} 
				\mult{\delta_{\kappa^2}|_{\cD_0}} 
				L_0^{-\nicefrac{1}{2}} w_0, w_0}{L_2(\cD_0)}}{\scalar{w_0, w_0}{L_2(\cD_0)}} 
		\\[-9pt]
		&= 
		\max_{\substack{V_0 \subset H^1_0(\cD_0), \\  
				\dim(V_0) = n }} \, 
		\min_{v_0 \in V_0\setminus \{0\}} 
		\frac{\scalar{ \mult{\delta_{\kappa^2}|_{\cD_0}} v_0, v_0}{L_2(\cD_0)}}{
			 \duality{L_0 v_0, v_0}{} } , 
	\end{align*} 
	where we also used that 
	$\hdot{1}{L_0} \cong H^1_0(\cD_0)$. 
	If $\overline{v}_0\from\cD \to \bbR$ denotes the zero extension 
	of $v_0\from \cD_0\to \bbR$, 
	then 
	$\overline{v}_0\in H^1_0(\cD)$ holds 
	if and only if $v_0\in H^1_0(\cD_0)$, 
	cf.~\cite[Theorem~5.29]{AdamsSobolev2003}. 
	Consequently,  
	if we define the closed subspace 
	${\cV_0 
	:=
	\bigl\{ v \in H^1_0(\cD) \, \bigl| \, 
	\exists v_0 \in H^1_0(\cD_0)
	\text{ such that } v = \overline{v}_0   
		\bigr\}\subset H^1_0(\cD)} 
		\cong \Hdot{1}$, 
	we find 	
	\begin{align*}
		0< \widetilde{\mu}_n^+
		& 
		= 
		\max_{\substack{V \subset \cV_0, \\  
				\dim(V) = n }} \, 
		\min_{v\in V\setminus \{0\}} 
		\frac{\scalar{ \mult{\delta_{\kappa^2}} v, v}{L_2(\cD)}}{ 
			 \duality{L v, v}{} }  
		\leq 
		\max_{\substack{V \subset H_0^1(\cD), \\  
				\dim(V) = n }} \, 
		\min_{v\in V\setminus \{0\}} 
		\frac{\scalar{ \mult{\delta_{\kappa^2}} v, v}{L_2(\cD)}}{ 
			 \scalar{ L^{\nicefrac{1}{2}} v, L^{\nicefrac{1}{2}} v}{L_2(\cD)}}  
		\\[-3pt] 
		&= 
		\max_{\substack{U \subset L_2(\cD), \\
				\dim(U) = n }} \, 
		\min_{w\in U\setminus \{0\}} 
		\frac{\scalar{L^{-\nicefrac{1}{2}} \mult{\delta_{\kappa^2}} 
				L^{-\nicefrac{1}{2}} w, w}{L_2(\cD)}}{\scalar{w, w}{L_2(\cD)}} 
		=
		\mu_n^+ . 
	\end{align*} 
	We conclude that, 
	if $\delta_{\kappa^2}(\s) = \delta^+(\s) \geq \delta_0 > 0$ 
	for all $\s\in\cD_0$, then 
	${L^{-\nicefrac{1}{2}} \mult{\delta_{\kappa^2}}L^{-\nicefrac{1}{2}} 
	\in\cK( L_2(\cD) )}$ 
	has infinitely many positive eigenvalues 
	$\{\mu_n^+\}_{n\in\bbN}$ 
	satisfying $\mu_n^+ \geq \widetilde{\mu}^+_n$, 
	where $\{\widetilde{\mu}_n^+\}_{n\in\bbN}$  
	are the positive eigenvalues 
	of $L_0^{-\nicefrac{1}{2}} \mult{\delta_{\kappa^2}|_{\cD_0}} 
	L_0^{-\nicefrac{1}{2}}$\!.    
	
	\textbf{Step 2b}: 
	Suppose next that $\emptyset\neq \cD_0 \Subset \cD$ 
	is an open ball $\cD_0 := B(\s_0,r_0)$
	with center $\s_0\in\cD$ and radius $r_0\in\bbR_+$ 
	such that $\delta^-(\s)\geq \delta_0\in\bbR_+$ 
	for all $\s\in\clos{\cD}_0$. 
	Then, as in \textbf{Step~2a}
	we find that the operator 
	$L^{-\nicefrac{1}{2}} \mult{-\delta_{\kappa^2}}L^{-\nicefrac{1}{2}} 
	\in\cK( L_2(\cD) )$ 
	has infinitely many positive eigenvalues  
	$\{\mu_n^-\}_{n\in\bbN}$ 
	bounded from below by those 
	of  
	$L_0^{-\nicefrac{1}{2}} \mult{-\delta_{\kappa^2}|_{\cD_0}}L_0^{-\nicefrac{1}{2}}$ 
	denoted by $\{\widetilde{\mu}_n^-\}_{n\in\bbN}$. 
	
	\textbf{Step 3}: 
	Assume that $\kappa^2\neq\kappat^2$. 
	Then there exist 
	$\s_0\in\cD$ and $r_0,\delta_0\in\bbR_+$ 
	such that $B(\s_0, r_0)\Subset\cD$ and 
	such that 
	\textbf{a)}  
	$\delta^+(\s)\geq \delta_0$ 
	or 
	\textbf{b)}  
	$\delta^-(\s)\geq \delta_0$  
	for all $\s\in\clos{\cD}_0$. 
	By \textbf{Step~2} 
	the operator 
	$L^{-\nicefrac{1}{2}} \mult{\delta_{\kappa^2}} L^{-\nicefrac{1}{2}}$ 
	has 
	in case \textbf{a)} 
	infinitely many positive eigenvalues 
	$\{\mu_n^+\}_{n\in\bbN}$ which are bounded from below 
	by $\{\widetilde{\mu}_n^+\}_{n\in\bbN}$, i.e., by those of 
	$L_0^{-\nicefrac{1}{2}} \mult{\delta_{\kappa^2}|_{\cD_0}} L_0^{-\nicefrac{1}{2}}$\!,  
	and 
	in case \textbf{b)}  
	infinitely many negative eigenvalues 
	$\{-\mu_n^-\}_{n\in\bbN}$
	which are bounded from above 
	by $\{-\widetilde{\mu}_n^-\}_{n\in\bbN}$, i.e., by those of  
	$L_0^{-\nicefrac{1}{2}} \mult{\delta_{\kappa^2}|_{\cD_0}} L_0^{-\nicefrac{1}{2}}$\!.  
	By \textbf{Step~1} the eigenvalues 
	of $L_0^{-\nicefrac{1}{2}} \mult{\delta_{\kappa^2}|_{\cD_0}} L_0^{-\nicefrac{1}{2}}$ 
	are not square-summable~and 
	by \textbf{Step~2} neither 
	those 
	of $L^{-\nicefrac{1}{2}} \mult{\delta_{\kappa^2}} L^{-\nicefrac{1}{2}}$ 
	can be, i.e., 
	$L^{-\nicefrac{1}{2}} \mult{\delta_{\kappa^2}} L^{-\nicefrac{1}{2}}\notin \cL_2(L_2(\cD))$. 
\end{proof}    

We now can combine the aforegoing 
Lemmas~\ref{lem:HS:d<4}, \ref{lem:a-notcompact} and \ref{lem:kappa-notHS}  
with the general results on 
Gaussian measures 
with fractional-order covariance operators 
of Section~\ref{section:gm-on-hs} 
to deduce the following result 
on\vspace{-4mm}\linebreak\pagebreak

\noindent 
equivalence of 
Gaussian measures 
of generalized Whittle--Mat\'ern type. 

\begin{theorem}\label{thm:equivalence} 
	Let $d\in\bbN$, $\beta,\betat\in(\nicefrac{d}{4},\infty)$ 
	be such that $2\beta\notin\frakE$, with 
	$\frakE$ as in~\eqref{eq:exception}, 
	and suppose Assumption~\ref{ass:a-kappa-D}.III. 
	Let
	$L$ and $\Lt$ be defined as in~\eqref{eq:L-div}, 
	with coefficients $\ac,\kappa$ and $\act,\kappat$, respectively, 
	where each of the 
	tuples   
	$(\ac,\kappa)$ and~$(\act,\kappat)$ 
	fulfills Assumptions~\ref{ass:a-kappa-D}.I--II. 
	Assume that $m, \mt \in L_2(\cD)$ 
	and let the Gaussian measures 
	$\mu_d(m;\beta,\ac,\kappa)$ 
	and 
	$\mu_d(\mt; \betat,\act,\kappat)$ 
	be defined according to \eqref{eq:def:mu}. 
	\begin{enumerate}[label={\normalfont\Roman*.}, topsep=2pt]
		\item\label{thm:equivalence-d<4} 
		In dimension $d\leq 3$, 
		the Gaussian measures 
		$\mu_d(m;\beta,\ac,\kappa)$ 
		and 
		$\mu_d(\mt; \betat,\act,\kappat)$ 
		are equivalent 
		if and only if 
		$\beta=\betat$, 
		$m-\mt \in\Hdot{2\beta}$\!, 
		$\ac=\act$, and the boundary 
		conditions \eqref{eq:condition-iso} hold for 
		every $j\in\bbN_0$ with $j\leq\lfloor\beta-\nicefrac{5}{4}\rfloor$.  
		\item\label{thm:equivalence-d>3}  
		In dimension $d\geq 4$, 
		the Gaussian measures 
		$\mu_d(m;\beta,\ac,\kappa)$ 
		and 
		$\mu_d(\mt;\betat,\act,\kappat)$ 
		are equivalent 
		if and only if 
		$\beta=\betat$, 
		$m-\mt \in\Hdot{2\beta}$\!, 
		$\ac=\act$, 
		and $\kappa^2=\kappat^2$.  
	\end{enumerate}
\end{theorem} 

\begin{proof} 
	For the derivation of these results  
	we apply the 
	Feldman--H\'ajek theorem, 
	see Theorem~\ref{thm:feldman-hajek} 
	in Appendix~\ref{appendix:feldman-hajek}.  
	To this end, we let 
	$\cC=L^{-2\beta}$ and $\CC=\Lt^{-2\betat}$ 
	denote the covariance operators  
	corresponding to  
	$\mu_d(m;\beta,\ac,\kappa)$ 
	and 
	$\mu_d(\mt;\betat,\act,\kappat)$, 
	respectively.  
	By Theorem~\ref{thm:CM} 
	the Cameron--Martin 
	spaces $\cC(L_2(\cD))=\Hdot{2\beta}$   
	and $\CC(L_2(\cD))=\hdot{2\betat}{\Lt}$  
	are norm equivalent spaces  
	(and thus condition~\ref{fh-1} 
	of Theorem~\ref{thm:feldman-hajek} is fulfilled)  
	if and only 
	if $\beta=\betat$ and 
	\eqref{eq:condition-iso} holds 
	for all 
	$j\in\bbN_0$ with $j\leq\lfloor\beta-\nicefrac{5}{4}\rfloor$. 
	Next we note that condition~\ref{fh-2} 
	of Theorem~\ref{thm:feldman-hajek} 
	is equivalent to requiring that $m-\mt\in\Hdot{2\beta}$\!.  
	
	Assuming that $\beta=\betat$ 
	with $2\beta\notin\frakE$, 
	we complete the proof by showing that 
	conditions \ref{fh-1} and \ref{fh-3} 
	of Theorem~\ref{thm:feldman-hajek} hold simultaneously, i.e., 
	$\Lt^\beta L^{-\beta}$ is an isomorphism on $L_2(\cD)$ 
	and 
	the operator 
	$L^{-\beta} \Lt^{2\beta} L^{-\beta} - \id_{L_2(\cD)}$ is 
	Hilbert--Schmidt on $L_2(\cD)$ 
	\ref{thm:equivalence-d<4} 
	in dimension $d\leq 3$ if and only if 
	$\ac=\act$ and \eqref{eq:condition-iso} holds 
	for all 
	$j\in\bbN_0$ with $j\leq\lfloor\beta-\nicefrac{5}{4}\rfloor$; 
	and 
	\ref{thm:equivalence-d>3} 
	for $d\geq 4$ if and only if 
	$\ac=\act$ and $\kappa^2=\kappat^2$\!. 
	
	\ref{thm:equivalence-d<4} If $d\leq 3$, 
	$\ac=\act$, and 
	\eqref{eq:condition-iso} holds 
	for all 
	$j\in\bbN_0$ with $j\leq\lfloor\beta-\nicefrac{5}{4}\rfloor$, 
	then $\Lt^\beta L^{-\beta}$ 
	is an isomorphism on $L_2(\cD)$ 
	and  
	$L^{-\beta}  \Lt^{2\beta} L^{-\beta} - \id_{L_2(\cD)}$ 
	is  Hilbert--Schmidt  on~$L_2(\cD)$ 
	by Lemma~\ref{lem:HS:d<4}. 
	Conversely, if
	$\Lt^\beta L^{-\beta}$ 
	is an isomorphism on $L_2(\cD)$ 
	and 
	${L^{-\beta}  \Lt^{2\beta} L^{-\beta} - \id_{L_2(\cD)} \in \cL_2(L_2(\cD))}$, 
	then by Lemma~\ref{lem:beta-gamma}\ref{lem:beta-gamma-HS} 
	for every $\gamma\in[-\beta,\beta]$  
	also the operator $\Lt^\gamma L^{-\gamma}$ 
	is an isomorphism on $L_2(\cD)$ 
	and 
	$L^{-\gamma}  \Lt^{2\gamma} L^{-\gamma} - \id_{L_2(\cD)}$ 
	is a Hilbert--Schmidt operator on $L_2(\cD)$. 
	Since $2\beta\notin\frakE$ is assumed, 
	by Theorem~\ref{thm:CM} 
	the conditions 
	\eqref{eq:condition-iso} 
	have to be satisfied  
	for all $j\in\bbN_0$ with $j\leq\lfloor\beta-\nicefrac{5}{4}\rfloor$. 
	Furthermore, the choice $\gamma=\nicefrac{1}{4}$ shows that 
	$L^{-\nicefrac{1}{4}} \Lt^{\nicefrac{1}{2}} L^{-\nicefrac{1}{4}} - \id_{L_2(\cD)}$ is 
	Hilbert--Schmidt and, thus, compact on~$L_2(\cD)$. 
	Lemma~\ref{lem:a-notcompact} (with $c=1$) therefore 
	implies then that 
	$\ac = \act$ has to hold.  
	
	\ref{thm:equivalence-d>3} If $d\geq 4$, 
	$\ac=\act$ and $\kappa^2=\kappat^2$, 
	then $L=\Lt$ so that the isomorphism property 
	of $\Lt^\beta L^{-\beta}$ and the Hilbert--Schmidt property of 
	$L^{-\beta}  \Lt^{2\beta} L^{-\beta} - \id_{L_2(\cD)}$ are trivial. 
	Conversely, if 
	$\Lt^\beta L^{-\beta}$ is an isomorphism on $L_2(\cD)$ 
	and 
	$L^{-\beta}\Lt^{2\beta}L^{-\beta} - \id_{L_2(\cD)} \in\cL_2(L_2(\cD))$ 
	in dimension $d\geq 4$, 
	then $\beta > \nicefrac{d}{4} \geq 1$ and 
	by Lemma~\ref{lem:beta-gamma}\ref{lem:beta-gamma-HS} 
	also  the operators 
	$L^{-\nicefrac{1}{4}} \Lt^{\nicefrac{1}{2}} L^{-\nicefrac{1}{4}} - \id_{L_2(\cD)}$  
	as well as 
	$L^{-\nicefrac{1}{2}} \Lt L^{-\nicefrac{1}{2}} - \id_{L_2(\cD)}$ 
	are Hilbert--Schmidt (and, thus, compact) on $L_2(\cD)$. 
	By Lemma~\ref{lem:a-notcompact} 
	$\ac = \act$ follows and, 
	subsequently,  
	Lemma~\ref{lem:kappa-notHS} 
	shows that $\kappa^2=\kappat^2$\!. 
\end{proof}

\subsection{Uniformly asymptotically optimal linear prediction}
\label{subsec:wm-D:kriging} 

In contrast to equivalence of the Gaussian measures 
$\mu_d(m;\beta,\ac,\kappa)$ and 
$\mu_d(\mt;\betat,\act,\kappat)$, 
the necessary and sufficient conditions 
for uniformly asymptotically optimal linear prediction 
\eqref{eq:prop:A-kriging-1}, \eqref{eq:prop:A-kriging-2} 
when misspecifying 
$\mu_d(m;\beta,\ac,\kappa)$ by  
$\mu_d(\mt;\betat,\act,\kappat)$ 
derived in this subsection 
will not depend on the dimension~$d$ 
of the spatial domain $\cD\subset\bbR^d$\!. 
The key to prove this result is the next lemma. 

\begin{lemma}\label{lem:compact}  
	Let $d\in\bbN$, $c\in\bbR_+$, and let  
	$\beta \in (\nicefrac{d}{4},\infty)$ be such that 
	$2\beta\notin\frakE$, 
	where $\frakE$ is as in \eqref{eq:exception}. 
	In addition, suppose Assumption~\ref{ass:a-kappa-D}.III 
	and 
	let the operators~$L, \Lt$ be defined as 
	in \eqref{eq:L-div} 
	with coefficients $\ac,\kappa$ 
	and $\act,\kappat$, respectively, where 
	$\ac$ fulfills Assumption~\ref{ass:a-kappa-D}.I, 
	$c\ac=\act$, 
	and $\kappa,\kappat$ are such that  
	Assumption~\ref{ass:a-kappa-D}.II 
	is satisfied and 
	\eqref{eq:condition-iso} holds for all $j\in\bbN_0$ with  
	$j\leq\lfloor\beta-\nicefrac{5}{4}\rfloor$. 
	Then, 
	$\Lt^\beta L^{-\beta}$ is an isomorphism on~$L_2(\cD)$ 
	and 
	$L^{-\beta} \Lt^{2\beta} L^{-\beta} - c^{2\beta} \id_{L_2(\cD)}\in\cK( L_2(\cD) )$. 
\end{lemma} 

\begin{proof}
	By Theorem~\ref{thm:CM} 
	$\Lt^\gamma L^{-\gamma}$ is an isomorphism on $L_2(\cD)$ 
	for all $\gamma\in[-\beta,\beta]$. 
	To prove compactness 
	of 
	$L^{-\beta} \Lt^{2\beta} L^{-\beta} - c^{2\beta} \id_{L_2(\cD)}$, 
	similarly as in the proof of Lemma~\ref{lem:HS:d<4}, 
	we distinguish two cases, 
	\textbf{Case I:} 
	$d\in\{1,2,3\}$, $\beta\in(\nicefrac{d}{4},1)$ 
	and 
	\textbf{Case II:} $\beta\in[1,\infty)$. 
	
	\textbf{Case I:} If $d\in\{1,2,3\}$ and  
	$\beta\in(\nicefrac{d}{4},1)$, 
	then we use the identity 
	\begin{equation}\label{eq:proof:lem:compact-1} 
		\begin{split} 
			L^{-\beta} \Lt^{2\beta} L^{-\beta} - c^{2\beta} \id_{L_2(\cD)} 
			&= 
			\tfrac{1}{2}
			L^{-\beta} 
			\bigl( \Lt^{\beta} + c^\beta L^\beta \bigr) 
			\bigl( \Lt^{\beta} - c^\beta L^\beta \bigr) L^{-\beta} 
			\\
			&\quad + 
			\tfrac{1}{2} 
			\bigl[ 
			L^{-\beta} 
			\bigl( \Lt^{\beta} + c^\beta L^\beta \bigr) 
			\bigl( \Lt^{\beta} - c^\beta L^\beta \bigr) 
			L^{-\beta}  
			\bigr]^*\!. 
		\end{split} 
	\end{equation} 
	Clearly, $L^{-\beta}\Lt^{\beta}\in\cL(L_2(\cD))$ 
	is bounded, 
	since $\Lt^\beta L^{-\beta}$ 
	is an isomorphism. 
	Furthermore, since 
	${(\Lt-cL)\psi = \delta_{c,\kappa^2}\psi}$, 
	where 
	$\delta_{c,\kappa^2}:=\kappat^2-c\kappa^2\in C^\infty(\clos{\cD})$, 
	we find that $\Lt-cL \in \cL( L_2(\cD) )$.  
	Thus, by Lemma~\ref{lem:A-alpha-difference}  
	and Remark~\ref{rem:A-alpha-difference}
	also $\Lt^\beta-c^\beta L^\beta\in\cL(L_2(\cD))$ 
	holds.  
	Combining these observations with \eqref{eq:proof:lem:compact-1}  
	and $L^{-\beta}\in\cK(L_2(\cD))$ 
	shows that 
	$L^{-\beta} \Lt^{2\beta} L^{-\beta} - c^{2\beta} \id_{L_2(\cD)} \in\cK(L_2(\cD))$. 
	
	\textbf{Case II:}
	Define the operator $L_c := c L$. 
	Then, also $\Lt^\gamma L_c^{-\gamma}$ 
	is an isomorphism on $L_2(\cD)$ 
	for all $\gamma\in[-\beta,\beta]$. 
	By Theorem~\ref{thm:CM},  
	for all $\eta\in\frakN_\beta$, 
	where $\frakN_\beta$ is as in \eqref{eq:frakNset}, 
	and 
	all $j\in\bbN_0$ with $j\leq \lfloor\eta-\nicefrac{5}{4}\rfloor$, 
	\[ 
		\forall v \in \hdot{2\eta}{L_c} = \Hdot{2\eta} : 
		\quad 
		\bigl( \kappa^2 - \nabla\cdot(\ac\nabla)\bigr)^j \bigl( \delta_{c,\kappa^2} v \bigr) = 0 
		\;\ \text{in}\;\; 
		L_2(\partial\cD) . 
	\] 
	We pick $\eps_0\in(0,2)$ such that 
	$2\eta-\eps_0\notin\frakE$ for all $\eta\in\frakN_\beta$, 
	and we fix $\eta\in\frakN_\beta$, 
	$\psi\in\Hdot{2\eta-\eps_0}$,~and  $\delta\in\bbR_+$. 
	By density of $\Hdot{2\eta}$ in $\Hdot{2\eta-\eps_0}$, 
	there exists $v_\delta\in\Hdot{2\eta}$ 
	such that $\norm{\psi-v_\delta}{2\eta-\eps_0,L}<\delta$. 
	Furthermore, by Lemma~\ref{lem:Hdot-Sobolev}
	on $\Hdot{2\eta-\eps_0}$ the norm $\norm{\,\cdot\,}{2\eta-\eps_0,L}$ 
	is equivalent to 
	the Sobolev norm $\norm{\,\cdot\,}{H^{2\eta-\eps_0}(\cD)}$.  
	 
	Thus, by the trace theorem,  
	Theorem~\ref{thm:trace} in Appendix~\ref{appendix:function-spaces}, 
	and by noting that 
	$\nicefrac{5}{2}-\eps_0\in(\nicefrac{1}{2}, \nicefrac{5}{2})$,  
	for all  
	$j\in\bbN_0$ with 
	$j\leq\lfloor\eta-\nicefrac{5}{4}\rfloor$, 
	we find that 
	\begin{align*} 
		\bigl\| 
		\bigl( \kappa^2 
		&- \nabla\cdot(\ac\nabla)\bigr)^j \bigl( \delta_{c,\kappa^2} \psi \bigr) 
		\bigr\|_{L_2(\partial\cD)} 
		= 
		\bigl\| 
		\bigl( \kappa^2 - \nabla\cdot(\ac\nabla)\bigr)^j 
		\bigl( \delta_{c,\kappa^2} (\psi - v_\delta) \bigr)  
		\bigr\|_{L_2(\partial\cD)}
		\\ 
		&\leq 
		C \bigl\| 
		\bigl( \kappa^2 - \nabla\cdot(\ac\nabla)\bigr)^j 
		\bigl( \delta_{c,\kappa^2} (\psi - v_\delta) \bigr)  
		\bigr\|_{H^{5/2-\eps_0}(\cD)} 
		\leq 
		\widehat{C}  
		\| \psi - v_\delta \|_{H^{2\eta-\eps_0}(\cD)} 
		< C' \delta, 
	\end{align*}  
	where the constants $C,\widehat{C},C'\!\in\bbR_+$ 
	are independent of $\delta,v_\delta$ and $\psi$. 
	Since $\psi\in\Hdot{2\eta-\eps_0}$ 
	and $\delta\in\bbR_+$ were arbitrary,  
	we thus find that,  
	for every $\eta\in \frakN_\beta$ and 
	all $j\in\bbN_0$ with $j\leq\lfloor\eta-\nicefrac{5}{4}\rfloor$, 
	\begin{equation}\label{eq:condition-compact}
		\forall \psi\in \hdot{2\eta-\eps_0}{L_c}=\Hdot{2\eta-\eps_0}:
		\quad 
		\bigl( \kappa^2 - \nabla \cdot (\ac \nabla) \bigr)^j 
		\bigl( \delta_{c,\kappa^2} \psi \bigr) 
		= 
		0 
		\;\ \text{in}\;\; 
		L_2(\partial\cD) . 
	\end{equation}

	Since $\eps_0\in(0,2)$, 
	\eqref{eq:condition-compact}  
	and the regularity of 
	$\delta_{c,\kappa^2} \in C^\infty(\clos{\cD})$ 
	imply by identifying $\hdot{2(\eta-1)}{L_c}$ as in \eqref{eq:Hdot-Sobolev} 
	that 
	$B_c := \Lt-cL \in 
	\cL\bigl( \hdot{2\eta-\eps_0}{L_c}, \hdot{2(\eta-1)}{L_c} \bigr)$ 
	and 
	$L_c^{\eta-1} B_c L_c^{-\eta+\nicefrac{\eps_0}{2}} 
	\in \cL(L_2(\cD))$. 
	Then, we find that 
	$K_\eta 
	:= 
	L_c^{\eta-1} \Lt L_c^{-\eta} - \id_{L_2(\cD)} 
	=
	L_c^{\eta-1} (\Lt - L_c) L_c^{-\eta} 
	=
	\bigl( L_c^{\eta-1} B_c L_c^{-\eta+\nicefrac{\eps_0}{2}} \bigr) L_c^{-\nicefrac{\eps_0}{2}}  
	\in 
	\cK(L_2(\cD))$
	for every $\eta\in\frakN_\beta$, because   
	$L_c^{-\nicefrac{\eps_0}{2}} 
	= 
	c^{-\nicefrac{\eps_0}{2}} L^{-\nicefrac{\eps_0}{2}}$ 
	is compact on $L_2(\cD)$. 
	Applying Lemma~\ref{lem:iff-A}\ref{lem:iff-A-compact} 
	(for $\At:=\Lt$ and $A:=L_c$, 
	using ${W_\eta:=\id_{L_2(\cD)}}$ 
	for every $\eta\in\frakN_\beta$)
	finally yields compactness of  
	$L^{-\gamma}\Lt^{2\gamma}L^{-\gamma} - c^{2\gamma}\id_{L_2(\cD)} 
	= 
	c^{2\gamma} \bigl( L_c^{-\gamma}\Lt^{2\gamma}L_c^{-\gamma} - \id_{L_2(\cD)} \bigr)$ 
	on $L_2(\cD)$ for all $\gamma\in[-\beta,\beta]$. 
\end{proof}

\begin{theorem}\label{thm:kriging} 
	Let $d\in\bbN$, $\beta,\betat\in(\nicefrac{d}{4},\infty)$ 
	be such that $2\beta\notin\frakE$, with 
	$\frakE$  as in~\eqref{eq:exception}, 
	and let Assumption~\ref{ass:a-kappa-D}.III 
	be satisfied. 
	Suppose that 
	$L, \Lt$ are defined as in~\eqref{eq:L-div}, 
	with coefficients $\ac,\kappa$ and $\act,\kappat$, respectively, 
	where each of the 
	tuples   
	$(\ac,\kappa)$ and~$(\act,\kappat)$ 
	fulfills Assumptions~\ref{ass:a-kappa-D}.I--II. 
	Let $m,\mt \in L_2(\cD)$ 
	and the Gaussian measures 
	$\mu_d(m;\beta,\ac,\kappa)$ 
	and 
	$\mu_d(\mt;\betat,\act,\kappat)$ 
	be defined according to \eqref{eq:def:mu}. 
	In addition, let $h_n, \widetilde{h}_n$ denote the best linear
	predictors of $h\in\cH$  based on~$\cH_n$ and the measures
	$\mu_d(m;\beta,\ac,\kappa)$ resp.\ 
	$\mu_d(\mt;\betat,\act,\kappat)$, 
	see \eqref{eq:def:cHn}--\eqref{eq:def:S-adm}. 
	Then, any of the four assertions in 
	\eqref{eq:prop:A-kriging-1},  
	\eqref{eq:prop:A-kriging-2} 
	holds for some $c\in\bbR_+$ and  all 
	$\{\cH_n\}_{n\in\bbN} \in \cS^\mu_{\mathrm{adm}}$
	if and only if 
	$\beta=\betat$,  
	$m-\mt\in\Hdot{2\beta}$\!, 
	the boundary conditions~\eqref{eq:condition-iso}  
	hold for every $j\in\bbN_0$ with $j\leq \lfloor\beta-\nicefrac{5}{4} \rfloor$, 
	and 
	there exists a  constant $\hat{c}\in\bbR_+$ such that 
	$\hat{c}\ac=\act$. 
\end{theorem} 

\begin{proof} 
	By \cite[][Theorem~3.8 and Lemma~B.1]{bk-kriging} 
	any of the assertions in  
	\eqref{eq:prop:A-kriging-1}, \eqref{eq:prop:A-kriging-2}  
	holds for 
	some constant $c\in\bbR_+$ 
	and all   
	$\{\cH_n\}_{n\in\bbN} \in \cS^\mu_{\mathrm{adm}}$ 
	if and only if 
	\begin{enumerate*}[label={\normalfont(\roman*)}]
		\item\label{proof:thm:kriging-i} 
		$\Hdot{2\beta}$ 
		and 
		$\hdot{2\betat}{\Lt}$ 
		are norm equivalent Hilbert spaces;   
		\item\label{proof:thm:kriging-ii} 
		$m - \mt \in \Hdot{2\beta}$\!; and 
		\item\label{proof:thm:kriging-iii} 
		$L^{-\beta} \Lt^{2\beta} L^{-\beta} - c^{-1} \id_{L_2(\cD)}$ 
		is compact on $E$. 
	\end{enumerate*}  
	
	By Theorem~\ref{thm:CM} 
	$\Hdot{2\beta}$ 
	and $\hdot{2\betat}{\Lt}$
	are norm equivalent 
	if and only 
	if $\beta=\betat$ and 
	\eqref{eq:condition-iso} holds 
	for every  
	${j\in\bbN_0}$ with $j\leq\lfloor\beta-\nicefrac{5}{4}\rfloor$. 
	Assuming that $\beta=\betat$ 
	with $2\beta\notin\frakE$, 
	sufficiency of 
	\begin{enumerate*}[label=(\alph*)]
		\item the boundary conditions \eqref{eq:condition-iso} holding  
		for all 
		$j\in\bbN_0$ with $j\leq\lfloor\beta-\nicefrac{5}{4}\rfloor$,  
		combined with  
		\item the existence of $\hat{c} \in\bbR_+$ such that  
		$\hat{c}\ac=\act$, 
	\end{enumerate*}
	for conditions \ref{proof:thm:kriging-i} and \ref{proof:thm:kriging-iii}   
	is proven in Lemma~\ref{lem:compact} (showing 
	in particular that $c^{-1}=\hat{c}^{2\beta}$). 
	Necessity of (a) and (b) follows from 
	Theorem~\ref{thm:CM} and Lemma~\ref{lem:a-notcompact}. 
\end{proof}

\section{Illustration by simulations}\label{section:numexp}

In this section we illustrate the theoretical findings 
of the previous sections by two different examples of kriging prediction 
based on misspecified generalized Whittle--Mat\'ern models 
\eqref{eq:whittle-matern}, see also \eqref{eq:Lbeta} 
and~\eqref{eq:L-div}, on $\cD := (0,1)$. 
We first consider a non-fractional model with $\beta=1$ 
and discuss the difference 
between a misspecification of $\kappa^2$ and of $\ac$. 
We then consider the role of $\beta$ 
when misspecifying $\kappa^2$\!. 
These examples verify, in particular, that one obtains asymptotic optimality 
even if $\kappa^2$ is misspecified for $\beta<\nicefrac{9}{4}$. 
In contrast, when $\beta > \nicefrac{9}{4}$, 
it is the behavior of $\kappa^2$ at the boundary $\partial\cD = \{0,1\}$  
of the domain $\cD=(0,1)$ 
that determines whether asymptotic optimality is achieved or not,  
see Table~\ref{tab:summary}.
The results are implemented in {\sc Matlab} using the \texttt{ppfem} 
package \citep{ppfem} for discretizing the models. 

To facilitate interpreting the parameters, 
we make a small adjustment to the Whittle--Mat\'ern model 
\eqref{eq:whittle-matern} by including 
a constant $\tau\in\bbR_+$ which scales the variance of the solution:
\begin{equation}\label{eq:model_ex2}
	\bigl( - \nabla \cdot (\ac \nabla) + \kappa^2 \bigr)^\beta (\tau \GP) 
	= \white \quad \text{in} \;\; \cD= (0,1).
\end{equation}
Note that this constant has no effect on the kriging prediction. 

\subsection{The difference between 
	\texorpdfstring{$\kappa^2$ and $\ac$}{reaction and diffusion coefficients}} 

Consider \eqref{eq:model_ex2} with 
$\beta = 1$, $\ac\equiv 1$,  
$\kappa^2 \equiv 1200$, and $\tau = \frac{1}{2} \kappa^{-\nicefrac{3}{2}}$\!.  
These choices result in a process $\GP$ with practical correlation range $0.1$ 
and a variance of approximately $1$ 
at the center of the domain, see~\eqref{eq:spde_var}. 

We approximate the solution 
$\GP\from[0,1]\times\Omega\to\bbR$ of \eqref{eq:model_ex2} 
using a finite element method (FEM) with $N=1000$ 
equally spaced 
continuous, piecewise linear 
basis functions 
$\{\varphi_k\}_{k=1}^N$, 
aka.\ ``hat functions''\!. 
The resulting approximation can be written as 
${\GP(\s) \approx \sum_{k=1}^N \gp_k \varphi_k(\s)}$, 
where the distribution of the weights $\{\gp_k\}_{k=1}^N$ 
is zero-mean multivariate Gaussian 
with covariance matrix $\mv{C} = \mv{L}^{-1}\mv{M}\mv{L}^{-1}$\!. 
The matrix $\mv{L}$ has elements 
$L_{jk} = a_L(\varphi_j,\varphi_k)$, 
where $a_L(\,\cdot\,, \,\cdot\,)$ denotes the bilinear form induced by $L$, 
see \eqref{eq:a-L} in Subsection~\ref{subsec:wm-D:summary}, 
and $\mv{M}$ is the mass matrix (aka.\ Gramian) with elements 
$M_{jk} = (\varphi_j,\varphi_k)_{L_2(\cD)}$. 
For details on the implementation, such as 
the assembling of the matrices, 
see \cite{lindgren11} or \cite{BK2020rational}. 

To evaluate the effect  of misspecifying the covariance function, 
the predictive performance of the correct model is compared 
with two misspecified models. 
For both misspecified models, we use 
the correct values of $\beta, \tau$, and we set
\begin{equation}\label{eq:models1}
	\bigl(\kappa^2(\s), \ac(\s) \bigr) = 
	\begin{cases}
		\bigl( 1200 f(\s)^{-1}, 1 \bigr) & \text{for model 1}, \\
		\bigl( 1200, f(\s) \bigr) & \text{for model 2}, 
	\end{cases}
	\qquad \s\in\clos{\cD}=[0,1]. 
\end{equation}
Here, 
$f(\s) = 1 + \frac1{2} \erf\Bigl( \frac{\delta(\s-0.5)}{\sqrt{2}} \Bigr)$ 
is a sigmoid function defined through the error function, 
and $\delta\in\bbR_+$ is a parameter 
that determines the rate of change of 
$f(\s)$ at $\s=0.5$. 
Thus, for the first model the coefficient 
$\kappa^2$\!, 
which is constant in the true model, 
is misspecified by a function, 
whereas in the second model 
this scenario applies to the coefficient~$\ac$. 
Note that, for both models, 
$\kappa^2(\s)$ and $\ac(\s)$ attain the correct values 
at $\s=0.5$.
The two misspecified models are approximated 
by means of a FEM approximation 
with the same basis functions as used for the  
true model. 

We consider kriging 
prediction in two different scenarios.
In both scenarios, we use 
\[
	\cE_n(h) := \frac{\pE\bigl[(\widetilde{h}_n-h)^2\bigr]}{\pE\bigl[(h_n - h)^2\bigr]} - 1,
\]
as a measure of efficiency of the   
best linear predictor obtained by a misspecified model. 
The quantity $\cE_n(h)$ is always nonnegative and should converge to zero 
if the misspecified model provides asymptotically optimal  
linear prediction, 
see Proposition~\ref{prop:A-kriging} and Theorem~\ref{thm:kriging}. 

In the first scenario, we predict integral values of $\GP$. 
Specifically, for $\ell\in\bbN$, 
let ${I_\ell := (\GP, e_\ell )_{L_2(\cD)}}$ 
be the integral over the product of 
the process 
with $e_\ell(\s) := \sqrt{2}\sin(\ell \pi \s)$,   
which is the $\ell$-th eigenfunction of the 
(negative) Dirichlet Laplacian $- \Delta$ on $\cD=(0,1)$.  
In order to evaluate $I_\ell$, 
we use the FEM approximation 
$I_\ell \approx \sum_{k=1}^N \gp_k (e_\ell,\varphi_k)_{L_2(\cD)}$ 
and evaluate the integral $\obsmat_{\ell k} = (e_\ell,\varphi_k)_{L_2(\cD)}$ 
by means of Gauss--Legendre quadrature. 
Collecting these elements in a matrix $\bobsmat$ 
we find that the 
joint distribution of $(I_1,\ldots, I_N)$ is 
multivariate Gaussian with 
mean zero and 
covariance matrix 
$\mv{\Sigma} = \bobsmat\mv{C}\bobsmat^\top$\!. 
Given  $I_1,\ldots, I_n$, 
we then predict $h = I_{\ell}$ for all $\ell\in\{ n+1,\ldots, N\}$. 
The variance of the error of this predictor can be obtained as 
$\pE\bigl[(h_n - h)^2\bigr] 
= 
\mv{\Sigma}_{\ell,\ell} - \mv{\Sigma}_{\ell,1:n}\mv{\Sigma}_{1:n,1:n}^{-1}\mv{\Sigma}_{\ell,1:n}^\top$. 
Here $\mv{\Sigma}_{\ell,1:n}$ denotes the first $n$ elements 
of the $\ell$-th row of $\mv{\Sigma}$ and $\mv{\Sigma}_{1:n,1:n}$ 
is the $n\times n$ sub-matrix corresponding to the $n$ observations. 
If we let $\widetilde{\mv{C}}$ denote the covariance matrix 
for the weights of a model with misspecified parameters 
and set $\widetilde{\mv{\Sigma}} = \bobsmat\widetilde{\mv{C}}\bobsmat^\top$, 
we similarly obtain that 
\[ 
\pE\bigl[(\widetilde{h}_n-h)^2\bigr] 
= 
\mv{\Sigma}_{\ell,\ell}  
+ 
\widetilde{\mv{\Sigma}}_{\ell,1:n} 
\widetilde{\mv{\Sigma}}_{1:n,1:n}^{-1} 
\mv{\Sigma}_{1:n,1:n}\widetilde{\mv{\Sigma}}_{1:n,1:n}^{-1} 
\widetilde{\mv{\Sigma}}_{\ell,1:n}^\top 
- 
2 \mv{\Sigma}_{\ell,1:n} 
\widetilde{\mv{\Sigma}}_{1:n,1:n}^{-1} 
\widetilde{\mv{\Sigma}}_{\ell,1:n}^\top.
\]
The left panel of Figure~\ref{fig1} shows
\begin{equation}\label{eq:En_max}
	\cE_{I,n}^{\max} 
	:= 
	\max \bigl\{ \cE_{I,n}^\ell :  n+1 \leq \ell \leq N\bigr\} , 
	\quad\;\; 
	\cE_{I,n}^\ell := \cE_n(I_\ell),\quad \ell \in \{n+1,\ldots,N\},
\end{equation}
as a function of $n$ for both misspecified models, 
where we consider values for $n$ up to $500$, 
so that the maximum in \eqref{eq:En_max} is taken 
over at least $500$ elements for each $n$. 
This error is computed for three different values of $\delta$, 
namely $\delta\in\{1,10,100\}$, 
where a larger value of $\delta$ intuitively should cause 
a bigger error for the misspecified model. 
\begin{figure}[t]
	\includegraphics[width=0.49\linewidth]{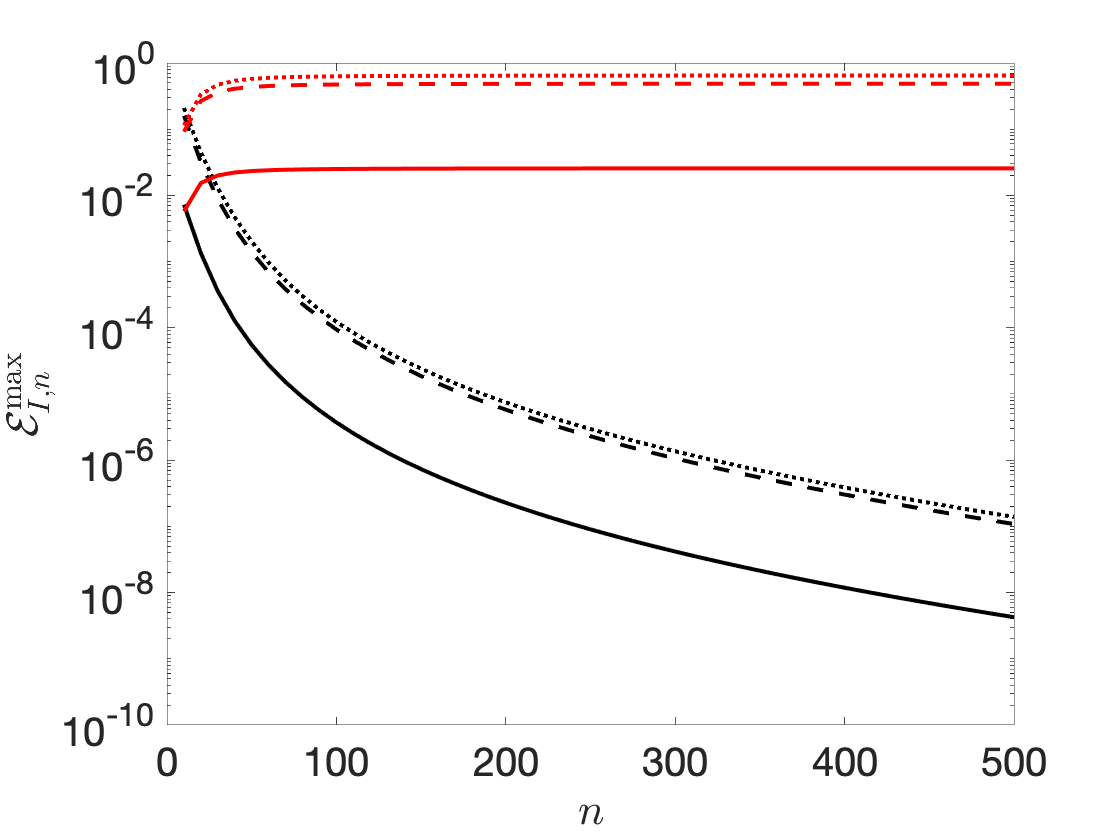}
	\hfill 
	\includegraphics[width=0.49\linewidth]{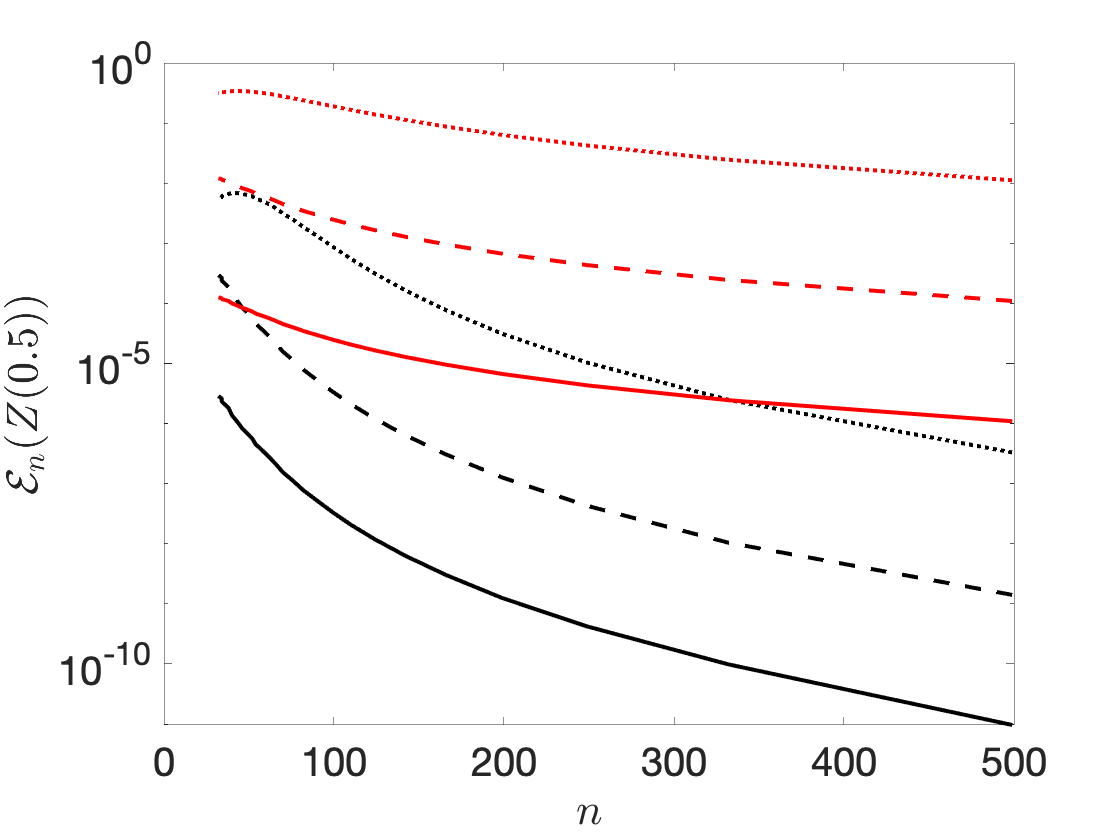} 
	\vspace{-1.5mm} 
	\caption{The results for model 1 (black) and model 2 (red) 
		for the first example \eqref{eq:models1} 
		with integral observations (left) and 
		point observations (right). 
		Solid lines correspond to $\delta= 1$, 
		dashed to $\delta = 10$, 
		and dotted to $\delta = 100$.}
	\label{fig1}
\end{figure}
We can clearly see that model 2 does not 
provide asymptotically optimal linear prediction in this scenario, 
but model 1 does. 
This holds for each of the three different values of $\delta$, 
and is in line with our theoretical findings: 
Theorem~\ref{thm:kriging} (see also Table~\ref{tab:summary}) 
shows that only the model with misspecified $\kappa^2$ 
should provide asymptotically optimal linear prediction.  

As a second scenario, we let $h = \GP(\s_0)$ 
with $\s_0 = 0.5$ and compute predictions of $h$ 
based on observations of $\GP(\s)$ at $n$ locations 
$\s_1, \s_2,\ldots$ in $\cD=(0,1)$ chosen as $\s_{2j} = s_0+j\delta_o$ 
and ${\s_{2j-1} = s_0 - j\delta_o}$ for $j\in\bbN$. 
Here $\delta_o\in(0,\nicefrac{1}{2})$ is a constant  
that determines the distance between the observations.  
The only difference in the calculations in this case is 
that the matrix $\bobsmat$ now contains the elements 
$\obsmat_{\ell k} = \varphi_k(\s_{\ell-1})$.
We again compute 
predictions based on both misspecified models 
and use $\cE_n(\GP(\s_0))$ 
to measure the accuracy. 
The right panel of Figure~\ref{fig1} shows 
the results as functions of $n$ 
for the two different models 
and the three different values of $\delta$. 
We can now see that  model~2 has a larger error 
compared to model~1. 
However, also the error of model~2 
seem to converge to zero in this case, 
although at a worse rate compared to model~1. 

We recall that Theorem~\ref{thm:kriging} 
in Subsection~\ref{subsec:wm-D:kriging} 
specifies necessary and sufficient conditions 
for \emph{uniform} asymptotic optimality of linear prediction 
based on misspecified Whittle--Mat\'ern models. 
Here, uniformity means that the supremum of 
$\cE_n(h)$ 
taken over all  ${h\in\cH_{-n}=\{h\in\cH : \pE[(h_n-h)^2] > 0 \}}$ 
should converge to zero as $n\to\infty$, 
see \eqref{eq:prop:A-kriging-1}. 
In particular, 
the outcomes of the second example, 
where one specific $h\in\cH$ is fixed, 
do not contradict 
the results of Subsection~\ref{subsec:wm-D:kriging}. 
Interestingly, they suggest, however, that 
the conditions of Theorem~\ref{thm:kriging} 
and of \cite[][Assumption~3.3]{bk-kriging}
are \emph{not necessary} for 
asymptotically optimal linear prediction 
when predicting the random field 
at a single location. 

\subsection{The effect of the smoothness parameter}

We again consider the Whittle--Mat\'ern model \eqref{eq:model_ex2} 
on $\cD= (0,1)$, this time for $\ac\equiv 1$ and $\beta\in\{1,2,3\}$. 
For the approximation of the solution $\GP$, 
we use a finite element discretization
with $N=2000$ equally spaced 
hat functions as basis functions. 
For $\beta=2$ and $\beta=3$ 
we follow the iterative approach of \cite{lindgren11} 
and~\cite{BK2020rational}. 
That is, we replace the matrix 
$\mv{L}$ (corresponding to the operator $L$ for $\beta=1$) 
by $\mv{L}\mv{M}^{-1}\mv{L}$ when $\beta=2$ 
and by $\mv{L}\mv{M}^{-1}\mv{L}\mv{M}^{-1}\mv{L}$ when $\beta=3$ 
(corresponding to the operators $L^2$ and $L^3$, respectively). 

As a baseline model, we consider  \eqref{eq:model_ex2} 
with $\ac\equiv 1$, 
${\tau 
= 
(4\pi)^{-\nicefrac{1}{4}} 
\kappa^{\nicefrac{1}{2}-2\beta} 
( \Gamma(2\beta-\nicefrac{1}{2}) / \Gamma(2\beta) )^{\nicefrac{1}{2}}}$\!, 
and $\kappa^2 \equiv 100(4\beta-1)$, 
so that the model has practical correlation range $0.2$ 
and variance close to $1$ at the center of the domain, 
cf.~\eqref{eq:spde_var}. 
For $\beta\in\{1,2,3\}$, 
we consider two different models 
of the form~\eqref{eq:model_ex2}, 
where we keep $\ac\equiv 1$ and the constant $\tau$ 
fixed to their correct values but misspecify $\kappa^2$ by  
\begin{equation}
	\kappa^2(\s)
	= 
	100(4\beta-1) 
	\cdot 
	\begin{cases}
		1 - 1.5\s^2 + \s^3 & \text{for model 1},\\
		1 + \s - 1.5 \s^3   & \text{for model 2}, 
	\end{cases}
	\qquad \s\in\clos{\cD}=[0,1], 
\end{equation}
see the right panel of Figure~\ref{fig2}.
In both cases, $\kappa^2(\s)$ takes the correct value 
at $\s=0$ 
and half of the correct value at $\s=1$. 
The main difference between the two models is 
that the derivative of $\kappa^2$ vanishes on the boundary 
for model~1, but not for model~2. 
Because of this, model~1 induces the same boundary conditions 
as the baseline model, 
whereas model~2 changes the boundary condition when $\beta=3$, 
cf.~\eqref{eq:Hdot-Sobolev} and Theorem~\ref{thm:CM}. 
From the results of Subsections~\ref{subsec:wm-D:summary} 
and~\ref{subsec:wm-D:kriging} 
we know that the behavior of $\kappa^2$ 
for the two alternative models implies that model~1 
will provide uniformly asymptotically optimal linear prediction 
for all values of $\beta\in\{1,2,3\}$ whereas model~2 only will do so 
for $\beta=1$ and $\beta=2$
(see Table~\ref{tab:summary}, 
where $c=1$ and $\delta_{c,\kappa^2}$ 
has a derivative that does not vanish at the boundary 
$\s\in\{0,1\}$). 

\begin{figure}[t]
	\includegraphics[width=0.49\linewidth]{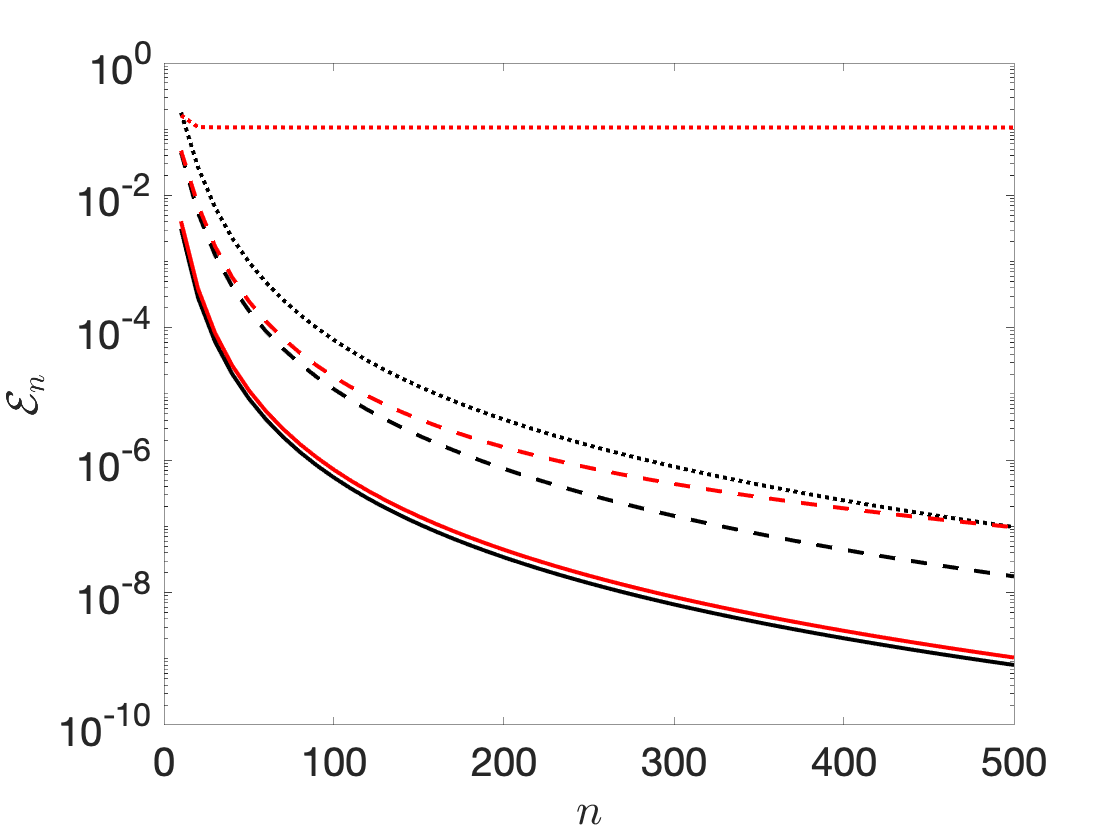}
	\hfill 
	\includegraphics[width=0.49\linewidth]{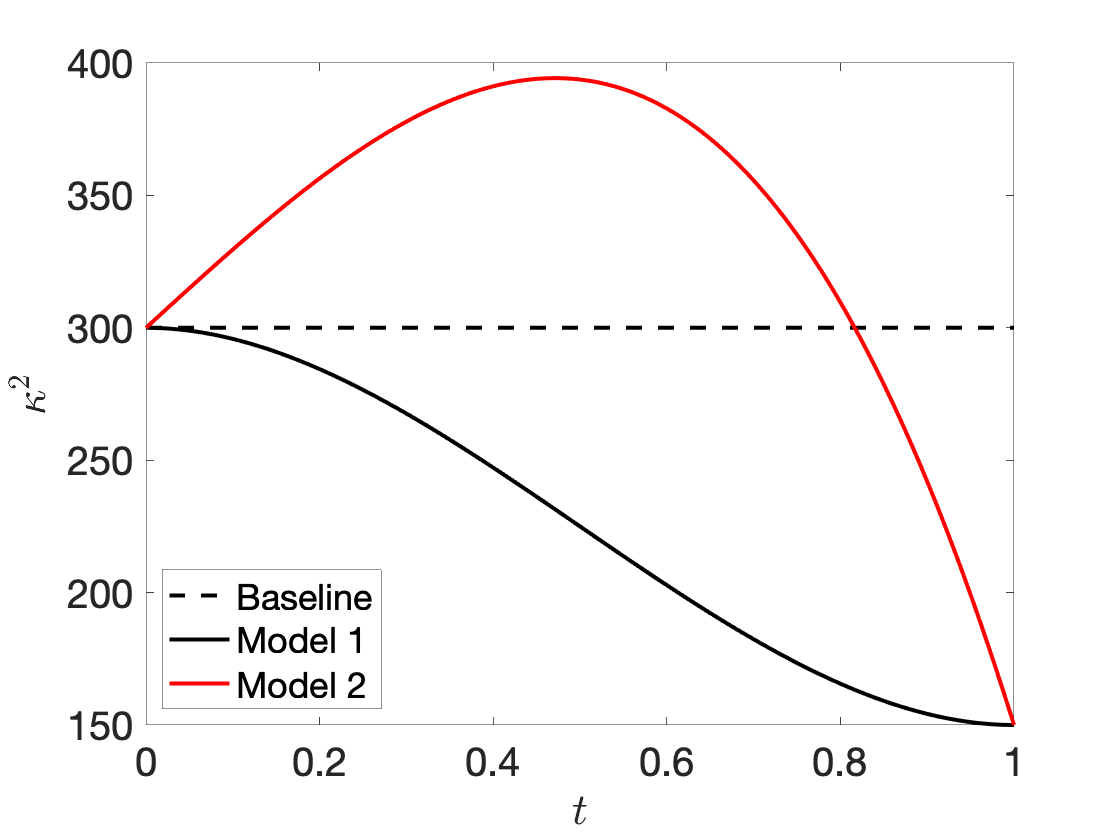} 
	\caption{Left: the results for model 1 (black) and model 2 (red) 
		in the second example, with $\beta=1$ (solid), 
		$\beta=2$ (dashed), and $\beta=3$ (dotted).  
		Right: $\kappa^2$ for the two models when $\beta=1$.}
	\label{fig2}
\end{figure}

To investigate this, 
we again consider predicting the integral values 
$I_\ell = (\GP, e_\ell )_{L_2(\cD)}$. 
Given observations of 
$I_1,\ldots, I_n$ we predict $h = I_\ell$ 
for $\ell\in\{n+1, \ldots, N\}$ 
and compute $\cE_{I,n}^{\max}$  
as the largest error among these predictions, see \eqref{eq:En_max}.  
Figure~\ref{fig2} shows $\cE_{I,n}^{\max}$ 
as a function of 
$n$ for both misspecified models 
in the three cases $\beta\in\{1,2,3\}$. 
The figure verifies that  both misspecified models 
provide asymptotically optimal predictions when $\beta\in\{1,2\}$ 
but, for $\beta=3$, only the predictions based on model~1 
behave asymptotically optimal.

\section{Discussion}\label{section:discussion}

In the general setting of Gaussian measures 
with fractional-order covariance operators 
on separable Hilbert spaces, 
we have derived necessary and sufficient conditions for 
\begin{enumerate*}[label={\normalfont\Roman*.}]
	\item\label{enum:discussion1} equivalence of Gaussian measures 
	in Proposition~\ref{prop:A-equiv}, and 
	\item\label{enum:discussion2} uniform asymptotic optimality 
	of linear (kriging) prediction 
	based on  
	misspecified Gaussian measures 
	in Proposition~\ref{prop:A-kriging}. 
\end{enumerate*} 
These conditions are formulated in terms of the  
non-fractional 
base operators,  
and are therefore in many situations 
simpler to verify than those for 
\ref{enum:discussion1}~as given by the Feldman--H\'ajek theorem 
and those for 
\ref{enum:discussion2} as stated in \citep[][Assumption~3.3]{bk-kriging}. 
As a first explicit example, we have applied 
these results to classical Whittle--Mat\'ern fields, 
see~Corollary~\ref{cor:stat-wm}. 

In the second part of the manuscript, 
we adopted the general results to derive necessary and sufficient conditions 
for \ref{enum:discussion1} and \ref{enum:discussion2} 
in terms of the (possibly function-valued) parameters 
of generalized Whittle--Mat\'ern fields 
on bounded Euclidean domains, 
see \eqref{eq:Lbeta}, \eqref{eq:L-div} and~\eqref{eq:def:mu}. 
The outcomes of Theorems~\ref{thm:CM},~\ref{thm:equivalence} 
and~\ref{thm:kriging} cover the whole 
range of admissible fractional orders $\beta\in(\nicefrac{d}{4},\infty)$ 
except for the cases 
$2\beta\in\frakE$, i.e., 
$\beta\in\{k+\nicefrac{1}{4}:k\in\bbN\}$, see also Table~\ref{tab:summary}. 
For ease of presentation, we refrained from detailing 
the results for $2\beta\in\frakE$ and we will now briefly comment 
on this situation. 
In the case that $r\in\frakE$ belongs to the 
discrete exception set \eqref{eq:exception}, 
on $\Hdot{r}$ the Sobolev norm $\norm{\,\cdot\,}{H^r(\cD)}$ 
will not be equivalent to the norm 
$\norm{\,\cdot\,}{r,L}=\norm{L^{\nicefrac{r}{2}} \,\cdot\,}{L_2(\cD)}$ 
defined through the 
fractional power operator~$L^{\nicefrac{r}{2}}$\!, 
as the topology on $\Hdot{r}$ 
is strictly finer than that on $H^r(\cD)$. 
It is well-known 
(see e.g.\ \cite[][Theorem~11.7 in Chapter~1]{LionsMagenesI})
that, for instance, for $r=\nicefrac{1}{2}$ 
the norm $\norm{\,\cdot\,}{\nicefrac{1}{2},L}$ is equivalent 
to the norm 
\[
	\norm{v}{H^{1/2}_{00}(\cD)} 
	:= 
	\bigl( 
	\norm{v}{H^{1/2}(\cD)}^2 
	+ 
	\norm{\rho^{-\nicefrac{1}{2}} v}{L_2(\cD)}^2 
	\bigr)^{\nicefrac{1}{2}} , 
\]
where $\rho\in C^\infty(\clos{\cD})$ is 
a function which is 
positive in the interior $\cD$ 
and for which the limit 
$\lim_{\s\to\s_0} \frac{\rho(\s)}{\mathrm{dist}(\s,\partial\cD)}$ exists 
and is not zero 
for all $\s_0\in\partial\cD$, 
where  $\mathrm{dist}(\s,\partial\cD)$ denotes 
the distance of $\s$ to the boundary $\partial\cD$. 
For example, in the case that 
$\beta=\nicefrac{5}{4}$ and 
$2\beta=\nicefrac{5}{2}\in\frakE$, 
we therefore expect the Cameron--Martin spaces 
of the Gaussian Whittle--Mat\'ern 
measures $\mu(0;\beta,\ac,\kappa)$ 
and $\mu(0;\beta,\ac,\kappat)$, see~\eqref{eq:def:mu}, to be 
isomorphic with equivalent norms 
for any choice of the coefficients $\kappa,\kappat\in C^\infty(\clos{\cD})$ 
since 
$\delta_{\kappa^2} = \kappat^2-\kappa^2\in C^\infty(\clos{\cD})$ 
and 
$\rho^{-\nicefrac{1}{2}} \delta_{\kappa^2} v \in L_2(\cD)$ 
for all $v \in \Hdot{\nicefrac{5}{2}}\cup\hdot{\nicefrac{5}{2}}{\Lt}
\subset H^{\nicefrac{1}{2}}_{00}(\cD)$. 
For 
$\beta=\nicefrac{9}{4}$, 
we expect this to hold 
if and only if 
\[ 
	\bigl( \kappa^2 - \nabla\cdot(\ac\nabla) \bigr) 
	\bigl( \delta_{\kappa^2} v \bigr) \in H^{\nicefrac{1}{2}}_{00}(\cD)
	\qquad\text{and}\qquad 
	\bigl( \kappat^2 - \nabla\cdot(\ac\nabla) \bigr) 
	\bigl( \delta_{\kappa^2} \widetilde{v} \bigr) \in H^{\nicefrac{1}{2}}_{00}(\cD), 
\]
for all $v\in\Hdot{\nicefrac{9}{2}}$ and 
all $\widetilde{v}\in\hdot{\nicefrac{9}{2}}{\Lt}$. 
Similarly, as in 
\eqref{eq:cond-delta-kappa}, \eqref{eq:condition-delta-kappa} 
this results in the condition 
$\rho^{-\nicefrac{1}{2}} (\ac\nabla\delta_{\kappa^2} ) \cdot\nabla v \in L_2(\cD)$ 
for all $v\in\Hdot{\nicefrac{9}{2}}\cup\hdot{\nicefrac{9}{2}}{\Lt}$. 
This means that $\ac\nabla\delta_{\kappa^2}$ 
has to satisfy a certain decay behavior towards to the boundary $\partial\cD$. 
Analogous conditions can also be derived 
for $\beta\in\{\nicefrac{13}{4},\nicefrac{17}{4}, \ldots\}$ 
and for the case that $\ac\neq\act$. 
Furthermore, although  
we have addressed only Gaussian measures in this work, 
the results for~\ref{enum:discussion2} 
extend to non-Gaussian processes, 
since the kriging predictor solely depends 
on the first two moments of the process. 

As a natural extension of the results of this work, 
generalized Whittle--Mat\'ern fields on manifolds or surfaces 
can be considered in future work.  
This extension is of relevance for practical applications in statistics, 
where for instance models on the sphere often play an important role. 
In fact, for a smooth surface $\cM$ without boundary, such as the sphere, 
the transition from the abstract results of Section~\ref{section:gm-on-hs} 
to Whittle--Mat\'ern fields on~$\cM$ 
should be more straightforward compared to the arguments used 
in Section~\ref{section:wm-D} for bounded Euclidean domains. 
This is suggested by the fact that on a smooth surface $\cM$ 
(and for smooth coefficients~$\ac,\kappa$)
the space $\Hdot{r}$ is isomorphic 
to the Sobolev space $H^r(\cM)$, 
and not to a proper subspace thereof \eqref{eq:Hdot-Sobolev} 
containing only functions which satisfy certain boundary 
conditions.

\begin{appendix}
	 
	\section{Function spaces, differential calculus and PDEs}
	\label{appendix:function-spaces}
	
	\subsection{Function spaces}
	\label{app:subsec:function-spaces}  
	
	Throughout this section, 
	let $\cD$  
	be a nonempty, connected, bounded and open domain
	in the Euclidean 
	space $\bbR^d$, $d\in\bbN$. 
	The closure of $\cD$ in $\bbR^d$ 
	is denoted by $\clos{\cD}$. 
	We assume that the boundary of $\cD$,  
	given by $\partial\cD=\clos{\cD}\setminus\cD$, 
	is Lipschitz continuous, 
	see Definition~\ref{def:smooth-bdry}.  
	We write 
	$\cD_0\Subset\cD$ 
	whenever $\clos{\cD}_0\subset\cD$. 
	
	In what follows, we introduce 
	several vector spaces 
	of real-valued functions 
	on $\cD$ or $\clos{\cD}$.

	\subsubsection{Continuous, continuously differentiable, and smooth functions} 
	\label{app:subsubsec:Ck} 
	
	For a $d$-tuple  
	$\boldsymbol{\alpha} = (\alpha_1,\ldots,\alpha_d) \in \bbN_0^d$ 
	of nonnegative integers, 
	we call $\boldsymbol{\alpha}$ a multi-index 
	and define the differential operator 
	$D^{\boldsymbol{\alpha}} 
	:= \tfrac{\partial^{\alpha_1}}{\partial \s_1^{\alpha_1}} 
	\cdots 
	\tfrac{\partial^{\alpha_d}}{\partial \s_d^{\alpha_d}}$. 
	In this case, 
	we say that $D^{\boldsymbol{\alpha}}$ 
	has order $|\boldsymbol{\alpha}| := \sum_{j=1}^d \alpha_j$.  
	Note, in particular, that $D^{(0,\ldots,0)} f = f$. 
	Furthermore, the gradient $\nabla$ 
	and the Laplace operator~$\Delta$ 
	(aka.\ Laplacian) 
	are given by 
	$\nabla 
	:= 
	\bigl( 
	\tfrac{\partial}{\partial \s_1} , 
	\ldots, 
	\tfrac{\partial}{\partial \s_d}
	\bigr)^\top$ 
	and 
	$\Delta := \nabla \cdot \nabla 
	= \sum_{j=1}^d \tfrac{\partial^2}{\partial \s_j^2}$, 
	respectively. 
	
	For any nonnegative integer $k\in\bbN_0$ 
	we then introduce the function spaces 
	\begin{align*} 
	C^k(\cD)  
	&:= 
	\{f\from \cD\to\bbR : 
	D^{\boldsymbol{\alpha}} f \text{ is continuous on $\cD$ for all } 
	|\boldsymbol{\alpha}|\leq k\}, 
	\\  
	C^k(\clos{\cD})  
	&:= 
	\{f\from \cD\to\bbR : 
	D^{\boldsymbol{\alpha}} f \text{ is uniformly continuous on $\cD$ for all }
	|\boldsymbol{\alpha}|\leq k\}. 
	\end{align*}
	Note that, since $\cD\subset\bbR^d$ 
	is assumed to be bounded, 
	the definition of $C^k(\clos{\cD})$ implies that 
	$D^{\boldsymbol{\alpha}} f$ has a continuous extension  
	to $\clos{\cD}$ for any $|\boldsymbol{\alpha}|\leq k$. 
	Furthermore, we define 
	$C^\infty(\cD) 
	:= 
	\bigcap\nolimits_{k\in\bbN_0} C^k(\cD)$, 
	$C^\infty(\clos{\cD}) 
	:= 
	\bigcap\nolimits_{k\in\bbN_0} C^k(\clos{\cD})$, 
	and we let the vector space $C_c^\infty(\cD)$ 
	consist of those functions in $C^\infty(\cD)$ 
	that have compact support. 
	Here, we say that a function 
	$f\from \cD\to\bbR$ 
	has compact support 
	if its support, defined by 
	$\operatorname{supp}(f)
	:= \clos{\{s\in\cD:f(s)\neq 0\}}$, 
	satisfies 
	$\operatorname{supp}(f) \Subset \cD$.

	\subsubsection{Lebesgue spaces} 
	\label{app:subsubsec:Lp} 
	
	For $p\in[1,\infty)$, we let $L_p(\cD)$ 
	be the space of all Lebesgue measurable,  
	real-valued functions $f$ defined on $\cD$ 
	for which 
	\begin{equation}\label{eq:def:Lp-norm}
		\textstyle 
		\| f \|_{L_p(\cD)} 
		:=
		\left(
		\int_\cD 
		|f(\s)|^p \, \rd \s
		\right)^{\nicefrac{1}{p}} 
		<\infty. 
	\end{equation}
	Identifying functions in $L_p(\cD)$ which  
	are equal almost everywhere in $\cD$  
	renders $L_p(\cD)$ 
	a vector space
	containing all the equivalence classes 
	of Lebesgue measurable functions 
	satisfying \eqref{eq:def:Lp-norm}, 
	such that two functions are equivalent 
	whenever they are equal almost everywhere in~$\cD$. 
	Furthermore, for every $p\in[1,\infty)$, 
	$L_p(\cD)$ is a Banach space 
	when equipped with the norm  
	$\|\,\cdot\,\|_{L_p(\cD)}$ 
	in~\eqref{eq:def:Lp-norm}; 
	in the case that $p=2$ 
	this norm is induced by an 
	inner product that renders 
	$L_2(\cD)$ a Hilbert space. 
	
	A Lebesgue measurable function 
	$f\from\cD\to\bbR$ on $\cD$ 
	is called essentially bounded if there exists 
	a constant $K\in\bbR_+$ such that $|f(\s)|\leq K$ 
	holds for almost all $\s\in\cD$. 
	The greatest lower bound 
	of such constants is called the essential supremum 
	of $|f|$ and is denoted by 
	$\operatorname{ess\,sup}_{\s\in\cD}|f(\s)|$. 
	We let $L_\infty(\cD)$ be the vector space 
	of all Lebesgue measurable functions 
	which are essentially bounded on $\cD$,  
	where again functions are identified 
	if they are equal almost everywhere in $\cD$. 
	This vector space is a Banach space 
	with respect to the norm 
	\begin{equation}\label{eq:def:Linf-norm} 
		\| f \|_{L_\infty(\cD)} 
		:= 
		\operatorname{ess\,sup}_{\s\in\cD}|f(\s)| 
		< \infty.  
	\end{equation} 
	
	It is common to ignore the distinction 
	between a function and its equivalence 
	class. We thus write $f=0$ in $L_p(\cD)$ 
	whenever $f(\s)=0$ for almost all $\s\in\cD$, 
	and $f\in L_p(\cD)$ whenever 
	$f$ satisfies \eqref{eq:def:Lp-norm} if $p\in[1,\infty)$ 
	or \eqref{eq:def:Linf-norm} if $p=\infty$. 
	For more details on the Lebesgue spaces $L_p(\cD)$, 
	see e.g.~\cite[Chapter~2]{AdamsSobolev2003}. 
	
	In the context of boundary value problems, 
	also Lebesgue spaces of functions 
	defined on the boundary play an important role. 
	For $p\in[1,\infty)$, we define  
	$L_p(\partial\cD)$ as the vector space containing 
	all (equivalence classes of) real-valued 
	functions $f\from\partial\cD\to\bbR$ such that 
	$\| f \|_{L_p(\partial\cD)} 
	:= 
	\bigl( 
	\int_{\partial\cD} |f(\s)|^p \, \rd S 
	\bigr)^{\nicefrac{1}{p}} 
	< \infty$, 
	where $\rd S$ denotes the $(d-1)$-dimensional surface 
	measure on $\partial\cD$. 
	We refer, e.g., to \cite[Paragraph~5.35]{AdamsSobolev2003} 
	or \cite[Section~7.3]{LionsMagenesI} for a detailed definition 
	of the surface measure $\rd S$.

	\subsubsection{Sobolev spaces} 
	\label{app:subsubsec:Hr} 
	
	For a multi-index $\boldsymbol{\alpha}=(\alpha_1,\ldots,\alpha_d)$ 
	we call $v_{\boldsymbol{\alpha}} \in L_1(\cD)$ 
	the $\boldsymbol{\alpha}$-th weak derivative 
	of $u\in L_1(\cD)$ provided that 
	$v_{\boldsymbol{\alpha}}$ satisfies 
	\[
		\textstyle 
		\int_{\cD} u(s) D^{\boldsymbol{\alpha}} \phi(s) \, \rd s
		= 
		(-1)^{|\boldsymbol{\alpha}|} 
		\int_{\cD} 
		v_{\boldsymbol{\alpha}}(s) \phi(s) 
		\, \rd s 
		\quad 
		\forall \phi \in C^\infty_c(\cD). 
	\]
	In this case, we write 
	$D^{\boldsymbol{\alpha}} u = v_{\boldsymbol{\alpha}}$ 
	and understand the derivative in weak sense.\pagebreak 
	
	For an integer $k\in\bbN$, 
	the Sobolev space $H^k(\cD)$ is then defined 
	as follows, 
	\[
		H^k(\cD) 
		:= 
		\left\{ 
		u \in L_2(\cD) : 
		D^{\boldsymbol{\alpha}} u 
		\text{ exists in the weak sense and } 
		D^{\boldsymbol{\alpha}} u \in L_2(\cD) 
		\text{ for all } 
		|\boldsymbol{\alpha}|\leq k 
		\right\}, 
	\]
	equipped with the norm 
	$\| u\|_{H^k(\cD)}^2  
	:= 
	\sum\nolimits_{0\leq |\boldsymbol{\alpha}| \leq k} 
	\| D^{\boldsymbol{\alpha}} u \|_{L_2(\cD)}^2$ 
	and the corresponding inner product. 
	Further, we set $H^0(\cD) := L_2(\cD)$ 
	and we let $H^1_0(\cD)\subset H^1(\cD)$ be the 
	closure of $C^\infty_c(\cD)$ in $H^1(\cD)$. 
	
	We proceed with the definition of 
	$H^r(\cD)$ for general values of $r\in\bbR_+$, 
	following \cite[Section~2]{DiNezza2012} 
	and \cite[Section~1.11.4]{Yagi2010}. 
	To this end, we first introduce, for $\theta\in(0,1)$, the 
	Gagliardo seminorm 
	\[
		|u|_{H^\theta(\cD)} 
		:= 
		\biggl( 
		\int_{\cD} \int_{\cD} 
		\frac{ |u(\s) - u(t)|^2 }{\| \s - t \|_{\bbR^d}^{d+2\theta} } 
		\, \rd \s \, \rd t 
		\biggr)^{\nicefrac{1}{2}} 
	\]
	and define, for a non-integer $r\in\bbR_+\setminus\bbN$, the 
	Sobolev--Slobodeckij space $H^r(\cD)$ by 
	\[
		H^r(\cD) 
		:= 
		\bigl\{
		u \in H^{\lfloor r\rfloor}(\cD) : 
		|D^{\boldsymbol{\alpha}} u|_{H^{r-\lfloor r \rfloor}(\cD)} < \infty
		\text{ for all } 
		|\boldsymbol{\alpha}| = \lfloor r\rfloor
		\bigr\}. 
	\]
	The norm 
	$\|u \|_{H^r(\cD)}^2
	:= 
	\| u \|_{L_2(\cD)}^2  
	+ 
	\sum_{0\leq|\boldsymbol{\alpha}|\leq  \lfloor r\rfloor} 
	| D^{\boldsymbol{\alpha}} u |_{H^{r-\lfloor r \rfloor}(\cD)}^2$ 
	and the corresponding inner product 
	render $H^r(\cD)$ a Hilbert space. 
	For any $r\in\bbR_+$, we let 
	$H^{-r}(\cD)$ denote the 
	dual space of $H^r(\cD)$ (after identifying $L_2(\cD)$ with 
	its dual). Note that these definitions yield 
	the continuous and dense embedding 
	$H^{r_2}(\cD) \hookrightarrow H^{r_1}(\cD)$ 
	for any $r_1,r_2\in\bbR$ with $r_1\leq r_2$. 
	
	Finally, we remark that there are various approaches in the literature 
	to define 
	the fractional-order Sobolev space 
	$H^r(\cD)$ for $r\in\bbR_+\setminus\bbN$. 
	For instance, 
	Adams and Fournier \cite{AdamsSobolev2003} 
	or 
	Lions and Magenes~\cite{LionsMagenesI} 
	define these spaces 
	(for complex-valued functions)
	by means of complex interpolation. 
	It is important to note that these definitions 
	are all equivalent in the sense that they yield 
	equivalent norms. 
	
	We refer the reader, e.g., to
	\cite[Chapters~3 and~7]{AdamsSobolev2003}, 
	\cite[Section~2]{DiNezza2012}, 
	\cite[Chapter~1, Section~9]{LionsMagenesI} 
	or 
	\cite[Sections~1.11.4/5]{Yagi2010} 
	for more details on (fractional-order) 
	Sobolev spaces.

	\subsection{Differential calculus} 
	\label{app:subsec:diff-calc} 
	
	\subsubsection{The trace \texorpdfstring{theorem on $H^r(\cD)$}{theorem}} 
	\label{app:subsubsec:trace} 
	
	The following definition of a smooth boundary $\partial\cD$ 
	is taken from \cite[Appendix~C.1]{Evans1997}. 
	
	\begin{definition}\label{def:smooth-bdry}
		Let $k\in\bbN$. 
		The boundary $\partial\cD$ of $\cD$ is said to be of class $C^k$ 
		if for every $\s_0\in\partial\cD$ 
		there exist a radius 
		$r\in\bbR_+$ and a $k$-times continuously 
		differentiable function $\gamma\from \bbR^{d-1}\to\bbR$ 
		such that, upon relabeling and reorienting 
		the coordinate axes if necessary, 
		we have 
		\begin{equation}\label{eq:def:smooth-bdry} 
			\cD \cap \clos{B}(\s_0,r) 
			= 
			\bigl\{ \s \in \clos{B}(\s_0,r) : \s_d > \gamma(\s_1,\ldots,\s_{d-1}) \bigr\}. 
		\end{equation}
		Here, $\clos{B}(\s_0,r)$ denotes the closed 
		ball in $\bbR^d$ with center 
		$\s_0$ and radius $r$. 
		We call the boundary of $\cD$ 
		Lipschitz (continuous) if \eqref{eq:def:smooth-bdry} 
		holds for a Lipschitz continuous function $\gamma$. 
		Finally, $\partial\cD$ is 
		smooth, or of class $C^\infty$\!, 
		if $\partial\cD$ is of class $C^k$ for every $k\in\bbN$. 
	\end{definition} 
	
	\begin{remark}\label{rem:onormal} 
		In the case that $\partial\cD$ is of class $C^1$\!, 
		the outward pointing unit normal vector field 
		is defined along $\partial\cD$ and 
		denoted by 
		$\partial\cD\ni \s_0 \mapsto \onormal(\s_0) 
		= (\mathrm{n}_1(\s_0),\ldots,\mathrm{n}_d(\s_0))^\top$\!. 
	\end{remark} 
	
	The next result is a general version of the trace theorem, 
	taken from 
	\cite[Theorem~9.4 in Chapter~1]{LionsMagenesI}. 
	It plays a crucial role for  
	characterizing the Cameron--Martin 
	spaces of generalized Whittle--Mat\'ern fields 
	in Subsection~\ref{subsec:wm-D:CM}. 
	
	\begin{theorem}\label{thm:trace} 
		Assume that $\partial\cD$ is smooth. 
		Then, the mapping 
		$u \mapsto
		\bigl\{ 
		\tfrac{\partial^j u}{\partial \onormal^j} 
		: j=0,1,\ldots,\ell
		\bigr\}$
		from $C^\infty(\clos{\cD})$ 
		to $[C^\infty(\partial\cD)]^\ell$ 
		extends by continuity 
		to a bounded linear mapping 
		\begin{equation}\label{eq:thm:trace} 
			\textstyle 
			u \mapsto
			\bigl\{ 
			\tfrac{\partial^j u}{\partial \onormal^j} 
			: j=0,1,\ldots,\ell
			\bigr\}
			\qquad 
			\text{of} 
			\qquad 
			H^r(\cD) 
			\to 
			\prod_{j=0}^\ell H^{r-j-\nicefrac{1}{2}}(\partial\cD), 
		\end{equation}
		where $\ell\in\bbN_0$ is the greatest 
		nonnegative integer such that 
		$\ell < r - \nicefrac{1}{2}$.  
		The mapping \eqref{eq:thm:trace} 
		is surjective. 
	\end{theorem} 
	
	\begin{remark}\label{rem:trace} 
		Theorem~\ref{thm:trace} shows, in particular, that 
		the trace maps $u\mapsto u|_{\partial\cD}$ 
		and $u\mapsto \tfrac{\partial u}{\partial \onormal}$ 
		have continuous extensions 
		as linear mappings from $H^r(\cD)$ 
		to $L_2(\partial\cD)$ 
		provided that $r\in(\nicefrac{1}{2},\infty)$ 
		and $r\in(\nicefrac{3}{2},\infty)$,
		respectively. 
	\end{remark} 
	
	\subsubsection{A divergence theorem} 
	\label{app:subsubsec:div-thm} 
	
	\begin{theorem}\label{thm:div-thm} 
		Assume that $\partial\cD$ 
		and $\ac\from\clos{\cD}\to\bbR^{d\times d}$ 
		are smooth, and let $\ac(\s)\in\bbR^{d\times d}$ be symmetric 
		for every $\s\in\clos{\cD}$. 
		Then,  
		for all $v_1, v_2 \in H^2(\cD)$,  
		\begin{equation}\label{eq:thm:div-thm}
		\scalar{ \nabla\cdot(\ac \nabla) v_1, v_2}{L_2(\cD)} 
		- 
		\scalar{ v_1, \nabla\cdot(\ac \nabla) v_2}{L_2(\cD)} 
		= 
		\int_{\partial\cD} 
		\bigl[ v_2 (\ac\nabla v_1 \cdot \onormal) - v_1 (\ac\nabla v_2 \cdot \onormal) \bigr] 
		\, \rd S. 
		\end{equation}
	\end{theorem} 
	
	\begin{proof} 
		For $v_1,v_2\in C^\infty(\clos{\cD})$, 
		\eqref{eq:thm:div-thm} 
		follows from the multidimensional 
		integration by parts formula, see e.g.\ 
		\cite[][Theorem~2 in Chapter~C.2]{Evans1997}. 
		For the general case $v_1,v_2 \in H^2(\cD)$ 
		we may approximate $v_1,v_2$ in $H^2(\cD)$ by smooth 
		functions $(\phi_j)_{j\in\bbN}$, $(\psi_j)_{j\in\bbN}$  
		in $C^\infty(\clos{\cD})$, 
		see e.g.~\cite[][Theorem~3 in Chapter~5.3]{Evans1997}. 
		By the trace theorem, 
		see Theorem~\ref{thm:trace} and Remark~\ref{rem:trace}, 
		also their traces 
		$(\phi_j|_{\partial\cD})_{j\in\bbN}$, 
		$(\psi_j|_{\partial\cD})_{j\in\bbN}$ 
		and
		$\bigl( \tfrac{\partial\phi_j}{\partial\onormal} \bigr)_{j\in\bbN}$, 
		$\bigl( \tfrac{\partial\psi_j}{\partial\onormal} \bigr)_{j\in\bbN}$
		converge in $L_2(\partial\cD)$ 
		so that the equality \eqref{eq:thm:div-thm} 
		for $v_1:=\phi_j$ and $v_2:=\psi_j$ shows that 
		for general $v_1, v_2 \in H^2(\cD)$ 
		by passing to the limit $j\to\infty$. 
	\end{proof} 
	
	\subsection{Elliptic second-order partial differential equations} 
	\label{app:subsec:pdes} 
	
	The purpose of this subsection is to 
	present some results on 
	the Dirichlet boundary value problem 
	\begin{equation}\label{eq:dir-pde} 
		\begin{cases}
			- \nabla \cdot(\ac \nabla u) + \kappa^2 u = f 
			&\text{in}\;\; \cD, \\
			\hspace*{2.57cm} u =0 
			&\text{on}\;\; \partial\cD. 
		\end{cases} 
	\end{equation}
	Here, 
	we suppose that 
	we are in the setting of Assumption~\ref{ass:a-kappa-D}, i.e., 
	the boundary $\partial\cD$ and 
	the coefficients 
	$\kappa$, $\ac = (\ac_{jk})_{j,k=1}^d$ 
	are smooth, 
	$\kappa,\ac_{jk}\in C^\infty(\clos{\cD})$, 
	and $\ac$ is symmetric and uniformly positive definite. 
	We remark that the boundary conditions in \eqref{eq:dir-pde}
	are referred to as homogeneous Dirichlet boundary 
	conditions ($u$ is vanishing at the boundary). 
	In what follows, we focus on existence and uniqueness 
	of a solution to \eqref{eq:dir-pde} 
	as well as its regularity in $H^r(\cD)$ for $r\in\bbR_+$. 
	
	To this end, we first note that 
	the differential operator in \eqref{eq:dir-pde} 
	induces 
	a symmetric, 
	continuous and coercive   
	bilinear form on~$H^1_0(\cD)$,    
	namely 
	\[ 
		a_L \from H^1_0(\cD)\times H^1_0(\cD) \to \bbR, 
		\qquad 
		a_L (u,v) 
		:= 
		\scalar{ \ac \nabla u, \nabla v }{L_2(\cD)}
		+ 
		\scalar{ \kappa^2 u, v }{L_2(\cD)}. 
	\] 
	Since $H^1_0(\cD)$ is a Hilbert space, 
	the Lax--Milgram theorem 
	is applicable which shows that, 
	for every $f\in H^1_0(\cD)^*$\!, 
	there exists a unique element $u\in H^1_0(\cD)$, 
	called the weak solution to~\eqref{eq:dir-pde}, 
	such that 
	$a_L(u,v) = \duality{f,v}{}$ holds 
	for all $v\in H^1_0(\cD)$. 
	In other words, 
	the operator associated with $a_L$, 
	\[
		L \from H^1_0(\cD) \to H^1_0(\cD)^*, 
		\qquad 
		\duality{Lu, v}{} 
		:= 
		a_L(u,v) 
		\quad 
		\forall u,v\in H^1_0(\cD), 
	\]
	is an isomorphism 
	as a linear mapping from $H^1_0(\cD)$ to its dual. 
	Moreover, by the Rellich--Kondrachov compactness theorem 
	(see~\cite[Theorem~6.3]{AdamsSobolev2003}  
	or~\cite[Theorem~1 and subsequent remark in Chapter~5.7]{Evans1997}) 
	the embedding $H^1_0(\cD)\hookrightarrow L_2(\cD)$ 
	is compact for any $d\in\bbN$,  
	denoted by $H^1_0(\cD) \overset{c}{\hookrightarrow} L_2(\cD)$, 
	and therefore 
	\[
		L^{-1} \from L_2(\cD) \to L_2(\cD) , 
		\qquad 
		L_2(\cD) 
		\hookrightarrow 
		H_0^1(\cD)^* 
		\overset{L^{-1}}{\longrightarrow} 
		H^1_0(\cD) 
		\overset{c}{\hookrightarrow} 
		L_2(\cD)
	\]
	is also compact. 
	On the space $H^2(\cD)\cap H^1_0(\cD)$ 
	the operator $L$ maps continuously to $L_2(\cD)$ 
	and, for all $u \in H^2(\cD)\cap H^1_0(\cD)$, 
	there exists a constant $C_u\in\bbR_+$ such that 
	\begin{equation}\label{eq:L-domain}
		\duality{Lu,v}{} 
		= 
		\scalar{Lu, v}{L_2(\cD)}
		= 
		\scalar{- \nabla \cdot(\ac \nabla u) + \kappa^2 u, v}{L_2(\cD)}	
		\leq C_u \|v\|_{L_2(\cD)} 
		\quad 
		\forall v \in H^1_0(\cD). 
	\end{equation}
	In fact, $H^2(\cD)\cap H^1_0(\cD)$ is 
	the smallest subspace of $H^1_0(\cD)$ 
	satisfying \eqref{eq:L-domain}. 
	Thus, 
	the domain of $L$, 
	when considered as an unbounded operator on $L_2(\cD)$, 
	is given by $H^2(\cD)\cap H^1_0(\cD)$, 
	where formally  
	\[
		L v  = - \nabla \cdot(\ac \nabla v) + \kappa^2 v,
		\qquad 
		v \in \scrD(L) = H^2(\cD) \cap H^1_0(\cD). 
	\]
	
	The question of interest in regularity theory 
	for \eqref{eq:dir-pde} is the following: 
	If we assume more regularity for $f=Lu$, e.g., 
	$f\in L_2(\cD)$ or $f\in H^r(\cD)$ for some $r\in \bbR_+$, 
	does this imply also more regularity for 
	the weak solution $u=L^{-1}f$?  
	The answer is ultimately related 
	to the regularity of the boundary and 
	the coefficients $\kappa,\ac$. 
	Since we assume that they are smooth, 
	the operator 
	${L^{-1} \from L_2(\cD)\to H^2(\cD)\cap H^1_0(\cD)}$ 
	is indeed bounded, see~\cite[Theorem~8.12]{GilbargTrudinger2001}. 
	More generally, for nonnegative integers $r\in\bbN_0$, 
	we have the following result on $H^{2+r}(\cD)$-regularity 
	taken from~\cite[Theorem~5.1 in Chapter~2]{LionsMagenesI}. 
	It is an essential component 
	for characterizing 
	the space $\Hdot{r} = \scrD\bigl( L^{\nicefrac{r}{2}} \bigr)$ 
	in Lemma~\ref{lem:Hdot-Sobolev}. 
	
	\begin{theorem}\label{thm:regularity} 
		Let $L\from\scrD(L) =H^2(\cD)\cap H^1_0(\cD) \to L_2(\cD)$ 
		be as described in this subsection, 
		and let $r\in\bbN_0$. Then, there is 
		$C_r\in\bbR_+$ (depending on $r$) 
		such that, for all 
		$u\in \scrD(L)$ with 
		$Lu\in H^r(\cD)$, 
		\[
			\| u \|_{H^{2+r}(\cD)} 
			\leq 
			C_r 
			\bigl(
			\| Lu \|_{H^r(\cD)}  
			+ 
			\| u \|_{H^{1+r}(\cD)} 
			\bigr). 
		\] 
	\end{theorem}

	\section{The Feldman--H\'ajek theorem}\label{appendix:feldman-hajek}
	
	In this section we 
	recall the Feldman--H\'ajek 
	theorem from \cite[Theorem~2.25]{daPrato2014} 
	that characterizes equivalence 
	of two Gaussian measures 
	on a Hilbert space 
	in terms of three necessary and sufficient 
	conditions. 
	For this, we let 
	$(E, \scalar{\,\cdot\,, \,\cdot\,}{E})$ 
	be a separable Hilbert space over $\bbR$ 
	with $\dim(E) = \infty$. 
	
	\begin{theorem}[Feldman--H\'ajek]\label{thm:feldman-hajek}
		Two Gaussian measures 
		$\mu = \normal(m,\cC)$ 
		and 
		$\mut = \normal(\mt,\CC)$ 
		on $E$ 
		(with positive definite covariance operators $\cC$ and $\CC$)
		are either equivalent or orthogonal.  
		They are equivalent 
		if and only if 
		the following three conditions 
		are satisfied:
		\begin{enumerate}[label = {\normalfont(\roman*)}, topsep=2pt]
			\item\label{fh-1} 
			The Cameron--Martin spaces 
			$\bigl( \cC^{\nicefrac{1}{2}} (E), 
				\scalar{\cC^{-\nicefrac{1}{2}} \, \cdot\,, \cC^{-\nicefrac{1}{2}} \,\cdot\,}{E}\bigr)\,$ 
			and 
			$\bigl( \CC^{\nicefrac{1}{2}} (E), 
				\scalar{\CC^{-\nicefrac{1}{2}} \, \cdot\,, \CC^{-\nicefrac{1}{2}} \,\cdot\,}{E}\bigr)$ 
			are norm equivalent spaces, 
			$\cC^{\nicefrac{1}{2}} (E) = \CC^{\nicefrac{1}{2}} (E) =: E_1$. 
			\item\label{fh-2} 
			The difference of the means 
			is an element of the Cameron--Martin space,  
			$m - \mt \in E_1$.
			\item\label{fh-3} 
			The operator 
			$\bigl(\cC^{\nicefrac{1}{2}} \CC^{-\nicefrac{1}{2}}\bigr)
			\bigl(\cC^{\nicefrac{1}{2}} \CC^{-\nicefrac{1}{2}}\bigr)^* - \id_{E}$ 
			is a Hilbert--Schmidt operator on $E$.  
		\end{enumerate}
	\end{theorem}
	
	We remark that 
	Theorem~\ref{thm:feldman-hajek}
	is a slight reformulation 
	of \cite[Theorem~2.25]{daPrato2014}: 
	Instead of the operator
	$\bigl( \CC^{-\nicefrac{1}{2}} \cC^{\nicefrac{1}{2}} \bigr) 
	\bigl( \CC^{-\nicefrac{1}{2}} \cC^{\nicefrac{1}{2}} \bigr)^* - \id_{E}$, 
	in~\ref{fh-3} 
	we require 
	$\bigl( \cC^{ \nicefrac{1}{2} } \CC^{-\nicefrac{1}{2}} \bigr) 
	\bigl( \cC^{\nicefrac{1}{2}} \CC^{-\nicefrac{1}{2}} \bigr)^* - \id_{E}$
	to be Hilbert--Schmidt on $E$. 
	Since $\cC^{\nicefrac{1}{2}} \CC^{-\nicefrac{1}{2}}$ 
	is the adjoint of $\CC^{-\nicefrac{1}{2}} \cC^{\nicefrac{1}{2}}$ 
	and since $\cC,\CC$ are assumed to be strictly positive definite, 
	these two conditions 
	are equivalent whenever~\ref{fh-1} holds, see
	\citep[Lemma~6.3.1(ii)]{Bogachev1998}. 
	
	The following lemma characterizes the conditions~\ref{fh-1} and \ref{fh-3} 
	of the Feldman--H\'ajek theorem and 
	of~\cite[Assumptions~3.3.I and~III]{bk-kriging}, respectively. 
	
	\begin{lemma}\label{lem:HS-compact-property} 
		Let $T\in\cL(E)$ be a bounded linear operator 
		on~$(E,  \scalar{\,\cdot\,, \,\cdot\,}{E})$. 
		\begin{enumerate}[label={\normalfont\Roman*.}, topsep=2pt]
			\item\label{lem:HS-compact-property-HS}  
			The following two statements are equivalent:
			\begin{enumerate}[label={\normalfont(\roman*)}, topsep=2pt] 
				\item $T$ is invertible on~$E$ 
				and $T T^* - \id_E$ is a Hilbert--Schmidt operator on $E$.
				\item $T = U(\id_E + S)$, 
				where $U\in\cL(E)$ is an orthogonal operator 
				and $S$ is a self-adjoint Hilbert--Schmidt operator on $E$  
				such that $\id_E + S$ is invertible.
			\end{enumerate} 
			\item\label{lem:HS-compact-property-compact} 
			The following two statements are equivalent:
			\begin{enumerate}[label={\normalfont(\roman*)}, topsep=2pt]  
				\item $T$ is invertible on~$E$ 
				and $T T^* - \id_E$ is a compact operator on $E$.
				\item $T = W(\id_E + K)$, where $W\in\cL(E)$ is an orthogonal 
				operator and $K$ is a self-adjoint compact operator on $E$ 
				such that $\id_E + K$ is invertible.
			\end{enumerate}
		\end{enumerate} 
	\end{lemma}
	
	\begin{proof} 
		Assertion~\ref{lem:HS-compact-property-HS}   
		is proven in \cite[Lemma~6.3.1(i)]{Bogachev1998}. 
		
		Assertion~\ref{lem:HS-compact-property-compact} 
		can be shown similarly: 
		Suppose that $T$ satisfies \ref{lem:HS-compact-property-compact}(ii), 
		then
		\[
		T T^*  
		= 
		W (\id_E + K) (\id_E + K)^* W^* 
		= 
		\id_E + 2 W K W^* + W K^2 W^* 	
		= 
		\id_E + K_0, 
		\] 	
		where 
		$K_0 = 2 W K W^* + W K^2 W^*$ 
		is compact on $E$, since $K\in\cK(E)$, $W,W^*\in\cL(E)$, and the space 
		of compact operators $\cK(E)$ forms a two-sided ideal in $\cL(E)$. 
		This shows that $T T^* - \id_E = K_0$ is compact on $E$.  
		Since $W$ is orthogonal 
		and $\id_E + K$ is boundedly invertible on $E$, 
		also $T\from E \to E$ is invertible with 
		$T^{-1} = (\id_E+ K)^{-1} W^*\in\cL(E)$. 
		
		Conversely, assume now that $T$ satisfies 
		\ref{lem:HS-compact-property-compact}(i). 
		Since the operator $T$ is boundedly invertible, 
		$T^*$ has the polar decomposition 
		$T^* = \widetilde{W} \sqrt{T T^*}$, 
		where $\widetilde{W} := T^{-1} \sqrt{T T^*}$ 
		is orthogonal, 
		$\scalar{ \widetilde{W}^*\phi, \widetilde{W}^*\psi}{E} 
		= \scalar{ (T T^* ) (T^{-1})^* \phi, (T^{-1})^* \psi}{E} 
		= \scalar{\phi,\psi}{E}$. 
		We define the compact operator $K_0 := T T^* - \id_E$ 
		and write 
		$T^* = \widetilde{W} \sqrt{\id_E + K_0} 
		= \widetilde{W} (\id_E + \widetilde{K})$,  
		where the operator 
		$\widetilde{K} = \sqrt{\id_E + K_0} - \id_E\in\cL(E)$ 
		is self-adjoint. 
		Furthermore, since $K_0\in\cK(E)$, by 
		Lemma~\ref{lem:I+K} below also 
		$\widetilde{K}\in\cK(E)$ is compact. 
		Finally, we obtain that 
		$T = (T^*)^* = (\id_E + \widetilde{K}) \widetilde{W}^*
		= W (\id_E + K)$, 
		where  
		$W := \widetilde{W}^*$ 
		and  
		$K := \widetilde{W} \widetilde{K} \widetilde{W}^*$\!. 
		Here, $\id_E + K$ is invertible, 
		since $W$ is orthogonal and $T$ is invertible. 
	\end{proof}  
	
	\begin{lemma}\label{lem:I+K} 
		Suppose that 
		$S\in\cL_2(E)$ and $K\in\cK(E)$ are such 
		that $\id_E + S$ and $\id_E + K$ are self-adjoint and 
		nonnegative definite. 
		Then, 
		$\sqrt{\id_E + S} - \id_E \in \cL_2(E)$ 
		and 
		$\sqrt{\id_E + K} - \id_E \in \cK(E)$. 
	\end{lemma}   
	
	\begin{proof} 
		Taking the eigenbases of the operators 
		$S$ and $K$ with eigenvalues 
		$( s_j )_{j\in\bbN}$ 
		and $( k_j )_{j\in\bbN}$, 
		respectively, 
		where $s_j, k_j \in[-1,\infty)$ for all $j\in\bbN$,  
		we find that 
		$\sqrt{\id_E + S} - \id_E$ 
		and 
		$\sqrt{\id_E + K} - \id_E$ 
		have eigenvalues 
		$\widetilde{s}_j := \sqrt{1+s_j} - 1$ 
		and 
		$\widetilde{k}_j := \sqrt{1+k_j} - 1$, 
		respectively. 
		Clearly, 
		$\lim_{j\to\infty} \widetilde{k}_j = 0$ 
		and $\sqrt{\id_E + K} - \id_E \in \cK(E)$ follows. 
		Furthermore, since the sequence 
		$( s_j )_{j\in\bbN}$ is square-summable, 
		there exists $J_0\in\bbN$ such that 
		$| s_j | \leq \nicefrac{1}{2}$ for all $j > J_0$. 
		Then, by the mean value theorem, 
		applied for the function $t\mapsto \sqrt{t}$, 
		we obtain that 
		\[  
			\sum\limits_{j\in\bbN} \widetilde{s}_j^2 
			= 
			\sum\limits_{j=1}^{J_0} \widetilde{s}_j^2 
			+ 
			\sum\limits_{j>J_0} \bigl( \sqrt{1+s_j} - 1 \bigr)^2 
			\leq 
			\sum\limits_{j=1}^{J_0} \widetilde{s}_j^2  
			+ 
			\frac{1}{2} 
			\sum\limits_{j>J_0} s_j^2 
			< \infty, 
		\]
		which shows that $\sqrt{\id_E + S} - \id_E \in \cL_2(E)$. 
	\end{proof}

	\section{An auxiliary result on fractional operators}
	\label{appendix:A-alpha}

	\begin{lemma}\label{lem:A-alpha-difference} 
		Assume that $A\from\scrD(A) \subseteq E \to E$ and 
		$\At\from\scrD(\At)\subseteq E \to E$ 
		are two densely defined, self-adjoint, 
		positive definite linear operators 
		with compact inverses 
		on a separable Hilbert space $(E, \scalar{\,\cdot\,, \,\cdot\,}{E})$ 
		over $\bbR$ such that 
		$\scrD(\At) = \scrD(A)$.  
		If the operator 
		$B := \At - A \in \cL\bigl(\hdot{2\vartheta}{A}; \hdot{2\eta}{\At} \bigr)$  
		is bounded 
		for some 
		$\vartheta,\eta \in \bbR$, 
		with $\hdot{2\vartheta}{A}\!, \hdot{2\eta}{\At}$ 
		defined as in \eqref{eq:def:hdotA}, 
		then also 
		$\At^\gamma - A^\gamma \in 
		\cL\bigl( \hdot{2\theta}{A}; \hdot{2\rho}{\At} \bigr)$ 
		holds for every $\gamma\in(0,1)$,  
		$\rho\in[\eta, 1+\eta)$, 
		$\theta\in(\vartheta-1, \vartheta]$ 
		with $\gamma+(\rho-\eta)+(\vartheta-\theta) < 1$. 
	\end{lemma} 
	
	\begin{proof} 
		We first observe the following chain of identities 
		\begin{align}  
			( t\id_{E} + \At )^{-1} \At  
			- 
			A ( t\id_{E} + A )^{-1}  
			&= 
			( t\id_{E} + \At )^{-1} 
			\bigl( 
			\At ( t\id_{E} + A ) 
			- 
			( t\id_{E} + \At ) A
			\bigr) 
			( t\id_{E} + A )^{-1} 	
			\notag 
			\\
			&= 
			t \, ( t\id_{E} + \At )^{-1} 
			B 
			( t\id_{E} + A )^{-1} .  	
			\label{eq:At-A-resolvent-difference} 
		\end{align}  
		We let 
		$\psi\in\scrD\bigl( A^{1-\theta} \bigr)\cap E$ 
		so that $A^{-\theta}\psi \in\scrD(A) = \scrD(\At)$ 
		and 
		combine the latter  
		equality with 
		the following integral representation 
		of $A^\gamma \phi$ for $\gamma\in(0,1)$, 
		\begin{equation}\label{eq:Aalpha-integral} 
			A^\gamma \phi 
			= 
			\frac{\sin(\pi\gamma)}{\pi} 
			\int_0^\infty 
			t^{\gamma-1} 
			A \, ( t\id_E + A )^{-1} \phi 
			\, \rd t, 	
			\qquad \phi\in\scrD(A), 
		\end{equation} 
		which converges pointwise 
		in~$E$, 
		see \cite[Theorem~6.9 in Chapter~2.6]{Pazy1983}. 
		Specifically, we apply the representation 
		\eqref{eq:Aalpha-integral}  
		for $\At^\gamma \bigl( A^{-\theta} \psi \bigr)$ 
		and $A^{\gamma} \bigl( A^{-\theta} \psi \bigr)$. 
		Since both $A, \At$ are closed operators 
		and commute with their respective resolvents, 
		by invoking \eqref{eq:At-A-resolvent-difference} 
		this yields  
		\[ 
			\At^{\rho} \bigl( \At^\gamma - A^\gamma \bigr) A^{-\theta} \psi 
			= 
			\frac{\sin(\pi\gamma)}{\pi} 
			\int_0^\infty 
			t^{\gamma} \,
			( t\id_{E} + \At )^{-1} 
			\At^{\rho} B A^{-\theta}
			( t\id_{E} + A )^{-1} 
			\psi 
			\, \rd t, 	
		\] 
		and, therefore, we obtain the bound  
		\begin{equation}\label{eq:proof:Aalpha-difference} 
		\begin{split}  
			\bigl\| \At^{\rho} \bigl( \At^\gamma - A^\gamma \bigr) A^{-\theta} \psi \bigr\|_{E}	
			&\leq 
			\frac{\sin(\pi\gamma)}{\pi} \, 
			\bigl\| \At^{\eta} B A^{-\vartheta} \bigr\|_{\cL(E)} 
			\norm{\psi}{E}
			\\
			&\times
			\int_0^\infty 
			t^{\gamma} \,
			\bigl\| ( t\id_{E} + \At )^{-1} \At^{\rho-\eta} \bigr\|_{\cL(E)}
			\bigl\| ( t\id_{E} +   A )^{-1}    A^{\vartheta-\theta} \bigr\|_{\cL(E)}
			\, \rd t. 	
			\hspace{-0.3cm} 
		\end{split} 	
		\end{equation}
		By assumption 
		$B \in \cL\bigl( \hdot{2\vartheta}{A}; \hdot{2\eta}{\At} \bigr)$ 
		and, thus, 
		$\bigl\| \At^{\eta} B A^{-\vartheta} \bigr\|_{\cL(E)} 
		< \infty$. 
		Furthermore, 
		with $\lambda_1, \lambdat_1 \in\bbR_+$ 
		denoting the smallest 
		eigenvalues of $A$ and of $\At$, 
		respectively, we obtain 
		by self-adjointness and positivity of $A$ that,   
		for all $t\in\bbR_+$ and $\nu\in[0,1]$, 
		\[ 
			\textstyle 
			\| ( t\id_{E} + A )^{-1} A^\nu \|_{\cL(E)} 
			\leq \sup_{x\in[\lambda_1, \infty)} 
			\frac{ x^\nu }{ t+x } 
			\leq  
			\sup_{x\in[\lambda_1, \infty)} 
			\frac{1}{ (t+x )^{1-\nu} } 
			\leq 
			\min\bigl\{ t^{-(1-\nu)}, \lambda_1^{-(1-\nu)} \bigr\}, 
		\] 
		and by the same argument 
		$\| ( t\id_{E} + \At )^{-1} \At^\nu \|_{\cL(E)} 
		\leq \min\bigl\{ t^{-(1-\nu)}, \lambdat_1^{-(1-\nu)} \bigr\}$. 
		For this reason, we can bound the remaining integral 
		in \eqref{eq:proof:Aalpha-difference} 
		by 
		\begin{align*}
			\int_0^\infty 
			t^{\gamma} \,
			\bigl\| ( t\id_{E} + \At )^{-1} \At^{\rho-\eta} \bigr\|_{\cL(E)}
			&
			\bigl\| ( t\id_{E} +   A )^{-1}    A^{\vartheta-\theta} \bigr\|_{\cL(E)}
			\, \rd t 
			\\
			&\leq 
			\lambdat_1^{-1 + \rho - \eta} 
			\lambda_1^{-1 + \vartheta - \theta } 
			\int_0^1  
			t^\gamma 
			\, \rd t 
			+ 
			\int_1^\infty 
			t^{- 2 + \gamma + \rho - \eta + \vartheta - \theta} \, \rd t 
			\\ 
			&= 
			\lambdat_1^{-1 + \rho - \eta} 
			\lambda_1^{-1 + \vartheta - \theta } 
			(1 + \gamma)^{-1} 
			+ 
			\tfrac{1}{ 1 - \gamma - (\rho - \eta) - (\vartheta - \theta) }, 
		\end{align*} 
		since $\gamma + (\rho-\eta) + (\vartheta-\theta) \in (0,1)$ 
		by assumption. 
		We set 
		\[ 
			C_{\gamma,\rho,\eta,\vartheta,\theta} 
			:= 
			\tfrac{\sin(\pi\gamma)}{\pi} 
			\Bigl( 
			\lambdat_1^{-1 + \rho - \eta} 
			\lambda_1^{-1 + \vartheta - \theta } 
			(1 + \gamma)^{-1} 
			+ 
			\tfrac{1}{ 1 - \gamma - (\rho - \eta) - (\vartheta - \theta) }
			\Bigr) \in\bbR_+,
		\]  
		and conclude that 
		\[
			\bigl\| \At^{\rho} ( \At^\gamma - A^\gamma ) A^{-\theta} \psi \bigr\|_{E}	
			\leq  
			C_{\gamma,\rho,\eta,\vartheta,\theta} 
			\| B \|_{\cL\bigl( \hdot{2\vartheta}{A}; \hdot{2\eta}{\At} \bigr)}  
			\norm{\psi}{E} 
			\quad 
			\forall
			\psi\in\scrD\bigl( A^{1-\theta} \bigr) \cap E. 
		\]
		This completes the proof, since 
		$\scrD\bigl( A^{1-\theta} \bigr) \cap E = E$ 
		for all $\theta \in [1,\infty)$ and, 
		in the case that 
		$\theta\in(-\infty,1)$, 
		${\scrD\bigl( A^{1-\theta} \bigr) \cap E = \scrD\bigl( A^{1-\theta} \bigr)}$ 
		is dense in $E$. 
	\end{proof} 
	
	\begin{remark}\label{rem:A-alpha-difference}  
		Lemma~\ref{lem:A-alpha-difference} 
		shows in particular that if
		$\scrD(\At)=\scrD(A)$ and $\At-A\in\cL(E)$, 
		we obtain boundedness 
		of the difference of the fractional operators, 
		$\At^\gamma - A^\gamma \in \cL\bigl( \hdot{2\theta}{A}; \hdot{2\rho}{A} \bigr)$, 
		for all 
		$\rho\in[0,1)$, $\theta\in(-1,0]$ with 
		$\gamma + \rho - \theta < 1$ 
		(here, we used that 
		$\hdot{2\rho}{\At} = \hdot{2\rho}{A}$ for all $\rho\in[0,1]$ since 
		$\scrD(\At)=\scrD(A)$, 
		see also Lemma~\ref{lem:hdot-interpol} in Appendix~\ref{appendix:interpol}). 
	\end{remark}  
	
	\section{Proofs of \texorpdfstring{Lemmas~\ref{lem:beta-gamma} 
		and~\ref{lem:iff-A}}{Lemmas~2.1 and~2.2}}\label{appendix:proofs-lemmas}
	
	To exploit auxiliary results based on complex interpolation theory, 
	we will consider the complexification of 
	a separable Hilbert space $(E, \scalar{\,\cdot\,, \,\cdot\,}{E})$ over $\bbR$, 
	and 
	the complexification 
	of a linear operator $A\from \scrD(A)\subseteq E \to E$. 
	Specifically, we introduce the vector space 
	\begin{equation}\label{eq:complexification-space} 
		E_{\bbC} 
		:= 
		E + i E 
		= 
		\{\phi + i \psi : \phi, \psi \in E \}
	\end{equation}
	over $\bbC$, 
	which is a Hilbert space 
	with respect to the inner product 
	(see, e.g.\ \citep[][Section~5.2]{Luna2012}) 
	\begin{equation}\label{eq:complexification-ip} 
		\scalar{\phi + i \psi, \phi' + i \psi'}{E_\bbC} 
		:= 
		\scalar{\phi, \phi'}{E} 
		+ 
		\scalar{\psi, \psi'}{E} 
		+ 
		i \, [ 
		\scalar{\psi, \phi'}{E} 
		- 
		\scalar{\phi, \psi'}{E} ]. 
	\end{equation}
	In addition, we define the operator 
	$A_\bbC \from \scrD(A_\bbC) \subseteq E_\bbC \to E_\bbC$ 
	as the canonical linear extension of the operator $A$ 
	from Section~\ref{section:gm-on-hs}, i.e., 
	$A_\bbC (\phi + i \psi) := A\phi + i A\psi$.  
	Note that $A_\bbC$ is densely defined, 
	self-adjoint and positive definite on $E_\bbC$, 
	since $A$ has these properties on~$E$.   
	Furthermore, 
	the inverse ${A_\bbC^{-1}\from E_\bbC \to E_\bbC}$ 
	satisfies 
	$A_\bbC^{-1}(\phi+ i \psi) = A^{-1} \phi + i A^{-1} \psi 
	= \bigl(A^{-1}\bigr)_{\bbC}(\phi+ i \psi)$ 
	and 
	inherits compactness from $A^{-1} \in \cK(E)$. 
	Finally, we emphasize that the 
	eigenvectors $\{e_j\}_{j\in\bbN}$ of $A$ 
	(and thus of~$A_\bbC$) 
	form an orthonormal 
	basis for $E$ (over $\bbR$) 
	and for $E_\bbC$ (over $\bbC$).

	Similarly  
	as in \eqref{eq:Aalpha-integral}, 
	one may also represent the 
	fractional inverse 
	by a Bochner integral: 
	\begin{equation}\label{eq:Aalphainv-integral} 
		A^{-\theta} 
		= 
		\frac{\sin(\pi\theta)}{\pi} 
		\int_0^\infty 
		t^{-\theta} 
		( t\id_E + A )^{-1}  
		\, \rd t, 	
		\qquad 
		\theta\in(0,1), 
	\end{equation}
	with convergence of this integral 
	in the operator norm on~$\cL(E)$ 
	(see e.g.\ \citep[][Equation~(6.4) in Chapter~2.6]{Pazy1983}). 
	This integral representation 
	is exploited in the proof of Lemma~\ref{lem:beta-gamma}.  
	
	\begin{proof}[Proof of Lemma~\ref{lem:beta-gamma}] 	
		We first show that the 
		isomorphism property of 
		$\At^{\beta} A^{-\beta}$ on $E$ 
		implies that also 
		for every $\gamma\in[-\beta,\beta]$ 
		the operator 
		$\At^{\gamma} A^{-\gamma}$ 
		is an isomorphism on $E$, i.e., 
		$\At^\gamma A^{-\gamma} \from E \to E$ 
		is bounded and admits a bounded inverse 
		on $E$.  
		For $\gamma\in\{0,\beta\}$ this clearly holds. 
		Suppose next that $\gamma\in(0,\beta)$. 
		Note that 
		the isomorphism property of  
		$\At^\beta A^{-\beta} \from E \to E$ 
		implies that the Hilbert spaces 
		$\hdot{2\beta}{A}$ and~$\hdot{2\beta}{\At}$ 
		are isomorphic, with equivalent norms. 
		This means that there exist constants 
		$c_0, c_1 \in \bbR_+$ such that 
		$c_0\norm{v}{2\beta,A} 
		\leq \norm{v}{2\beta,\At} 
		\leq c_1 \norm{v}{2\beta,A}$ 
		holds 
		for all $v\in\hdot{2\beta}{A} = \hdot{2\beta}{\At}$\!. 
		Let $A_\bbC$ and $\At_\bbC$ denote\vspace{-1mm}\linebreak\pagebreak 
		
		\noindent 
		the complexifications 
		of the linear operators~$A$ and~$\At$   
		with respect to the complex Hilbert 
		space~$E_\bbC$,  
		see~\eqref{eq:complexification-space} 
		and~\eqref{eq:complexification-ip}. 
		Then, we obtain  
		$c_0\norm{v}{2\beta,A_\bbC} 
		\leq \norm{v}{2\beta,\At_\bbC} 
		\leq c_1 \norm{v}{2\beta,A_\bbC}$  
		for all $v\in\hdot{2\beta}{A_\bbC} = \hdot{2\beta}{\At_\bbC}$\!. 
		By Lemma~\ref{lem:hdot-interpol} in Appendix~\ref{appendix:interpol} 
		we have that 
		$\hdot{2\gamma}{A_\bbC} 
		= \bigl[ \hdot{0}{A_\bbC}, \hdot{2\beta}{A_\bbC} \bigr]_{\nicefrac{\gamma}{\beta}}$ 
		and 
		$\hdot{2\gamma}{\At_\bbC} 
		= \bigl[ \hdot{0}{\At_\bbC}, \hdot{2\beta}{\At_\bbC} \bigr]_{\nicefrac{\gamma}{\beta}}$ 
		for every $\gamma\in(0,\beta)$ 
		with isometries, 
		where 
		$[E_0, E_1]_\theta$ for $\theta\in[0,1]$ 
		denotes the complex interpolation 
		space between $E_0$ and~$E_1$, see 
		Definition~\ref{def:complex-interpol}. 
		Since furthermore  
		$\norm{\,\cdot\,}{0,A_\bbC} 
		= \norm{\,\cdot\,}{E_\bbC} 
		= \norm{\,\cdot\,}{0,\At_\bbC}$, 
		we then conclude 
		with \cite[][Theorem~2.6]{Lunardi2018}  	
		that, 
		for all $v\in\hdot{2\gamma}{A}=\hdot{2\gamma}{\At}$\!,\\[-4pt]
		\[
			c_0^{\nicefrac{\gamma}{\beta}} 
			\norm{v}{2\gamma,A} 
			= 
			c_0^{\nicefrac{\gamma}{\beta}} 
			\norm{v}{2\gamma,A_\bbC} 
			\leq 
			\norm{v}{2\gamma,\At_\bbC} 
			= 
			\norm{v}{2\gamma,\At}   
			\leq 
			c_1^{\nicefrac{\gamma}{\beta}} 
			\norm{v}{2\gamma,A_\bbC} 
			= 
			c_1^{\nicefrac{\gamma}{\beta}} 
			\norm{v}{2\gamma,A}, 
		\]
		i.e., 
		$\At^\gamma A^{-\gamma}\from E \to E$ is an isomorphism 
		for every $\gamma\in(0,\beta)$. 
		For $\gamma\in[-\beta,0)$, 
		the isomorphism property of $\At^\gamma A^{-\gamma}$ on $E$ 
		follows from that of $\At^{-\gamma} A^\gamma$\!, since we have
		$\At^\gamma A^{-\gamma} = 
		\bigl( ( \At^{-\gamma} A^{\gamma} )^{-1} \bigr)^*$\!. 

		Proof of \ref{lem:beta-gamma-HS}: 
		Define 
		the operators 
		$\cA := A^{2\beta}$ and  
		$\cAt := \At^{2\beta}$\!. 
		Since $A\from\scrD(A)\subseteq E \to E$ is densely defined, 
		self-adjoint, 
		positive definite, and 
		has a compact inverse, 
		the same applies to $\cA = A^{2\beta}$\!, 
		and we have that $\cA e_j = \alpha_j e_j$ 
		with $\alpha_j := \lambda_j^{2\beta}$\!, 
		see~\eqref{eq:def:Abeta}. 
		Thanks to the identity  
		\[
			(t\id_E + \cAt )^{-1} 
			- 
			(t\id_E + \cA)^{-1} 
			= 
			(t\id_E + \cAt )^{-1} 
			( \cA - \cAt )
			(t\id_E + \cA)^{-1} , 
		\]
		and the integral representation 
		of the fractional inverse \eqref{eq:Aalphainv-integral}
		applied for the operators 
		$\cAt^{-\theta}$\!,  
		$\cA^{-\theta}$ and $\theta \in (0,1)$, 
		we obtain the following equality in $\cL(E)$,  
		\begin{equation}\label{eq:lem:beta-gamma-proof_1} 
			\cAt^{-\theta} - \cA^{-\theta}
			= 
			\frac{\sin(\pi\theta)}{\pi}
			\int_0^\infty t^{-\theta} 
			(t\id_E + \cAt )^{-1} 
			(\cA - \cAt)
			(t\id_E + \cA)^{-1} 
			\, \rd t . 
		\end{equation} 
		Note that, 
		for fixed $j\in\bbN$, 
		we have 
		$\cA^{\nicefrac{(\theta+1)}{2}} (t\id_E +\cA)^{-1} e_j  
		= \alpha_j^{\nicefrac{(\theta+1)}{2}}
		(t + \alpha_j)^{-1} e_j$.  
		We now assume that 
		$\gamma\in(0,\beta)$ and choose the parameter 
		$\theta:=\nicefrac{\gamma}{\beta}\in(0,1)$. 
		For this choice, we have the equality
		$\At^{\gamma}
		\bigl( \At^{-2\gamma} - A^{-2\gamma} \bigr) 
		A^{\gamma} 
		= 
		\cAt^{\nicefrac{\theta}{2}} 
		\bigl( \cAt^{-\theta} - \cA^{-\theta} \bigr)
		\cA^{\nicefrac{\theta}{2}}$
		and by \eqref{eq:lem:beta-gamma-proof_1} 
		we obtain that, for all $j\in\bbN$,  
		\begin{align} 
			\bigl\| \At^{\gamma}
			( \At^{-2\gamma} - &A^{-2\gamma} ) 
			A^{\gamma} e_j 
			\bigr\|_{E} 
			\leq 
			\bigl\| \cAt^{-\nicefrac{1}{2}} 
			( \cA - \cAt \, )
			\cA^{-\nicefrac{1}{2}} e_j
			\bigr\|_E
			\notag
			\\
			&\quad\times
			\frac{\sin(\pi\theta)}{\pi}
			\int_0^\infty 
			t^{-\theta} 
			\bigl\|
			(t\id_E + \cAt)^{-1} 
			\cAt^{\nicefrac{(1+\theta)}{2}} 
			\bigr\|_{\cL(E)} \,
			\alpha_j^{\nicefrac{(1+\theta)}{2}}  
			(t + \alpha_j)^{-1} 
			\, \rd t 
			\notag
			\\
			&\leq 
			\bigl\| \cAt^{-\nicefrac{1}{2}}
			\cA^{\nicefrac{1}{2}}
			\bigr\|_{\cL(E)} 
			\bigl\| \cA^{-\nicefrac{1}{2}}
			(\cAt - \cA )
			\cA^{-\nicefrac{1}{2}} e_j
			\bigr\|_E 	
			\notag	
			\\
			&\quad\times
			\frac{\sin(\pi\theta)}{\pi}
			\, 
			\alpha_j^{\nicefrac{(1+\theta)}{2}}  
			\int_0^\infty 
			t^{-\theta} 
			(t + \alpha_j)^{-1} 
			\bigl\|
			(t\id_E + \cAt)^{-1} 
			\cAt^{\nicefrac{(1+\theta)}{2}}  
			\bigr\|_{\cL(E)} 
			\, \rd t . 
			\label{eq:lem:beta-gamma-proof_2}
		\end{align} 	
		Self-adjointness and positive definiteness 
		of the operator $\cAt=\At^{2\beta}$ imply that  
		the estimate 
		\[ 
			\textstyle 
			\bigl\| ( t\id_{E} + \cAt )^{-1} \cAt^\nu \bigr\|_{\cL(E)} 
			\leq \sup_{x\in[\widetilde{\alpha}_1, \infty)} 
			\frac{ x^\nu }{ t+x } 
			\leq  
			\sup_{x\in[\widetilde{\alpha}_1, \infty)} 
			\frac{1}{( t+x )^{1-\nu}}  
			\leq 
			\min\bigl\{ t^{-(1-\nu)}, \widetilde{\alpha}_1^{-(1-\nu)} \bigr\} 
		\] 
		holds for all $t\in\bbR_+$ and every $\nu\in[0,1]$, 
		where $\widetilde{\alpha}_1\in\bbR_+$ 
		denotes the smallest eigenvalue of $\cAt$. 
		Therefore, we can 
		bound the integral in \eqref{eq:lem:beta-gamma-proof_2} 
		as follows,  
		\begin{align*}
			\int_0^\infty 
			t^{-\theta} 
			(t + \alpha_j)^{-1} 
			\bigl\|
			(t\id_E + \cAt)^{-1} 
			\cAt^{\nicefrac{(1+\theta)}{2}}  
			\bigr\|_{\cL(E)} 
			\, \rd t
			&\leq 
			\int_0^\infty 
			t^{-\nicefrac{(1+\theta)}{2}} 
			(t + \alpha_j)^{-1} 
			\, \rd t
			\\
			&= 
			\tfrac{ \pi }{ \sin( \pi(1+\theta)/2 ) } 
			\,  
			\alpha_j^{-\nicefrac{(1+\theta)}{2}}, 
		\end{align*} 
		where we have used the identity 
		$\frac{\sin(\pi\eta)}{\pi} 
		\int_0^\infty t^{-\eta} (t+\lambda)^{-1}\, \rd t 
		= 
		\lambda^{-\eta}$ 
		which holds for every $\eta\in(0,1)$ 
		and $\lambda\in\bbR_+$, 
		see~\cite[Equation 3.194.4 on p.~316]{GradshteynRyzhik2007}, 
		combined with 
		the fact that $\nicefrac{(1+\theta)}{2} \in ( \nicefrac{1}{2}, 1)$, 
		since $\theta\in(0,1)$. 
		Inserting this bound 
		in \eqref{eq:lem:beta-gamma-proof_2}  
		and recalling that by definition 
		$\cAt^{-\nicefrac{1}{2}} 
		\cA^{\nicefrac{1}{2}} 
		= 
		\At^{-\beta} 
		A^{\beta}$ 
		and 
		$\cA^{-\nicefrac{1}{2}}
		( \cAt - \cA )
		\cA^{-\nicefrac{1}{2}} 
		= 
		A^{-\beta}
		\bigl( \At^{2\beta} - A^{2\beta} \bigr)
		A^{-\beta}$\!, 
		yields the estimate 
		\[ 
			\bigl\| \At^{\gamma}
			\bigl( \At^{-2\gamma} - A^{-2\gamma} \bigr) 
			A^{\gamma} e_j  
			\bigr\|_{E} 
			\leq 
			C_{\beta, \gamma} 
			\bigl\| A^{-\beta}
			\bigl( \At^{2\beta} - A^{2\beta} \bigr)
			A^{-\beta} e_j
			\bigr\|_E 
			\quad 
			\forall j\in\bbN,  
		\] 
		where 
		$C_{\beta, \gamma} 
		:= 
		\frac{\sin(\pi\gamma/\beta)}{\sin(\pi(1+\nicefrac{\gamma}{\beta})/2)} 
		\bigl\|
		\At^{-\beta} 
		A^{\beta}
		\bigr\|_{\cL(E)} \in \bbR_+$. 
		We thus conclude that 
		\begin{align*}
			\bigl\| \At^{\gamma}
			&\bigl( \At^{-2\gamma} - A^{-2\gamma} \bigr) 
			A^{\gamma}  
			\bigr\|_{\cL_2(E)}^2 	
			=
			\sum\limits_{j\in\bbN} 
			\bigl\| \At^{\gamma}
			\bigl( \At^{-2\gamma} - A^{-2\gamma} \bigr) 
			A^{\gamma} e_j  
			\bigr\|_{E}^2  	
			\\
			&\leq 
			C_{\beta,\gamma}^2 
			\sum\limits_{j\in\bbN} 
			\, 
			\bigl\| A^{-\beta}
			\bigl( \At^{2\beta} - A^{2\beta} \bigr)
			A^{-\beta} e_j
			\bigr\|_E^2 
			= 
			C_{\beta,\gamma}^2 
			\bigl\| A^{-\beta} \At^{2\beta} A^{-\beta} 
			- \id_E 
			\bigr\|_{\cL_2(E)}^2  
			<\infty, 		
		\end{align*}
		which combined with 
		the identity 
		$A^{\gamma}
		\At^{-2\gamma} 
		A^{\gamma} 
		- 
		\id_E 
		= 
		A^{\gamma}
		\bigl( \At^{-2\gamma} - A^{-2\gamma} \bigr) 
		A^{\gamma}$\!, and 
		\[ 
			\bigl\| A^{\gamma} 
			\bigl( \At^{-2\gamma} - A^{-2\gamma} \bigr) 
			A^{\gamma} 
			\bigr\|_{\cL_2(E)}  
			\leq 
			\bigl\| 
			A^{\gamma} \At^{-\gamma}
			\bigr\|_{\cL(E)} 	
			\bigl\| 
			\At^{\gamma}
			\bigl( \At^{-2\gamma} - A^{-2\gamma} \bigr) 
			A^{\gamma}  
			\bigr\|_{\cL_2(E)}  	
		<\infty 
		\] 
		shows that 
		$A^{\gamma}
		\At^{-2\gamma}  
		A^{\gamma} - \id_E \in\cL_2(E)$ 
		for all $\gamma\in(0,\beta)$ 
		and, hence, 
		$A^{-\gamma}
		\At^{2\gamma}  
		A^{-\gamma} - \id_E$ 
		is Hilbert--Schmidt on $E$  
		for every $\gamma\in(-\beta,0)$. 
		Clearly, 
		$A^{-\gamma}
		\At^{2\gamma}  
		A^{-\gamma} - \id_E \in\cL_2(E)$ 
		also holds for $\gamma \in \{0,\beta\}$. 
		As the isomorphism property of $\At^\gamma A^{-\gamma}\in\cL(E)$ 
		has already been established 
		for all $\gamma\in[-\beta,\beta]$, 
		applying 
		\citep[Lemma~6.3.1(ii)]{Bogachev1998} 
		completes the proof of~\ref{lem:beta-gamma-HS} 
		for the whole range $\gamma\in[-\beta,\beta]$. 
		
		Proof of~\ref{lem:beta-gamma-compact}: 
		Since a linear operator $T\from E\to E$ 
		is compact on $E$ if and only if its complexification $T_\bbC$ 
		is compact on $E_\bbC$, 
		see \eqref{eq:complexification-space} and \eqref{eq:complexification-ip}, 
		we obtain, for all $\gamma\in\bbR$, 
		the equivalence 
		\[ 
			A_\bbC^{-\gamma} \At_\bbC^{2\gamma} A_\bbC^{-\gamma} 
			- \id_{E_\bbC} 
			= 
			\bigl( A^{-\gamma} \At^{2\gamma} A^{-\gamma} 
			- \id_E \bigr)_\bbC 
			\in \cK(E_\bbC) 
			\quad 
			\Longleftrightarrow 
			\quad 
			A^{-\gamma} \At^{2\gamma} A^{-\gamma} 
			- \id_E 
			\in \cK(E). 
		\] 
		By assumption 
		$A^{-\beta} \At^{2\beta} A^{-\beta} 
		- \id_E \in \cK(E)$ and, thus, 
		$A_\bbC^{-\beta} \At_\bbC^{2\beta} A_\bbC^{-\beta} 
		- \id_{E_\bbC} \in \cK(E_\bbC)$
		so that by Lemma~\ref{lem:beta-gamma-complex} 
		in Appendix~\ref{appendix:interpol}, 
		$A_\bbC^{-\gamma} \At_\bbC^{2\gamma} A_\bbC^{-\gamma} 
		- \id_{E_\bbC} 
		\in \cK(E_\bbC)$ 
		holds for every $\gamma\in[0,\beta]$. 
		Then, by the above equivalence 
		$A^{-\gamma} \At^{2\gamma} A^{-\gamma} 
		- \id_E \in \cK(E)$ 
		follows first for every $\gamma\in[0,\beta]$ 
		and, subsequently, 
		since 
		$\At^\gamma A^{-\gamma}$ is an isomorphism on $E$, 
		by \citep[][Lemma~B.1]{bk-kriging} also 
		for all $\gamma\in[-\beta,0)$. 
	\end{proof} 
	
	\begin{proof}[Proof of Lemma~\ref{lem:iff-A}]
		\ref{lem:iff-A-HS} 
		Assume first that 
		for all $\eta\in\frakN_\beta$  
		there exist an orthogonal operator ${U_\eta\in\cL(E)}$ 
		and $S_\eta\in\cL_2(E)$ such that 
		$A^{\eta-1} \At A^{-\eta} = U_\eta (\id_E+ S_\eta)$  
		and $\id_E + S_\eta$ is invertible on~$E$. 
		We show via induction that, 
		for all $n\in\frakN_\beta\cap\bbN$, 
		there are   
		$\widetilde{U}_n\in\cL(E)$ orthogonal 
		and $\widetilde{S}_n\in\cL_2(E)$ self-adjoint, 
		such that 
		$\id_E + \widetilde{S}_n$ is invertible on $E$ and 
		$\At^n A^{-n} = \widetilde{U}_n (\id_E + \widetilde{S}_n)$. 
		For $n=1$, by assumption 
		$\At A^{-1} = U_1 (\id_E + S_1)$,  
		where $U_1\in\cL(E)$ is orthogonal and $S_1\in\cL_2(E)$. 
		In order to obtain a self-adjoint 
		operator $\widetilde{S}_1\in\cL_2(E)$, 
		we perform a polar decomposition: 
		\[
		\id_E + S_1 
		= 
		\widehat{U}_1 \, | \id_E + S_1 | 
		= 
		\widehat{U}_1 \sqrt{ (\id_E + S_1^* ) (\id_E + S_1 ) }
		=
		\widehat{U}_1 \sqrt{ \id_E + \bar{S}_1  }. 
		\]
		Here, 
		$\bar{S}_1 
		:= 
		\dual{S}_1 + S_1 + \dual{S}_1 S_1 
		\in \cL_2(E)$ 
		and 
		$\widehat{U}_1 
		:= 
		(\id_E + \dual{S}_1 )^{-1} \sqrt{ (\id_E + \dual{S}_1 ) (\id_E + S_1 ) }$ 
		is orthogonal on~$E$, 
		since $\id_E + S_1$ is invertible on $E$. 
		The operator   
		$\widetilde{S}_1 := \sqrt{ \id_E + \bar{S}_1 } - \id_E$ 
		is self-adjoint  
		and ${\widetilde{S}_1\in\cL_2(E)}$ 
		by Lemma~\ref{lem:I+K} 
		in Appendix~\ref{appendix:feldman-hajek}. 
		We conclude that 
		$\At A^{-1} = \widetilde{U}_1 (\id_E + \widetilde{S}_1)$, 
		where the operator $\widetilde{U}_1 = U_1 \widehat{U}_1$ is orthogonal on~$E$,  
		$\widetilde{S}_1\in\cL_2(E)$ is self-adjoint, and  
		$\id_E + \widetilde{S}_1 = \dual{\widehat{U}}_1 ( \id_E + S_1)$ 
		is invertible on~$E$. 
		For the induction step $n-1\to n$, 
		we let $n\in\{2,\ldots,\lfloor\beta\rfloor\}$ and find 
		by the induction hypothesis and by the assumption 
		on $A^{n-1}\At A^{-n}$ that  
		\begin{equation}\label{eq:proof:lem:iff-A-0-a}
			\begin{split} 
				\At^n A^{-n} 
				&= 
				\At^{n-1} A^{-(n-1)} A^{n-1} \At A^{-n}
				= 
				\widetilde{U}_{n-1} 
				(\id_E + \widetilde{S}_{n-1}) 
				U_n (\id_E + S_n) 
				\\
				&= 
				\widetilde{U}_{n-1} 
				U_n 
				(\id_E + S'_n), 
				\quad 
				\text{where}
				\quad 
				S'_n
				:= 
				\dual{U}_n \widetilde{S}_{n-1} U_n  
				+ S_n 
				+ \dual{U}_n \widetilde{S}_{n-1} U_n S_n. 
			\end{split} 
		\end{equation} 
		Here, 
		$S'_n\in\cL_2(E)$ 
		and $\id_{E} + S_n'$ 
		is invertible on $E$, 
		so that 
		(similarly as for $n=1$)
		a polar decomposition 
		applied to $\id_E + S'_n$ yields the existence of 
		an orthogonal operator 
		$\widehat{U}_n$ on $E$
		and a self-adjoint  
		operator 
		$\widetilde{S}_n\in\cL_2(E)$ such that 
		$\id_E + S_n' = \widehat{U}_n (\id_E+ \widetilde{S}_n)$ 
		and 
		\begin{equation}\label{eq:proof:lem:iff-A-0-b}
			\id_E+ \widetilde{S}_n 
			= 
			\dual{ \widehat{U} }_n  
			(\id_E + S_n')
			= 
			\dual{ \widehat{U} }_n   
			\dual{ U }_n  
			(\id_E + \widetilde{S}_{n-1}) 
			U_n (\id_E + S_n)  
		\end{equation} 
		is invertible on $E$, since 
		$\id_E + \widetilde{S}_{n-1}$ 
		and $\id_E + S_n$ are. 
		For all $\gamma\in\frakN_\beta\cap\bbN$, 
		we thus have the representation 
		$A^{-\gamma} \At^\gamma 
		= 
		\bigl( \At^\gamma A^{-\gamma} \bigr)^*
		= 
		\dual{ \widetilde{U} }_\gamma 
		\bigl( \id_E + \widetilde{U}_\gamma \widetilde{S}_\gamma \dual{ \widetilde{U} }_\gamma \bigr)$, 
		and we conclude with  
		Lemma~\ref{lem:HS-compact-property}.I 
		in Appendix~\ref{appendix:feldman-hajek} 
		that 
		\begin{equation}\label{eq:proof:lem:iff-A-0} 
			\At^\gamma A^{-\gamma} 
			\text{ is an isomorphism on $E$}, 
			\qquad  
			A^{-\gamma} \At^{2\gamma} A^{-\gamma} - \id_E \in \cL_2(E). 
		\end{equation} 
		Subsequently, \eqref{eq:proof:lem:iff-A-0} 
		follows for all $\gamma\in[-\lfloor\beta\rfloor,\lfloor\beta\rfloor]$ 
		from Lemma~\ref{lem:beta-gamma}\ref{lem:beta-gamma-HS}.  
		Finally, in the case that $\beta\notin\bbN$,  
		we may again apply Lemma~\ref{lem:HS-compact-property}.I, 
		this time 
		for $\gamma := \beta-1 \in (0,\lfloor\beta\rfloor)$. 
		This yields the representation 
		$A^{-(\beta-1)}\At^{\beta-1} = 
		\bar{U}_{\beta-1} \bigl( \id_E + \bar{S}_{\beta-1} \bigr)$, 
		where $\bar{U}_{\beta-1}$ is orthogonal on $E$, 
		$\bar{S}_{\beta-1}\in\cL_2(E)$, and the operator 
		$\id_E + \bar{S}_{\beta-1}$ is invertible on~$E$. 
		By assumption we also have that 
		$A^{\beta-1}\At A^{-\beta} = U_\beta(\id_E+S_\beta)$. 
		Therefore, we obtain  
		\eqref{eq:proof:lem:iff-A-0} 
		for all $\gamma\in[-\beta,\beta]$ 
		by similar steps as above,  
		see \eqref{eq:proof:lem:iff-A-0-a} 
		and \eqref{eq:proof:lem:iff-A-0-b}. 
		
		Conversely, if $A^{-\gamma} \At^{2\gamma} A^{-\gamma} - \id_E \in \cL_2(E)$  
		and $\At^\gamma A^{-\gamma}$ is an isomorphism on~$E$ 
		for all ${\gamma\in[-\beta,\beta]}$, 
		then by Lemma~\ref{lem:HS-compact-property}.I, 
		for every $\gamma\in[-\beta,\beta]$, 
		there exist an orthogonal operator $\bar{U}_\gamma$ 
		on $E$ 
		as well as a self-adjoint operator 
		$\bar{S}_\gamma\in\cL_2(E)$ 
		such that 
		$A^{-\gamma} \At^\gamma 
		= 
		\bar{U}_\gamma (\id_E + \bar{S}_\gamma)$ 
		and $\id_E + \bar{S}_\gamma$ is 
		boundedly invertible on $E$. 
		Therefore, for every 
		$\eta\in\frakN_\beta$, 
		we obtain  
		\begin{align*} 
			A^{\eta-1} \At A^{-\eta} 
			&=
			\big( A^{\eta-1} \At^{-(\eta-1)} \bigr) \bigl( \At^\eta A^{-\eta} \bigr) 
			= 
			\bar{U}_{-(\eta-1)} \bigl( \id_E + \bar{S}_{-(\eta-1)} \bigr)  
			\bigl( \id_E + \bar{S}_\eta \bigr) \bar{U}_\eta^*  
			\\
			&= 
			\bar{U}_{-(\eta-1)}  
			\bigl( 
			\id_E + \bar{S}_{-(\eta-1)} + \bar{S}_\eta 
			+ \bar{S}_{-(\eta-1)} \bar{S}_\eta
			\bigr) 
			\bar{U}_\eta^*  
			= 
			U_\eta (\id_E + S_\eta). 
		\end{align*} 
		Here, $U_\eta := \bar{U}_{-(\eta-1)} \bar{U}_\eta^*$ 
		is orthogonal on $E$ and,  
		since $\bar{S}_{-(\eta-1)}, \bar{S}_\eta \in \cL_2(E)$, 
		\[
		S_\eta 
		:= 
		\bar{U}_\eta \bar{S}_{-(\eta-1)} \bar{U}_\eta^*  
		+ \bar{U}_\eta \bar{S}_\eta \bar{U}_\eta^*  
		+ \bar{U}_\eta \bar{S}_{-(\eta-1)} \bar{S}_\eta \bar{U}_\eta^*  
		\in\cL_2(E) 
		\]
		inherits the Hilbert--Schmidt property. 
		Finally, 
		$\id_E + S_\eta 
		= 
		\dual{U}_\eta \bigl( A^{\eta-1} \At^{-(\eta-1)} \bigr) 
		\bigl( \At^\eta A^{-\eta} \bigr)$
		is invertible on $E$ 
		by the isomorphism property 
		of $\At^{\gamma} A^{-\gamma}$ 
		which is assumed for all $\gamma\in[-\beta,\beta]$. 
		
		Assertion~\ref{lem:iff-A-compact}  
		can be proven along the same lines, 
		using Lemmas~\ref{lem:beta-gamma}\ref{lem:beta-gamma-compact}
		and~\ref{lem:HS-compact-property}.II. 
		
		\ref{lem:iff-A-iso} 
		Assume first that  
		$B := \At - A 
		\in    \cL\bigl( \hdot{2\eta}{A};   \hdot{2(\eta-1)}{A}   \bigr) 
		\cap \cL\bigl( \hdot{2\eta}{\At}; \hdot{2(\eta-1)}{\At} \bigr)$   
		for $\eta\in\{1,\beta\}$. 
		In the case that $\beta>1$,  
		by complex interpolation we then obtain that, 
		for all $\theta\in[0,1]$,  
		\[
			B_\bbC 
			\in    \cL\bigl( \bigl[ \hdot{2}{A_\bbC}, \hdot{2\beta}{A_\bbC}\bigr]_\theta;  
			\bigl[ \hdot{0}{A_\bbC}, \hdot{2(\beta-1)}{A_\bbC}\bigr]_\theta  \bigr) 
			\cap  \cL\bigl( \bigl[ \hdot{2}{\At_\bbC}, \hdot{2\beta}{\At_\bbC}\bigr]_\theta;  
			\bigl[ \hdot{0}{\At_\bbC}, \hdot{2(\beta-1)}{\At_\bbC}\bigr]_\theta  \bigr) 
		\] 
		for the complexification 
		$B_\bbC = (\At - A)_\bbC = \At_\bbC - A_\bbC$ of $B$. 
		The 
		choice $\theta := \tfrac{\eta-1}{\beta-1}$ 
		combined with Lemma~\ref{lem:hdot-interpol} 
		in Appendix~\ref{appendix:interpol}
		yields 
		$B_\bbC 
		\in    \cL\bigl( \hdot{2\eta}{A_\bbC};   \hdot{2(\eta-1)}{A_\bbC}   \bigr) 
		\cap \cL\bigl( \hdot{2\eta}{\At_\bbC}; \hdot{2(\eta-1)}{\At_\bbC} \bigr)$ 
		for all $\eta\in[1,\beta]$, and  
		\begin{equation}\label{eq:proof:lem:iff-A-1}
			\forall \eta \in [1,\beta]: 
			\quad 
			A^{\eta-1} B A^{-\eta} \in \cL(E), 
			\quad\;\; 
			\At^{\eta-1} B \At^{-\eta} \in \cL(E). 
		\end{equation} 
		
		For $\gamma=0$, 
		$\At^0 A^{-0} = \id_E$ clearly is 
		an isomorphism on $E$. 
		We next show via induction 
		that, for all $n\in\frakN_\beta\cap\bbN=\{1,\ldots,\lfloor\beta\rfloor \}$, 
		the operator $\At^n A^{-n}$  
		is bounded on~$E$. 
		For $n=1$, we obtain boundedness of 
		$\At A^{-1} 
		= 
		(A + B) A^{-1} 
		= 
		\id_E + B A^{-1}$
		on $E$ from \eqref{eq:proof:lem:iff-A-1}. 
		
		For the 
		induction step $n - 1 \to n$, 
		let $n\in\{2,\ldots,\lfloor\beta\rfloor\}$. 
		Then, 
		\begin{equation}\label{eq:proof:lem:iff-A-2}
			\At^n A^{-n} 
			= 
			\At^{n-1}( A + B ) A^{-n} 
			= 
			\At^{n-1} A^{-(n-1)}  
			+ 
			\bigl( \At^{n-1} A^{-(n-1)} \bigr) A^{n-1} B A^{-n} .
		\end{equation} 
		By the induction hypothesis 
		and by \eqref{eq:proof:lem:iff-A-1}, 
		respectively, 
		$\At^{n-1} A^{-(n-1)}$ 
		and 
		$A^{n-1} B A^{-n}$ are bounded on~$E$ 
		and, thus, \eqref{eq:proof:lem:iff-A-2} shows that 
		$\At^n A^{-n}\in\cL(E)$ holds 
		for all $n\in\{0,\ldots,\lfloor\beta\rfloor\}$, 
		which is equivalent to  
		$\bigl( \hdot{2n}{A}, \norm{\,\cdot\,}{2n,A} \bigr) 
		\hookrightarrow 
		\bigl( \hdot{2n}{\At}, \norm{\,\cdot\,}{2n,\At} \bigr)$.  
		Again by using the identification 
		of $\hdot{2\gamma}{A_\bbC}\!, \hdot{2\gamma}{\At_\bbC}$ 
		with the corresponding complex interpolation spaces, 
		see Lemma~\ref{lem:hdot-interpol} in Appendix~\ref{appendix:interpol}, 
		we obtain 
		boundedness of the operator 
		$\At^\gamma A^{-\gamma}$ on $E$ 
		for all $\gamma\in[0,\lfloor\beta\rfloor]$. 
		In the case that $\beta\notin\bbN$, 
		we have $\beta-1\in(0,\lfloor\beta\rfloor)$, and  
		the identity in \eqref{eq:proof:lem:iff-A-2} 
		with $n$ replaced by $\beta$ 
		combined with boundedness of the operators 
		$\At^{\beta-1}A^{-(\beta-1)}$ 
		and $A^{\beta-1} B A^{-\beta}$\!, 
		see \eqref{eq:proof:lem:iff-A-1}, 
		shows that $\At^\beta A^{-\beta}\in\cL(E)$. 
		Then, 
		again by complexification and interpolation 
		$\At^\gamma A^{-\gamma}\in\cL(E)$ follows 
		for all $\gamma\in[0,\beta]$. 
		
		Since also $\At^{\eta-1} B \At^{-\eta} \in \cL(E)$ 
		holds for all $\eta\in[1,\beta]$, see \eqref{eq:proof:lem:iff-A-1}, 
		we may change the roles of $A$ and $\At$, 
		showing that 
		$A^\gamma \At^{-\gamma} \in \cL(E)$ for all $\gamma\in[0,\beta]$. 
		Thus, for every $\gamma\in[0,\beta]$, 
		$\At^\gamma A^{-\gamma}$ is bounded on~$E$ 
		and has a bounded inverse 
		$( \At^\gamma A^{-\gamma} )^{-1} 
		= A^\gamma \At^{-\gamma}$\!. 
		Due to the identity  
		$\At^{-\gamma} A^\gamma 
		= ( A^\gamma \At^{-\gamma} )^*$\!, 
		the same statement is true 
		for all $\gamma\in[-\beta,0)$. 
		
		Assume now that $\At^\gamma A^{-\gamma}\in\cL(E)$ 
		for all $\gamma\in[-\beta,\beta]$. 
		Then, for every $\eta\in\{1,\beta\}$, 
		\begin{align*} 
			A^{\eta-1} B A^{-\eta} 
			&= 
			A^{\eta-1} \At A^{-\eta} - \id_E 
			= 
			\bigl( \At^{-(\eta-1)} A^{\eta-1} \bigr)^* \bigl( \At^\eta A^{-\eta} \bigr) - \id_E 
			\in 
			\cL(E),  
			\\
			\At^{\eta-1} B \At^{-\eta} 
			&= 
			\id_E 
			\,- 
			\At^{\eta-1} A \At^{-\eta} 
			\,
			= 
			\id_E 
			- \,
			\bigl( \At^{\eta-1} A^{-(\eta-1)} \bigr) \bigl( \At^{-\eta} A^\eta \bigr)^*
			\in \cL(E), 
		\end{align*}  
		i.e., 
		$B\in\cL\bigl( \hdot{2\eta}{A};    \hdot{2(\eta-1)}{A}   \bigr)
		\cap \cL\bigl( \hdot{2\eta}{\At}; \hdot{2(\eta-1)}{\At} \bigr)$. 
	\end{proof} 
	
	\section{Proof of \texorpdfstring{Lemma~\ref{lem:a-notcompact}}{Lemma~4.9}}
	\label{appendix:proof-a-notcompact}
	
	For the proof of Lemma~\ref{lem:a-notcompact}, 
	the following auxiliary result, taken from 
	\cite[Theorem~2]{Kato1952}, will be crucial.
	
	\begin{lemma}\label{lem:kato} 
		Let $A\from\scrD(A)\subseteq E \to E$ 
		and $\At\from\scrD(\At)\subseteq E \to E$ 
		be self-adjoint, 
		nonnegative definite operators 
		on a separable Hilbert space 
		$(E, \scalar{\,\cdot\,, \,\cdot\,}{E})$ 
		over $\bbR$. 
		Whenever $\| \At \psi\|_E \leq \| A \psi\|_E$ 
		holds for all $\psi\in\scrD(A)$, 
		then, for every $\gamma\in[0,1]$, 
		we have that 
		$\| \At^\gamma \psi\|_E \leq \| A^\gamma \psi\|_E$ 
		holds for all $\psi\in\scrD(A^\gamma)$. 
	\end{lemma}
	
	\begin{proof}[Proof of Lemma~\ref{lem:a-notcompact}]	
		We prove the claim by contradiction. 
		To this end, 
		assume that the operator 
		${T_c = L^{-\nicefrac{1}{4}} \Lt^{\nicefrac{1}{2}} L^{-\nicefrac{1}{4}} 
		- c^{\nicefrac{1}{2}} \id_{L_2(\cD)}}$ 
		is compact on $L_2(\cD)$ 
		and let 
		$\widehat{L}\from \scrD(\widehat{L}) \subseteq H^1_0(\cD) \to L_2(\cD)$  
		be defined 
		as in \eqref{eq:L-div} 
		with coefficients 
		$\widehat{\ac}=\id_{\bbR^d}$ and $\widehat{\kappa}=0$, i.e., 
		$\widehat{L}=-\Delta$ is the 
		negative Dirichlet Laplacian.  
		By Theorem~\ref{thm:CM} 
		$S_{\gamma} := L^{\gamma} \widehat{L}^{-\gamma}$
		is an isomorphism on $L_2(\cD)$ 
		for all $\gamma\in[0,\nicefrac{5}{4})$
		and, therefore, also the operator 
		$\widehat{T}_c 
		:= 
		\widehat{L}^{-\nicefrac{1}{4}} \bigl(\Lt^{\nicefrac{1}{2}} 
		- c^{\nicefrac{1}{2}} L^{\nicefrac{1}{2}} \bigr) \widehat{L}^{-\nicefrac{1}{4}} 
		= 
		\dual{S}_{\nicefrac{1}{4}} T_c S_{\nicefrac{1}{4}}$ 
		is compact on $L_2(\cD)$. 

		Define 
		$\Theta := \act - c\ac \in C^\infty(\clos{\cD})^{d\times d}$\!. 
		As $c\ac \neq \act$ is assumed, 
		there exists $\s_0\in\cD$ 
		such that the symmetric $d\times d$ 
		matrix 
		$\Theta_0:=\Theta(\s_0)$  
		has a non-vanishing eigenvalue 
		and, in particular, 
		$\theta_1 \neq 0$ if  
		$\theta_1\in\bbR$ is chosen such that 
		$|\theta_1| = \max_{\lambda\in\sigma(\Theta_0)}|\lambda|$. 
		Since $\Theta$ is continuous, 
		there exists $r_0\in\bbR_+$ 
		such that $\cD_0 := B(\s_0,r_0) \Subset \cD$ and 
		$\norm{ \Theta(\s) -\Theta(\s_0) }{\bbR^{d\times d}} < \tfrac{|\theta_1|}{5}$  
		for all $\s \in \cD_0$, 
		where $B(\s_0,r_0)$ denotes 
		the open ball in $\bbR^d$\!, 
		centered at $\s_0$ 
		with radius $r_0$, and 
		$\norm{\,\cdot\,}{\bbR^{d\times d}}$ is the 	
		operator matrix norm, induced 
		by the Euclidean norm $\norm{\,\cdot\,}{\bbR^d}$. 
		We let 
		$q_1\in\bbR^d$ be an eigenvector 
		for $\Theta_0$ such that 
		$\Theta_0 q_1 = \theta_1 q_1$ 
		and $\norm{q_1}{\bbR^d}=1$. 
		In addition, we let 
		$Q_1$  
		be the $\bbR^d$-orthogonal projection onto 
		the linear space generated by $q_1$,  
		we set 
		$Q_1^\perp := \id_{\bbR^d}-Q_1$, 
		and we define $R\from\bbR^d\to\bbR^d$ 
		as the rotation 	
		that 
		maps $q_1$ to the first unit vector of the Cartesian coordinate system. 
		We then consider the domain 
		$\cD_1 := 
		\bigl\{ 
		\s\in\bbR^d  
		\, \big| \, 
		R(\s-\s_0) 
		\in\widehat{\cD}_1 
		\bigr\}$, 
		where 
		$\widehat{\cD}_1 
		:=
		\bigl( 0, \tfrac{r_0}{2d} \bigr) 
		\times 
		\bigl( 0, \tfrac{r_0}{\sqrt{d}} \bigr)^{d-1}$\!.
		Note that 
		$\cD_1\Subset\cD_0$ is 
		an affine transformation 
		of $\widehat{\cD}_1$, 
		see the illustration in Figure~\ref{fig:domains}. 
		For every $n\in\bbN$, 
		we let $C_n \in \bbR_+$ be a constant to be specified below 
		and define for all   
		${\widehat{\s}=(\widehat{\s}_1,\ldots,\widehat{\s}_d)^\top\!\in\widehat{\cD}_1}$  
		and 
		$\s\in\cD$: 
		\[ 
			\widehat{v}_n (\widehat{\s})  
			:= 
			C_n \sin\biggl( \frac{2dn\pi \widehat{\s}_1}{r_0} \biggr) 
			\prod\limits_{j=2}^d 
			\sin\biggl( \frac{\sqrt{d} n\pi \widehat{\s}_j}{r_0} \biggr) , 
			\quad\;\;  
			v_n(\s) 
			:= 
			\begin{cases} 
				\widehat{v}_n( R(\s-\s_0) ) 
				&
				\text{if } \s\in\cD_1 , 
				\\
				0 
				&
				\text{in } \cD\setminus\cD_1. 
			\end{cases} 
		\] 
		The so-defined sequence 
		$\{v_n \}_{n\in\bbN}$ 
		is orthogonal in $L_2(\cD)$ and 
		in $\hdot{1}{\widehat{L}} \cong H^1_0(\cD)$: 
		\begin{align} 
			\scalar{v_m, v_n}{L_2(\cD)} 
			&= 
			\scalar{ \widehat{v}_m, \widehat{v}_n }{L_2(\widehat{\cD}_1)}
			= 
			\delta_{mn} 
			C_n^2 \, 
			\tfrac{r_0}{4d} 
			\bigl( \tfrac{r_0}{2\sqrt{d}} \bigr)^{d-1} 
			= 
			\delta_{mn} C_n^2 \,  
			\tfrac{  r_0^d }{ 2^{d+1} d^{ \nicefrac{(d+1)}{2} } }, 
			\notag 
			\\
			\scalar{v_m, v_n}{1,\widehat{L}} 
			&= 
			\scalar{\nabla v_m, \nabla v_n}{L_2(\cD)} 
			= 
			\scalar{ R^* \nabla \widehat{v}_m, 
				R^* \nabla \widehat{v}_n }{L_2(\widehat{\cD}_1)}
			= 
			\scalar{ \widehat{v}_m, 
				- \Delta \widehat{v}_n }{L_2(\widehat{\cD}_1)} 
			\label{eq:proof:lem:a-notcompact-0}
			\\
			&= 
			\bigl( 4d^2 + (d-1) d \bigr) 
			\tfrac{ n^2 \pi^2 }{ r_0^2 } 	
			\scalar{ \widehat{v}_m, \widehat{v}_n }{L_2(\widehat{\cD}_1)} 
			= 
			\delta_{mn} 
			C_n^2 \, 
			\tfrac{ (5d-1) n^2 \pi^2 r_0^{d-2} }{ 2^{d+1} d^{ \nicefrac{(d-1)}{2}} }, 
			\notag 
		\end{align}  
		where $\delta_{mn}$ is the Kronecker delta. 
		We set  
		$C_n^2 := \frac{ 2^{d+1} d^{(d-1)/2} }{ (5d-1) n^2 \pi^2 r_0^{d-2} }$ 
		so that 
		$v_n$ is normalized in $\hdot{1}{\widehat{L}}$ 
		for all $n\in\bbN$. 
		We let $\cV \subset \hdot{1}{\widehat{L}}$ be the subspace generated by 
		this orthonormal system, i.e., 
		$\cV := \operatorname{span}\{v_n\}_{n\in\bbN}$. 
		Since $\{v_n\}_{n\in\bbN}$ 
		is an orthonormal system in $\hdot{1}{\widehat{L}}$, 
		the subspace $\cV$ is closed in $\hdot{1}{\widehat{L}}$. 
		
		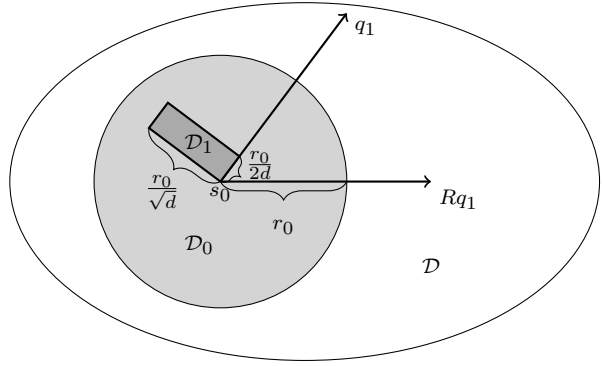
\begin{figure}[t] 
			\begin{center}
				\begin{tikzpicture}[scale=0.56] 
				\draw[fill={rgb:black,1;white,5}] (0,0) circle (3cm);
				\draw[thick,->] (0,0) -- (5,0) node[anchor=north west] {$R q_1$};
				\draw[thick,->] (0,0) -- (3,4) node[anchor=north west] {$q_1$};
				\draw[fill={rgb:black,1;white,2},thick,-] (0,0) -- (0.45,0.6) -- (-1.2471,1.8728) -- (-1.6971,1.2728) -- (0,0);
				\draw[decorate,decoration={brace,amplitude=7pt,mirror},xshift=0pt,yshift=0pt] (0,0) -- (3,0) node[midway,yshift=-0.6cm] {$r_0$};
				\draw[decorate,decoration={brace,amplitude=7pt},xshift=0pt,yshift=0pt] (0,0) -- (-1.6971,1.2728) node[midway,yshift=-0.5cm,xshift=-0.3cm] {$\frac{r_0}{\sqrt{d}}$};
				\draw[decorate,decoration={brace,amplitude=5pt,mirror},xshift=0pt,yshift=0pt] (0,0) -- (0.45,0.6) node[midway,yshift=0.05cm,xshift=0.4cm] {$\frac{r_0}{2d}$};
				\node (a) at (-0.5, -1.5) {$\cD_0$};
				\node (b) at (-0.5, 0.9) {$\cD_1$};
				\node (c) at (0, -0.3) {$s_0$};
				\node (c) at (5, -2) {$\cD$};
				\draw(2,0) ellipse (7cm and 4.25cm);
				\end{tikzpicture}
			\end{center}
			\vspace{-0.5cm} 
			\caption{\label{fig:domains}Illustration of the subdomains 
				$\cD_0 = B(s_0,r_0)\Subset\cD$ and $\cD_1\Subset\cD_0$ used in the proof 
				of Lemma~\ref{lem:a-notcompact}}
		\end{figure}
		
		In what follows, we distinguish the following two cases: 
		\textbf{Case I:} 
		$\theta_1 > 0$ and  
		\textbf{Case II:}  
		$\theta_1 < 0$. 
		
		\textbf{Case I:} 
		For any $v\in\cV$, we have 
		$\int_{\cD\setminus\cD_0} |v(\s)|^2 + \|\nabla v(\s)\|_{\bbR^d}^2 \, \rd \s = 0$  
		and find that 
		\begin{align} 
			\bigl\| \Lt^{\nicefrac{1}{2}} v &\bigr\|_{L_2(\cD)}^2   
			- 
			c 
			\, \bigl\| L^{\nicefrac{1}{2}} v \bigr\|_{L_2(\cD)}^2 
			=    
			\bigl\langle \bigl( \Lt - cL \bigr) v, v \bigr\rangle  
			= 
			\scalar{\Theta \nabla v, \nabla v}{L_2(\cD_0)} 
			+  
			\bigl( \bigl(\kappat^2 - c\kappa^2 \bigr) v, v \bigr)_{L_2(\cD_0)} 
			\notag 
			\\ 
			&\qquad\geq 
			\scalar{\Theta_0 Q_1 \nabla v, Q_1 \nabla v}{L_2(\cD_0)}  
			- 
			\bigl| \bigl( \Theta_0 Q_1^\perp \nabla v, Q_1^\perp \nabla v \bigr)_{L_2(\cD_0)} \bigr| 
			\notag 
			\\
			&\hspace*{4.8cm}  
			- 
			\bigl| \scalar{ (\Theta - \Theta_0) \nabla v, \nabla v}{L_2(\cD_0)} \bigr| 
			- 
			C_{\kappa^2} \norm{v}{L_2(\cD)}^2 
			\notag 
			\\
			&\qquad\geq 
			\theta_1 
			\left( \norm{Q_1 \nabla v}{L_2(\cD_0)}^2 
			- 
			\bigl\| Q_1^\perp \nabla v \bigr\|_{L_2(\cD_0)}^2 \right) 
			- 
			\tfrac{ \theta_1 }{5}  \norm{\nabla v}{L_2(\cD_0)}^2 
			- 
			C_{\kappa^2} \norm{v}{L_2(\cD)}^2, 
			\label{eq:proof:lem:a-notcompact-1} 
		\end{align} 
		where 
		$C_{\kappa^2} 
		:= 
		\bigl\| \kappat^2 - c\kappa^2 \bigr\|_{L_\infty(\cD_0)}\in\bbR_+$ 
		and the last step follows from the choice of $\cD_0\Subset\cD$. 
		Moreover, by the definition of  $\{v_n\}_{n\in\bbN}$ 
		(noting that 
		$Q_1^\perp \nabla v_m \perp Q_1^\perp \nabla v_n$ and 
		$Q_1 \nabla v_m \perp Q_1 \nabla v_n$ in $L_2(\cD)$ 
		whenever $m\neq n$ 
		since $Q_1 R^* \nabla = \tfrac{\partial}{\partial\widehat{\s}_1}$,  
		cf.~\eqref{eq:proof:lem:a-notcompact-0}) 
		we find that  
		\[
			\bigl\| Q_1^\perp \nabla v \bigr\|_{L_2(\cD_0)}^2  
			= 
			\bigl\| Q_1^\perp \nabla v \bigr\|_{L_2(\cD_1)}^2  
			= 
			\tfrac{d-1}{4d} 
			\norm{ Q_1 \nabla v }{L_2(\cD_1)}^2
			< 
			\tfrac{1}{4} \norm{ Q_1 \nabla v }{L_2(\cD_0)}^2  
			\quad 
			\forall v\in\cV. 
		\]
		This readily implies that 
		$\norm{\nabla v}{L_2(\cD)}^2 
		= 
		\norm{\nabla v}{L_2(\cD_0)}^2 
		> 
		5 \, \bigl\| Q_1^\perp \nabla v \bigr\|_{L_2(\cD_0)}^2$ 
		and, thus, 
		by \eqref{eq:proof:lem:a-notcompact-1} 
		\begin{equation}\label{eq:proof:lem:a-notcompact-2} 
		\begin{split}  
			\bigl\| \Lt^{\nicefrac{1}{2}} v \bigr\|_{L_2(\cD)}^2   
			- 
			c \, \bigl\| L^{\nicefrac{1}{2}} v \bigr\|_{L_2(\cD)}^2 
			&\geq 
			\tfrac{4}{5} \theta_1 
			\norm{\nabla v}{L_2(\cD_0)}^2 
			- 
			2 \theta_1 \bigl\| Q_1^\perp \nabla v \bigr\|_{L_2(\cD_0)}^2 
			- 
			C_{\kappa^2} \norm{v}{L_2(\cD)}^2 
			\\
			&> 
			\left( 
			\tfrac{2}{5} \theta_1  
			- 
			C_{\kappa^2} 
			\tfrac{ \norm{v}{L_2(\cD)}^2}{\norm{ \nabla v}{L_2(\cD)}^2 } 
			\right) 
			\norm{\nabla v}{L_2(\cD)}^2 
			\quad 
			\forall 
			v\in\cV. 
		\end{split}
		\end{equation} 
		
		\textbf{Case II:} For $\theta_1<0$, we similarly obtain that 
		\begin{equation}\label{eq:proof:lem:a-notcompact-3} 
			- \bigl\| \Lt^{\nicefrac{1}{2}} v \bigr\|_{L_2(\cD)}^2   
			+ 
			c \, \bigl\| L^{\nicefrac{1}{2}} v \bigr\|_{L_2(\cD)}^2 
			> 
			\left( 
			\tfrac{2}{5} |\theta_1|  
			- 
			C_{\kappa^2} 
			\tfrac{ \norm{v}{L_2(\cD)}^2}{\norm{ \nabla v}{L_2(\cD)}^2 } 
			\right) 
			\norm{\nabla v}{L_2(\cD)}^2 
			\quad 
			\forall 
			v\in\cV. 
		\end{equation}
		
		Next, define for $N\in\bbN$  
		the subspace $\cV_N^\perp \subset \cV$ by
		$\cV_N^\perp := \operatorname{span}\{v_n\}_{n=N+1}^\infty$ 
		and note that 
		\[ 
			\frac{\norm{v}{L_2(\cD)}^2}{\norm{\nabla v}{L_2(\cD)}^2} 
			= 
			\frac{\sum\nolimits_{n>N} \alpha_n^2 \norm{v_n}{L_2(\cD)}^2}{ 
				\sum\nolimits_{n>N} \alpha_n^2 } 
			\leq 
			\norm{ v_{N+1} }{L_2(\cD)}^2 
			= 
			\frac{ r_0^2 } {(5d-1)d (N+1)^2 \pi^2 }, 
		\] 
		for any $v = \sum_{n>N} \alpha_n v_n\in\cV_N^\perp$. 
		For this reason, there exists $N_0\in\bbN$ such that 
		$C_{\kappa^2} \tfrac{ \norm{v}{L_2(\cD)}^2}{\norm{\nabla v}{L_2(\cD)}^2 } 
		\leq \tfrac{|\theta_1|}{5}$ 
		holds for all $v \in \cV_{N_0}^{\perp}$.
		For this choice 
		\eqref{eq:proof:lem:a-notcompact-2} 
		and \eqref{eq:proof:lem:a-notcompact-3} 
		give,  
		for all $v\in\cV_{N_0}^\perp$, 
		the estimates 
		\begin{equation}\label{eq:proof:lem:a-notcompact-4} 
		\begin{split} 
			\text{\textbf{Case I:}}
			\quad 
			\bigl\| \Lt^{\nicefrac{1}{2}} v \bigr\|_{L_2(\cD)}^2   
			&> 
			c\, \bigl\| L^{\nicefrac{1}{2}} v \bigr\|_{L_2(\cD)}^2   
			+ 
			\tfrac{\theta_1}{5} 
			\bigl\| \widehat{L}^{\nicefrac{1}{2}} v \bigr\|_{L_2(\cD)}^2    
			\geq 
			(c + \theta_1') 
			\bigl\| L^{\nicefrac{1}{2}} v \bigr\|_{L_2(\cD)}^2 , 
			\hspace*{-0.2cm} 
			\\
			\text{\textbf{Case II:}}
			\;\;\,
			c\, \bigl\| L^{\nicefrac{1}{2}} v \bigr\|_{L_2(\cD)}^2  
			&>  
			\bigl\| \Lt^{\nicefrac{1}{2}} v \bigr\|_{L_2(\cD)}^2  
			+ 
			\tfrac{|\theta_1|}{5} 
			\bigl\| \widehat{L}^{\nicefrac{1}{2}} v \bigr\|_{L_2(\cD)}^2 
			\, \geq 
			(1+\widetilde{\theta}_1') 
			\bigl\| \Lt^{\nicefrac{1}{2}} v \bigr\|_{L_2(\cD)}^2 . 
			\hspace*{-0.2cm} 
		\end{split} 
		\end{equation} 
		Here, we used 
		that, for all $v\in\hdot{1}{\widehat{L}} \setminus\{0\}$, 
		we have 
		$\bigl\| \widehat{L}^{\nicefrac{1}{2}} v \bigr\|_{L_2(\cD)}^2 
		= 
		\langle \widehat{L} v, v \rangle 
		= \| \nabla v \|_{L_2(\cD)}^2$ and 
		\begin{equation}\label{eq:proof:lem:a-notcompact-5} 
			\frac{ \| \widehat{L}^{\nicefrac{1}{2}} v \|_{L_2(\cD)}^2 }{ 
				\| L^{\nicefrac{1}{2}} v \|_{L_2(\cD)}^2  }
			\geq 
			\norm{ L^{\nicefrac{1}{2}} \widehat{L}^{-\nicefrac{1}{2}} }{\cL(L_2(\cD))}^{-2} 
			\quad\;  
			\text{and} 
			\quad\;\; 
			\frac{ \| \widehat{L}^{\nicefrac{1}{2}} v \|_{L_2(\cD)}^2 }{ 
				\| \Lt^{\nicefrac{1}{2}} v \|_{L_2(\cD)}^2  }
			\geq 
			\norm{ \Lt^{\nicefrac{1}{2}} \widehat{L}^{-\nicefrac{1}{2}} }{\cL(L_2(\cD))}^{-2},  
		\end{equation} 
		since 
		$L^{\nicefrac{1}{2}} \widehat{L}^{-\nicefrac{1}{2}} $ 
		and 
		$\Lt^{\nicefrac{1}{2}} \widehat{L}^{-\nicefrac{1}{2}}$ 
		are isomorphisms on $L_2(\cD)$ by 
		Theorem~\ref{thm:CM}. 
		Thus, 
		${\theta_1', \widetilde{\theta}_1' \in \bbR_+}$ 
		in \eqref{eq:proof:lem:a-notcompact-4} 
		may be chosen as 
		$\theta_1' 
		:= 
		\tfrac{\theta_1}{5}  
		\norm{ L^{\nicefrac{1}{2}} \widehat{L}^{-\nicefrac{1}{2}} }{\cL(L_2(\cD))}^{-2}$ 
		and
		$\widetilde{\theta}_1' 
		:= 
		\tfrac{|\theta_1|}{5}  
		\norm{ \Lt^{\nicefrac{1}{2}} \widehat{L}^{-\nicefrac{1}{2}} }{\cL(L_2(\cD))}^{-2}$. 
		We next introduce the spaces  
		$\cU_{N_0}^\perp 
		:= 
		\widehat{L}^{\nicefrac{1}{2}} \bigl( \cV_{N_0}^\perp \bigr) 
		\subset L_2(\cD)$ 
		and 
		$\cW_{N_0}^\perp 
		:= 
		\widehat{L}^{\nicefrac{1}{4}} \bigl( \cV_{N_0}^\perp \bigr) 
		\subset\hdot{\nicefrac{1}{2}}{\widehat{L}}$\!.  
		By continuity of 
		$\widehat{L}^{\nicefrac{1}{2}}\from\hdot{1}{\widehat{L}}\to L_2(\cD)$ 
		and of 
		$\widehat{L}^{\nicefrac{1}{4}}\from\hdot{1}{\widehat{L}}\to \hdot{\nicefrac{1}{2}}{\widehat{L}}$
		as well as 
		closedness of $\cV_{N_0}^\perp$ in $\hdot{1}{\widehat{L}}$, 
		also $\cU_{N_0}^\perp$ and $\cW_{N_0}^\perp$ 
		are closed subspaces of $L_2(\cD)$ 
		and of $\hdot{\nicefrac{1}{2}}{\widehat{L}}$\!, respectively. 
		Thanks to self-adjointness of~$L$ and~$\Lt$ 
		we may 
		apply 
		Lemma~\ref{lem:kato} 
		for $\gamma:=\nicefrac{1}{2}$ 
		and the operators 
		$A := \Lt^{\nicefrac{1}{2}}$\!, 
		$\At := (c+\theta')^{\nicefrac{1}{2}} L^{\nicefrac{1}{2}}$\!, 
		respectively, 
		$A := c^{\nicefrac{1}{2}} L^{\nicefrac{1}{2}}$\!, 
		$\At := (1+\widetilde{\theta}_1')^{\nicefrac{1}{2}} \Lt^{\nicefrac{1}{2}}$\!, 
		to the relations in 
		\eqref{eq:proof:lem:a-notcompact-4}. 
		This in particular implies: 
		\begin{align*}
			\quad\text{\textbf{Case I:}}
			&
			& 
			\bigl\| \Lt^{\nicefrac{1}{4}} w \bigr\|_{L_2(\cD)}^2
			&\geq 
			(c+\theta_1')^{\nicefrac{1}{2}}
			\bigl\| L^{\nicefrac{1}{4}} w \bigr\|_{L_2(\cD)}^2  
			&& 
			\forall 
			w\in\cW_{N_0}^\perp,  
			\quad
			\\ 
			\quad\text{\textbf{Case II:}}
			&
			&
			c^{\nicefrac{1}{2}}  
			\bigl\| L^{\nicefrac{1}{4}} w \bigr\|_{L_2(\cD)}^2 
			&\geq 
			(1+\widetilde{\theta}_1')^{\nicefrac{1}{2}} 
			\bigl\| \Lt^{\nicefrac{1}{4}} w \bigr\|_{L_2(\cD)}^2 
			&& 
			\forall 
			w\in\cW_{N_0}^\perp. 
			\quad
		\end{align*}
		By using the lower bound   
		$\sqrt{x+y} \geq \sqrt{x} + \tfrac{y}{2\sqrt{x+y}}$ 
		for $x,y\in\bbR_+$, 
		we obtain the relations  
		\begin{align*}
			\quad\text{\textbf{Case I:}}
			&
			& 
			\bigl\| \Lt^{\nicefrac{1}{4}} w \bigr\|_{L_2(\cD)}^2 
			- c^{\nicefrac{1}{2}} \bigl\| L^{\nicefrac{1}{4}} w\bigr\|_{L_2(\cD)}^2 
			&\geq 
			c_1 \bigl\| \widehat{L}^{\nicefrac{1}{4}} w \bigr\|_{L_2(\cD)}^2 
			&& 
			\forall 
			w\in\cW_{N_0}^\perp,  
			\quad
			\\ 
			\quad\text{\textbf{Case II:}}
			&
			&
			c^{\nicefrac{1}{2}} \bigl\| L^{\nicefrac{1}{4}} w \bigr\|_{L_2(\cD)}^2 
			- \bigl\| \Lt^{\nicefrac{1}{4}} w \bigr\|_{L_2(\cD)}^2 
			&\geq 
			c_2  \bigl\| \widehat{L}^{\nicefrac{1}{4}} w \bigr\|_{L_2(\cD)}^2 
			&& 
			\forall 
			w\in\cW_{N_0}^\perp,  
			\quad
		\end{align*} 
		where we have proceeded similarly 
		as in \eqref{eq:proof:lem:a-notcompact-5} 
		using the isomorphism property 
		of $\widehat{L}^{\nicefrac{1}{4}} L^{-\nicefrac{1}{4}}$ 
		and $\widehat{L}^{\nicefrac{1}{4}} \Lt^{-\nicefrac{1}{4}}$ on $L_2(\cD)$, 
		see Theorem~\ref{thm:CM}, 
		and defined 
		the constants 
		$c_1, c_2 \in \bbR_+$ 
		by 
		\[
			c_1 := 
			\tfrac{\theta_1'}{2}
			(c+\theta_1')^{-\nicefrac{1}{2}}
			\bigl\| \widehat{L}^{\nicefrac{1}{4}} L^{-\nicefrac{1}{4}} \bigr\|_{\cL(L_2(\cD))}^{-2} 
			\quad 
			\text{and} 
			\quad 
			c_2 
			:= 
			\tfrac{\widetilde{\theta}_1'}{2} 
			(1+\widetilde{\theta}_1')^{-\nicefrac{1}{2}}
			\bigl\| \widehat{L}^{\nicefrac{1}{4}} \Lt^{-\nicefrac{1}{4}} \bigr\|_{\cL(L_2(\cD))}^{-2} . 
		\]
		Since 
		$\widehat{L}^{-\nicefrac{1}{4}} \bigl( \cU_{N_0}^\perp \bigr) 
		= \cW_{N_0}^\perp$,  
		we may reformulate these relations as   
		\begin{align*}
			\quad\text{\textbf{Case I:}}
			&
			& 
			\bigl( \widehat{L}^{-\nicefrac{1}{4}} 
				\bigl(\Lt^{\nicefrac{1}{2}} - c^{\nicefrac{1}{2}} L^{\nicefrac{1}{2}} \bigr) 
				\widehat{L}^{-\nicefrac{1}{4}} u, u \bigr)_{L_2(\cD)} 
			&\geq 
			c_1 \norm{u}{L_2(\cD)}^2 
			&& 
			\forall 
			u\in\cU_{N_0}^\perp,  
			\quad
			\\
			\quad\text{\textbf{Case II:}}
			&
			&
			\bigl( \widehat{L}^{-\nicefrac{1}{4}} 
				\bigl(c^{\nicefrac{1}{2}} L^{\nicefrac{1}{2}} - \Lt^{\nicefrac{1}{2}} \bigr)
				\widehat{L}^{-\nicefrac{1}{4}} u, u \bigr)_{L_2(\cD)} 
			&\geq 
			c_2 \norm{u}{L_2(\cD)}^2 
			&& 
			\forall 
			u\in\cU_{N_0}^\perp. 
			\quad
		\end{align*} 

		Let $\cP_{N_0}^\perp$ denote the 
		$L_2(\cD)$-orthogonal projection onto $\cU_{N_0}^\perp$. 
		By the above observation, 
		the spectrum of 
		$\cP_{N_0}^\perp \widehat{T}_c |_{\cU_{N_0}^\perp} 
		\from \cU_{N_0}^\perp \to \cU_{N_0}^\perp$ 
		is in \textbf{Case I} bounded from below by $c_1>0$  
		and in \textbf{Case~II} bounded from above by $-c_2<0$. 
		Thus, the spectrum of 
		$\widehat{T}_c 
		= \widehat{L}^{-\nicefrac{1}{4}} 
		\bigl( \Lt^{\nicefrac{1}{2}} - c^{\nicefrac{1}{2}} L^{\nicefrac{1}{2}} \bigr) 
		\widehat{L}^{-\nicefrac{1}{4}}  
		\in\cL( L_2(\cD) )$ 
		does not accumulate only at zero 
		and $\widehat{T}_c$ is not compact on $L_2(\cD)$. 
		We conclude that in the case  
		$c\ac\neq\act$, the operator 
		$T_c = L^{-\nicefrac{1}{4}}\Lt^{\nicefrac{1}{2}} L^{-\nicefrac{1}{4}} 
		- c^{\nicefrac{1}{2}}\id_{L_2(\cD)}$ 
		is not compact on $L_2(\cD)$. 
	\end{proof} 
	
	\section{Auxiliary results based on interpolation theory}\label{appendix:interpol}
	
	In this section we derive several auxiliary results 
	which are needed in Section~\ref{section:wm-D} 
	and for 
	the proofs of Lemmas~\ref{lem:beta-gamma} 
	and~\ref{lem:iff-A}  
	in Appendix~\ref{appendix:proofs-lemmas}. 
	Since we use the concept of complex interpolation 
	for this purpose, throughout this 
	section all Hilbert spaces are assumed to 
	be infinite-dimensional vector spaces over $\bbC$. 
	We start by recalling the definition of complex interpolation 
	between two separable Hilbert spaces $E_0$ and $E_1$  
	for the case that $E_1$ is continuously and densely embedded 
	in $E_0$, similarly as it can be found e.g.\ in 
	\cite[Section~1.5.1]{Yagi2010}, 
	see also 
	\cite[Definitions~2.1 and~2.3]{Lunardi2018}. 
	
	\begin{definition}\label{def:complex-interpol} 
		Let $(E_0, \scalar{\,\cdot\,, \,\cdot\,}{E_0})$,  
		$(E_1, \scalar{\,\cdot\,, \,\cdot\,}{E_1})$ 
		be two separable Hilbert spaces 
		over~$\bbC$ 
		such that the embedding 
		$(E_1, \norm{\,\cdot\,}{E_1}) \hookrightarrow (E_0,\norm{\,\cdot\,}{E_0})$ 
		is continuous and dense.  
		Then, the interpolation space 
		between $E_0$ and $E_1$ with parameter 
		$\theta \in [0,1]$ is defined by 
		\[
			[E_0, E_1]_{\theta}  
			:= 
			\{f(\theta) : f\in\cF(E_0, E_1) \}, 
			\qquad 
			\norm{v}{[E_0, E_1]_{\theta} }
			:= 
			\inf_{\substack{ f\in\cF(E_0, E_1), \\ f(\theta)=v }} 
			\| f \|_{\cF(E_0, E_1)}, 
		\]
		where $\cF(E_0, E_1)$ is the space 
		of all functions  
		\[
			f \from S \to E_0, 
			\qquad 
			S := \{a + ib : a \in [0,1], \ b \in \bbR \} \subset \bbC, 
		\]
		such that 
		\begin{enumerate*} 
			\item $f$ is holomorphic in the interior of $S$ and 
			continuous and bounded up
			to the boundary of~$S$ with values in $E_0$; 
			and 
			\item the mapping  
			$\bbR\ni t \mapsto f(1+it) \in E_1$ 
			is continuous and bounded, 
		\end{enumerate*} 
		with   
		\[
			\norm{f}{\cF(E_0, E_1)} 	
			:= 
			\max\left\{ 
			\sup\nolimits_{t\in\bbR} \norm{ f(it) }{E_0}, 
			\; 
			\sup\nolimits_{t\in\bbR} \norm{ f(1+it) }{E_1} 
			\right\}
			< \infty. 
		\]
	\end{definition}  
	
	In the next lemma the interpolation spaces 
	between two specific closed subspaces 
	$V_0\subseteq E_0$ and 
	$V_1 := E_1 \cap V_0 \subseteq E_1$ 
	are characterized. 
	
	\begin{lemma}\label{lem:interpol-subspace}
		Let $\bigl(E_0, \scalar{\,\cdot\,, \,\cdot\,}{E_0} \bigr)$ and 
		$\bigl(E_1, \scalar{\,\cdot\,, \,\cdot\,}{E_1}\bigr)$ be separable Hilbert spaces 
		over $\bbC$ such that 
		the embedding 
		${( E_1, \norm{\,\cdot\,}{E_1} ) \hookrightarrow (E_0, \norm{\,\cdot\,}{E_0} )}$ 
		is continuous and dense, 
		and such that there exists 
		an orthonormal basis $\{e_j\}_{j\in\bbN}$ for $E_0$ 
		whose basis vectors are pairwise orthogonal in $E_1$. 
		In addition,  
		let~$N\in\bbN$ and $V_0:=U_N^\perp$ 
		be the $E_0$-orthogonal complement of 
		$U_N:=\operatorname{span}\{e_j\}_{j=1}^N$, 
		and define the closed subspace 
		$V_1 := E_1 \cap V_0$ of $E_1$. Then, 
		\[
			[(V_0, \norm{\,\cdot\,}{E_0}), (V_1, \norm{\,\cdot\,}{E_1})]_\theta 
			= 
			[(E_0, \norm{\,\cdot\,}{E_0}), (E_1, \norm{\,\cdot\,}{E_1})]_\theta  \cap V_0, 
		\]
		with isometry, i.e., 
		$\norm{\,\cdot\,}{[(V_0, \norm{\,\cdot\,}{E_0}), (V_1, \norm{\,\cdot\,}{E_1})]_\theta } 
		= 
		\norm{\,\cdot\,}{[(E_0, \norm{\,\cdot\,}{E_0}), (E_1, \norm{\,\cdot\,}{E_1})]_\theta}$. 
	\end{lemma} 
	
	\begin{proof} 	
		Since 
		$(V_0, \norm{\,\cdot\,}{E_0}) \subseteq (E_0, \norm{\,\cdot\,}{E_0})$ 
		and 
		$(V_1, \norm{\,\cdot\,}{E_1}) \subseteq (E_1, \norm{\,\cdot\,}{E_1})$ 
		are closed subspaces, the inclusion 
		$\cF(V_0, V_1) \subseteq \cF(E_0, E_1)$ 
		follows 
		(see Definition~\ref{def:complex-interpol}), 
		and 
		\[
			\norm{g}{\cF(E_0, E_1)} = \norm{g}{\cF(V_0, V_1)} 
			\quad \forall g\in \cF(V_0, V_1). 
		\]
		Therefore, we find, for any $v\in V_1$, 
		\begin{align*} 
			\norm{v}{[(E_0, \norm{\,\cdot\,}{E_0}), (E_1, \norm{\,\cdot\,}{E_1})]_\theta} 
			&= 
			\inf_{\substack{ f\in\cF(E_0, E_1), \\ f(\theta)=v}} \| f \|_{\cF(E_0, E_1)} 
			\leq 
			\inf_{\substack{ g\in\cF(V_0, V_1), \\ g(\theta)=v}} \| g \|_{\cF(E_0, E_1)} 
			\\
			&= 
			\inf_{\substack{ g\in\cF(V_0, V_1), \\ g(\theta)=v}} \| g \|_{\cF(V_0, V_1)}
			= 
			\norm{v}{[(V_0, \norm{\,\cdot\,}{E_0}), (V_1, \norm{\,\cdot\,}{E_1})]_\theta}. 
		\end{align*}
		This shows that 
		$[(V_0, \norm{\,\cdot\,}{E_0}), (V_1, \norm{\,\cdot\,}{E_1})]_\theta  
		\subseteq 
		[(E_0, \norm{\,\cdot\,}{E_0}), (E_1, \norm{\,\cdot\,}{E_1})]_\theta$. 
		Furthermore, clearly 
		$[(V_0, \norm{\,\cdot\,}{E_0}), (V_1, \norm{\,\cdot\,}{E_1})]_\theta 
		\subseteq V_0$. 
		It remains to prove that 
		\begin{align*} 
			[(E_0, \norm{\,\cdot\,}{E_0}), (E_1, \norm{\,\cdot\,}{E_1})]_\theta 
			\cap V_0 
			&\subseteq 
			[(V_0, \norm{\,\cdot\,}{E_0}), (V_1, \norm{\,\cdot\,}{E_1})]_\theta , 
			\\
			\norm{v}{[(V_0, \norm{\,\cdot\,}{E_0}), (V_1, \norm{\,\cdot\,}{E_1})]_\theta} 
			&\leq 
			\norm{v}{[(E_0, \norm{\,\cdot\,}{E_0}), (E_1, \norm{\,\cdot\,}{E_1})]_\theta} 
			\qquad\quad 
			\forall v \in V_1. 
		\end{align*} 
		To this end, 
		for $f \in \cF(E_0, E_1)$, 
		define 
		\[
			f_V \from S \to V_0 , 
			\qquad 
			z \mapsto f_V(z) := P_0 f(z),  
		\] 
		where $P_0 \from E_0 \to V_0$ is the $E_0$-orthogonal 
		projection onto $V_0$. 
		Note that by the assumptions on $V_0$, 
		we obtain that $P_0 v\in E_1$ holds for all $v\in E_1$, and  
		$\norm{P_0 v}{E_1}\leq \norm{v}{E_1}$. Therefore,  
		$f_V \in \cF(V_0, V_1)$ and 
		$\norm{f_V}{\cF(V_0, V_1)} 
		= 
		\norm{f_V}{\cF(E_0, E_1)} 
		\leq \norm{f}{\cF(E_0, E_1)}$. 
		We thus find that, for all $v\in V_1$, 
		\begin{align*} 
			\norm{v}{[(E_0, \norm{\,\cdot\,}{E_0}), (E_1, \norm{\,\cdot\,}{E_1})]_\theta} 
			&= 
			\inf_{\substack{ f\in\cF(E_0, E_1), \\ f(\theta)=v}} \norm{ f }{\cF(E_0, E_1)} 
			\geq 
			\inf_{\substack{ f\in\cF(E_0, E_1), \\ f(\theta)=v}} \norm{ f_V }{\cF(E_0, E_1)} 
			\\
			&= 
			\inf_{\substack{ f\in\cF(E_0, E_1), \\ f_V(\theta)=v}} \norm{ f_V }{\cF(V_0, V_1)}
			= 
			\norm{v}{[(V_0, \norm{\,\cdot\,}{E_0}), (V_1, \norm{\,\cdot\,}{E_1})]_\theta}, 
		\end{align*} 
		where the last equality follows from the fact that for every 
		$g\in \cF(V_0, V_1)$ there exists a function 
		${f \in \cF(E_0, E_1) \supseteq \cF(V_0, V_1)}$ such that 
		$g = f_V$ ($f$ may be chosen as $g$). 
	\end{proof} 
	
	In what follows, we let 
	$A \from \scrD(A) \subset E \to E$ be a 
	densely defined, self-adjoint, positive definite linear 
	operator on a separable Hilbert space 
	$(E, \scalar{\,\cdot\,, \,\cdot\,}{E} )$ over $\bbC$ with $\dim(E)=\infty$.  
	We furthermore assume that 
	$A$ has a compact inverse $A^{-1}\in\cK(E)$ 
	so that $A$ diagonalizes with respect to an 
	orthonormal eigenbasis $\{e_j\}_{j\in\bbN}$ for $E$ 
	and corresponding eigenvalues 
	$(\lambda_j)_{j\in\bbN}\subset\bbR_+$ accumulating only at~$\infty$. 
	For $r\in[0,\infty)$ the fractional 
	power operator $A^{\nicefrac{r}{2}}$ then is defined 
	as in the real-valued case, see \eqref{eq:def:Abeta}, 
	and its domain $\hdot{r}{A} = \scrD(A^{\nicefrac{r}{2}})$,  
	as given in \eqref{eq:def:hdotA}, 
	is itself a separable Hilbert space. 
	We identify $\hdot{0}{A}=E$ with its dual 
	and obtain 
	for $0\leq r_0 \leq r_1 < \infty$ 
	the inclusions 
	\[
		\bigl( \hdot{r_1}{A} , \norm{\,\cdot\,}{r_1, A} \bigr) 
		\hookrightarrow 
		\bigl( \hdot{r_0}{A} , \norm{\,\cdot\,}{r_0, A} \bigr) 
		\hookrightarrow 
		\bigl( \hdot{-r_0}{A} , \norm{\,\cdot\,}{-r_0, A} \bigr) 
		\hookrightarrow 
		\bigl( \hdot{-r_1}{A} , \norm{\,\cdot\,}{-r_1, A} \bigr),  
	\]
	where, for $r \geq 0$, 
	$\bigl(\hdot{-r}{A}, \norm{\,\cdot\,}{-r,A} \bigr)$ is 
	the dual space of $\bigl(\hdot{r}{A}, \norm{\,\cdot\,}{r,A} \bigr)$. 

	Furthermore, for every $r\in[r_0, r_1]$, 
	$\hdot{r}{A}$ coincides 
	with the corresponding complex interpolation 
	space between $\hdot{r_0}{A}$ and $\hdot{r_1}{A}$\!. 
	This identification readily follows from, e.g., 
	\cite[][Thorem~16.1]{Yagi2010}, see also \cite[][Chapter~3]{Lunardi2018}, 
	and we summarize it below.  
	
	\begin{lemma}\label{lem:hdot-interpol}
		Let $0\leq r_0 < r_1 < \infty$, 
		$r\in[r_0,r_1]$, 
		and define $\theta := \tfrac{r - r_0}{r_1 - r_0} \in [0,1]$. 
		Then, 
		\[
			\hdot{r}{A} 
			= 
			\bigl[ \hdot{r_0}{A}, \hdot{r_1}{A} \bigr]_{\theta}
			\quad 
			\text{and}
			\quad  
			\norm{\,\cdot\,}{r,A} 
			= 
			\norm{\,\cdot\,}{ \left[ \hdot{r_0}{A}, \hdot{r_1}{A} \right]_{\theta} }, 
		\]
		i.e., the space $\hdot{r}{A}$ is isometrically isomorphic 
		to the complex interpolation space 
		between $\hdot{r_0}{A}$ and $\hdot{r_1}{A}$ 
		with parameter~$\theta = \tfrac{r - r_0}{r_1 - r_0} \in [0,1]$. 
	\end{lemma} 
	
	\begin{proof}
		The linear operator 
		$\cA := A^{\nicefrac{ (r_1 - r_0) }{2}}$ 
		is self-adjoint and 
		positive definite 
		on the Hilbert space~$\hdot{r_0}{A}$. 
		Since 
		$A^{\nicefrac{ (r_1 - r_0) }{2}} 
		\from \hdot{\bar{r}}{A} \to \hdot{ \bar{r} + r_0 - r_1}{A}$ 
		is an isomorphism for every $\bar{r}\in\bbR$, the domain of 
		$\cA\from\scrD(\cA)\subseteq \hdot{r_0}{A}\to\hdot{r_0}{A}$ 
		is given by $\scrD(\cA) = \hdot{r_1}{A}\subseteq\hdot{r_0}{A}$\!.  
		By noting that the eigenpairs of $\cA$ 
		are 
		$\bigl\{ \bigl( \lambda_j^{ \nicefrac{ (r_1 - r_0) }{2}}, 
		\lambda_j^{-\nicefrac{r_0}{2}} e_j \bigr) \bigr\}_{j\in\bbN} 
		\subset \bbR_+ \times \hdot{r_0}{A}$\!, 
		we find that, for any $\eta\in[0,1]$,  
		\begin{align*}  
			\scrD ( \cA^\eta ) 
			&\textstyle 
			= 
			\Bigl\{
			\psi\in\hdot{ r_0 }{A} 
			: 
			\sum\nolimits_{j\in\bbN} 
			\bigl( \lambda_j^{ r_1 - r_0 } \bigr)^\eta \, 
			\bigl| \bigl( \psi, \lambda_j^{-\nicefrac{r_0}{2}} e_j \bigr)_{r_0,A} \bigr|^2  
			< \infty
			\Bigr\}
			\\
			&\textstyle 
			= 
			\Bigl\{
			\psi\in\hdot{r_0}{A} 
			: 
			\sum\nolimits_{j\in\bbN} 
			\lambda_j^{ \eta r_1 + ( 1 - \eta ) r_0 } \,
			| \scalar{\psi, e_j }{E} |^2  
			< \infty
			\Bigr\}
			=
			\hdot{ \eta r_1 + ( 1 - \eta ) r_0 }{A} 
			\subseteq \hdot{r_0}{A}\!, 
		\end{align*} 
		and 
		$\norm{\,\cdot\,}{2\eta,\cA} 
		= 
		\norm{\,\cdot\,}{\eta r_1 + ( 1 - \eta ) r_0, A}$.  
		Here, we have used Definition 
		\eqref{eq:def:hdotA}  for both~operators 
		$A$ and $\cA$. 
		By \cite[Theorem~16.1]{Yagi2010} we thus find that, 
		for any $\eta\in[0,1]$, 
		\[
			\hdot{\eta r_1 + ( 1 - \eta ) r_0}{A} 
			=
			\scrD( \cA^\eta ) 
			= 
			\bigl[\hdot{r_0}{A}, \scrD(\cA) \bigr]_{\eta} 
			= 
			\bigl[\hdot{r_0}{A}, \hdot{r_1}{A} \bigr]_{\eta}, 
		\]
		with isometry. 
		In particular, the choice 
		$\eta:=\theta=\tfrac{r - r_0}{r_1 - r_0}\in[0,1]$ 
		shows the assertion. 
	\end{proof}

	\begin{lemma}\label{lem:beta-gamma-complex}  
		Let $A\from\scrD(A) \subseteq E \to E$ and 
		$\At\from\scrD(\At)\subseteq E \to E$ 
		be two densely defined, self-adjoint, 
		positive definite linear operators 
		with compact inverses 
		on $E$, and assume that 
		there exists $\beta\in\bbR_+$ 
		such that 
		$A^{-\beta} \At^{2\beta} A^{-\beta} - \id_E \in\cK(E)$ 
		and, 
		for all $\gamma\in[-\beta,\beta]$, 
		the linear operator $\At^\gamma A^{-\gamma}$ 
		is an isomorphism on $E$. 
		Then, for all $\gamma\in[0,\beta]$, also the operator 
		$A^{-\gamma} \At^{2\gamma} A^{-\gamma} 
		- \id_E$ 
		is compact on $E$.  
	\end{lemma} 
	
	\begin{proof} 
		For $\gamma=0$ and $\gamma=\beta$, the claim 
		is trivial. 
		Suppose now that $\gamma\in(0,\beta)$. 
		For $N\in\bbN$, 
		define the subspace 
		$E_N = \operatorname{span}\{e_1, \ldots, e_N\}$ of $E$, 
		generated by the first $N$ eigenvectors of $A$, and 
		let $Q_N$ and $Q_{N}^\perp := \id_{E} - Q_N$ 
		denote the $E$-orthogonal projections onto $E_N$ 
		and onto the $E$-orthogonal complement 
		$E_N^\perp = \operatorname{span}\{e_j\}_{j=N+1}^\infty$, 
		respectively. 
		Compactness 
		of the operator $A^{-\beta} \At^{2\beta} A^{-\beta} 
		- \id_E$ on $E$ implies that, 
		for $\eps\in(0,1)$ fixed, 
		there exists an integer $N_\eps\in\bbN$ such that 
		\[
			\big\| 
			Q_{N_\eps}^\perp 
			\bigl( A^{-\beta} \At^{2\beta} A^{-\beta} - \id_E \bigr) 
			Q_{N_\eps}^\perp 
			\bigr\|_{\cL(E)} 
			\leq \eps.  
		\]
		Since the subspace $E_{N_\eps} \subset E$ is generated by the eigenvectors 
		of $A$,  
		we obtain the relation that 
		${\psi = A^{\beta} v \in E_{N_\eps}^\perp}$ holds 
		if and only if $v \in \hdot{2\beta}{A} \cap E_{N_\eps}^\perp$. 
		Consequently,  
		\begin{align}
			\sup_{v \in \hdot{2\beta}{A} \cap E_{N_\eps}^\perp \setminus\{0\}} 
			\,\Biggl|
			\frac{ \bigl( \At^{\beta} v, \At^{\beta} v \bigr)_{E} }{ 
				\bigl( A^{\beta} v, A^{\beta} v \bigr)_{E} } 
			- 1 		
			\Biggr|
			&= 
			\sup_{\psi \in E_{N_\eps}^\perp \setminus\{0\}} 
			\frac{ \bigl| \bigl( 
				\bigl( A^{-\beta} \At^{2\beta} A^{-\beta} - \id_{E_{N_\eps}^\perp} \bigr)
				\psi, \psi \bigr)_{E} \bigr| }{ \scalar{\psi, \psi}{E} } 
			\notag 
			\\
			&= 
			\sup_{\phi \in E\setminus\{0\}} 
			\frac{ \bigl| \bigl( \bigl( A^{-\beta} \At^{2\beta} A^{-\beta} - \id_E \bigr)
				Q_{N_\eps}^\perp \phi,  Q_{N_\eps}^\perp \phi \bigr)_{E} \bigr| }{ \scalar{\phi, \phi}{E} }
			\notag 
			\\
			&=
			\bigl\|  
			Q_{N_\eps}^\perp 
			\bigl( A^{-\beta} \At^{2\beta} A^{-\beta} - \id_E \bigr)
			Q_{N_\eps}^\perp 
			\bigr\|_{\cL(E)}
			\leq \eps. 
			\label{eq:lem:beta-gamma-proof-1}  
		\end{align} 
		Here, we used self-adjointness of the operator 
		$Q_{N_\eps}^\perp 
		\bigl( A^{-\beta} \At^{2\beta} A^{-\beta} - \id_E \bigr) 
		Q_{N_\eps}^\perp$ 
		which yields the equality with 
		the operator norm in the last step. 
		By combining the observation 
		\eqref{eq:lem:beta-gamma-proof-1}
		with the isomorphism property 
		of $\At^\beta A^{-\beta}$ on~$E$, 
		we conclude that 
		\[
			\sqrt{1-\eps}  
			\,
			\norm{v}{2\beta, A} 
			\leq 
			\norm{v}{2\beta, \At} 
			\leq 
			\sqrt{ 1+\eps }  
			\,
			\norm{v}{2\beta, A} 
			\quad
			\forall v \in \hdot{2\beta}{A}\cap E_{N_\eps}^\perp 
			= \hdot{2\beta}{\At}\cap E_{N_\eps}^\perp. 
		\]
		Furthermore, 
		$\norm{v}{0,A}
		=
		\norm{v}{E} 
		=
		\norm{v}{0,\At}$ 
		holds for all 
		$v \in E_{N_\eps}^\perp = E \cap E_{N_\eps}^\perp$. 
		Therefore, we obtain by interpolation 
		\cite[Theorem~2.6]{Lunardi2018} that, 
		for every $\theta\in(0,1)$ 
		and all $v \in \hdot{2\beta}{A}\cap E_{N_\eps}^\perp 
		= \hdot{2\beta}{\At}\cap E_{N_\eps}^\perp$,  
		\begin{equation}\label{eq:lem:beta-gamma-proof-2} 
			\hspace*{-0.14cm} 
			(1-\eps)^{\nicefrac{\theta}{2}}
			\norm{v}{\left[E_{N_\eps}^\perp, \hdot{2\beta}{A}\cap E_{N_\eps}^\perp\right]_\theta} 
			\leq 
			\norm{v}{\left[E_{N_\eps}^\perp, \hdot{2\beta}{\At}\cap E_{N_\eps}^\perp\right]_\theta}  
			\leq 
			(1+\eps)^{\nicefrac{\theta}{2}}
			\norm{v}{\left[E_{N_\eps}^\perp, \hdot{2\beta}{A}\cap E_{N_\eps}^\perp\right]_\theta}\!.  
		\end{equation}
		By Lemmas~\ref{lem:interpol-subspace} and~\ref{lem:hdot-interpol} 
		$\bigl[E_{N_\eps}^\perp, \hdot{2\beta}{A}\cap E_{N_\eps}^\perp \bigr]_\theta 
		= 
		\bigl[E, \hdot{2\beta}{A} \bigr]_\theta \cap E_{N_\eps}^\perp 
		= 
		\hdot{2\theta\beta}{A} \cap E_{N_\eps}^\perp$, 
		and 
		\[
			\norm{v}{\left[E_{N_\eps}^\perp, \hdot{2\beta}{A}\cap E_{N_\eps}^\perp\right]_\theta} 
			=
			\norm{v}{\left[E, \hdot{2\beta}{A}\right]_\theta} 
			=
			\norm{v}{2\theta\beta, A} 
			\quad 
			\forall 
			v \in 
			\hdot{2\theta\beta}{A} \cap E_{N_\eps}^\perp. 
		\] 
		By the same arguments 
		we also find that 
		\[
			\bigl[E_{N_\eps}^\perp, \hdot{2\beta}{\At}\cap E_{N_\eps}^\perp \bigr]_\theta 
			=
			\hdot{2\theta\beta}{\At} \cap E_{N_\eps}^\perp , 
			\qquad 
			\norm{\,\cdot\,}{
				\bigl[E_{N_\eps}^\perp, \hdot{2\beta}{\At}\cap E_{N_\eps}^\perp\bigr]_\theta} 
			=
			\norm{\,\cdot\,}{2\theta\beta,\At}. 
		\]
		Using these identities in \eqref{eq:lem:beta-gamma-proof-2} yields, 
		for all $\theta\in(0,1)$, that  
		\[ 
			(1-\eps)^{\theta} 
			\norm{v}{2\theta\beta, A}^2 
			\leq 
			\norm{v}{2\theta\beta, \At}^2 
			\leq 
			(1+\eps)^{\theta}
			\norm{v}{2\theta\beta, A}^2  
			\quad 
			\forall 
			v\in \hdot{2\theta\beta}{A} \cap E_{N_\eps}^\perp 
			= \hdot{2\theta\beta}{\At} \cap E_{N_\eps}^\perp. 
		\] 
		By subadditivity 
		of the function $x\mapsto x^\theta$
		for $\theta\in(0,1)$, we have the estimates  
		\[
			1 - \eps^\theta 
			= (1-\eps + \eps)^{\theta}  - \eps^\theta 
			\leq (1-\eps)^\theta  
			\quad 
			\text{and} 
			\quad 
			(1+\eps)^\theta 
			\leq 
			1 + \eps^\theta,  
		\]
		and conclude that, 
		for every $\theta\in(0,1)$, 
		\begin{equation}\label{eq:lem:beta-gamma-proof-3} 
			\bigl(1-\eps^{\theta} \bigr) 
			\norm{v}{2\theta\beta, A}^2 
			\leq 
			\norm{v}{2\theta\beta, \At}^2 
			\leq 
			\bigl( 1+\eps^{\theta} \bigr) 
			\norm{v}{2\theta\beta, A}^2  
			\quad 
			\forall 
			v\in \hdot{2\theta\beta}{A} \cap E_{N_\eps}^\perp 
			= \hdot{2\theta\beta}{\At} \cap E_{N_\eps}^\perp. 
		\end{equation}
		Since $A^{\theta\beta}v \in E_{N_\eps}^\perp$ 
		if and only if $v\in\hdot{2\theta\beta}{A}\cap E_{N_\eps}^\perp$, 
		using \eqref{eq:lem:beta-gamma-proof-3} 
		we find as in \eqref{eq:lem:beta-gamma-proof-1} that 
		\[
			\bigl\|  
			Q_{N_\eps}^\perp 
			\bigl( A^{-\theta\beta} \At^{2\theta\beta} A^{-\theta\beta} - \id_E \bigr)
			Q_{N_\eps}^\perp 
			\bigr\|_{\cL(E)} 
			=
			\sup_{v \in \hdot{2\theta\beta}{A} \cap E_{N_\eps}^\perp \setminus\{0\}} 
			\left|
			\frac{ \bigl( \At^{\theta\beta} v,  \At^{\theta\beta} v \bigr)_{E} }{ 
				\bigl( A^{\theta\beta} v, A^{\theta\beta} v \bigr)_{E} } - 1 		
			\right|
			\leq \eps^\theta. 
		\]
		Finally, the choice $\theta:=\nicefrac{\gamma}{\beta}$ 
		gives 
		$\bigl\| Q_{N_\eps}^\perp 
		\bigl( A^{-\gamma} \At^{2\gamma} A^{-\gamma} - \id_E \bigr)
		Q_{N_\eps}^\perp \bigr\|_{\cL(E)} 
		\leq \eps^{\nicefrac{\gamma}{\beta}}$ 
		showing that $A^{-\gamma} \At^{2\gamma} A^{-\gamma} - \id_E$ 
		is a $\cL(E)$-limit of finite-rank operators 
		and, thus, compact on $E$. 
	\end{proof} 
	
\end{appendix}

\section*{Acknowledgments} 
The authors thank the editor 
and the reviewers for their valuable comments 
which led to an improved, more accessible presentation 
of the results. 

\bibliographystyle{imsart-number} 
\bibliography{bk-matern-bib}

\end{document}

%% file: bk-matern_commands.tex
\usepackage[english]{babel}
\usepackage[latin1]{inputenc} 
\usepackage[T1]{fontenc}

\usepackage{amsfonts,amssymb,amsmath,amsthm}
\usepackage{stmaryrd}
\numberwithin{equation}{section}
\usepackage{dsfont}
\usepackage{multirow}
\usepackage{url}
\usepackage{algorithmic,algorithm}
\usepackage{graphics}
\usepackage{booktabs}
\usepackage{epsfig}
\usepackage{epstopdf}
\usepackage{psfrag}               
\usepackage{mathrsfs}
\usepackage{mathtools}            
\usepackage{subfigure}
\usepackage{longtable}	

\usepackage{paralist}
\usepackage[inline]{enumitem}
\usepackage{color}
\usepackage{nicefrac}
\usepackage{tikz}
\usetikzlibrary{decorations.pathreplacing}

\usepackage[colorlinks=true, citecolor=blue, linkcolor=blue]{hyperref}

\newcommand{\bbC}{\mathbb{C}}

\newcommand{\bbK}{\mathbb{K}}

\newcommand{\bbN}{\mathbb{N}}
\newcommand{\bbP}{\mathbb{P}}

\newcommand{\bbR}{\mathbb{R}}

\newcommand{\cA}{\mathcal{A}}
\newcommand{\cB}{\mathcal{B}}
\newcommand{\cC}{\mathcal{C}}
\newcommand{\cD}{\mathcal{D}}
\newcommand{\cE}{\mathcal{E}}
\newcommand{\cF}{\mathcal{F}}

\newcommand{\cH}{\mathcal{H}}

\newcommand{\cK}{\mathcal{K}}
\newcommand{\cL}{\mathcal{L}}
\newcommand{\cM}{\mathcal{M}}

\newcommand{\cP}{\mathcal{P}}

\newcommand{\cS}{\mathcal{S}}

\newcommand{\cU}{\mathcal{U}}
\newcommand{\cV}{\mathcal{V}}
\newcommand{\cW}{\mathcal{W}}
\newcommand{\cX}{\mathcal{X}}

\newcommand{\cZ}{\mathcal{Z}}

\newcommand{\proper}{\mathsf}

\newcommand{\pE}{\proper{E}}

\newcommand{\pN}{\proper{N}}

\newcommand{\scrD}{\mathscr{D}}
\newcommand{\frakE}{\mathfrak{E}} 
\newcommand{\frakN}{\mathfrak{N}}

\newcommand{\norm}[2]{     \| #1       \|_{ #2 }}
\newcommand{\seminorm}[2]{  | #1        |_{ #2 }}

\newcommand{\normiii}[2]{\vert\kern-0.25ex\vert\kern-0.25ex\vert #1 \vert\kern-0.25ex \vert\kern-0.25ex\vert_{ #2 }}
\newcommand{\Normiii}[2]{\left\vert\kern-0.25ex\left\vert\kern-0.25ex\left\vert #1 \right\vert\kern-0.25ex\right\vert\kern-0.25ex\right\vert_{ #2 }}
\newcommand{\scalar}[2]{     ( #1       )_{ #2 }}

\newcommand{\duality}[2]{     \langle #1       \rangle_{ #2 }}

\newcommand{\clos}[1]{\overline{#1}}
\newcommand{\dual}[1]{{#1}^{*}}
\newcommand{\mult}[1]{M_{#1}}

\newcommand{\normal}{\pN}
\newcommand{\white}{\cW}
\newcommand{\id}{\operatorname{Id}}

\newcommand{\rd}{\mathrm{d}}
\newcommand{\from}{\colon}

\newcommand{\mv}[1]{{\boldsymbol{\mathrm{#1}}}}
\newcommand{\trsp}{\ensuremath{\top}}
\newcommand{\s}{s}

\newcommand{\eps}{\varepsilon}
\newcommand{\GP}{Z}
\newcommand{\gp}{z}
\newcommand{\CC}{\widetilde{\cC}}
\newcommand{\ac}{\boldsymbol{a}} 
\newcommand{\act}{\widetilde{\ac}}
\newcommand{\At}{\widetilde{A}}
\newcommand{\cAt}{\widetilde{\cA}} 
\newcommand{\Lt}{\widetilde{L}}

\newcommand{\mut}{\widetilde{\mu}}
\newcommand{\mt}{\widetilde{m}}
\newcommand{\betat}{\widetilde{\beta}}
\newcommand{\gammat}{\widetilde{\gamma}}
\newcommand{\kappat}{\widetilde{\kappa}}
\newcommand{\lambdat}{\widetilde{\lambda}}
\newcommand{\sigmat}{\widetilde{\sigma}}
\newcommand{\taut}{\widetilde{\tau}}

\newcommand{\onormal}{\mathbf{n}} 

\newcommand{\erf}{\operatorname{erf}}
\newcommand{\obsmat}{\Phi}
\newcommand{\bobsmat}{\mv{\obsmat}}

\newcommand{\Hdot}[1]{\dot{H}^{#1}_L}
\newcommand{\hdot}[2]{\dot{H}^{#1}_{#2}}

\theoremstyle{plain} 
\newtheorem{lemma}{Lemma}[section] 
\newtheorem{proposition}[lemma]{Proposition}
\newtheorem{theorem}[lemma]{Theorem}
\newtheorem{corollary}[lemma]{Corollary} 

\theoremstyle{remark}
\newtheorem{remark}[lemma]{Remark} 

\newtheorem{assumption}[lemma]{Assumption} 
\newtheorem{definition}[lemma]{Definition}

